\documentclass[twoside]{article}
\usepackage{amsmath, amsfonts, amsthm, amssymb}
\usepackage{graphicx}
\usepackage{hyperref}
\hypersetup{
	colorlinks=true,
	linkcolor=blue,
	citecolor=blue
}
\usepackage[parfill]{parskip}
\usepackage{algpseudocode}
\usepackage{algorithm}
\usepackage{enumerate}
\usepackage[shortlabels]{enumitem}
\usepackage{mathtools}
\usepackage{tikz}
\usepackage{verbatim}
\usepackage{fullpage}
\usepackage{microtype}
\usepackage{bm}

\usepackage{natbib}

\DeclareFontFamily{U}{mathx}{\hyphenchar\font45}
\DeclareFontShape{U}{mathx}{m}{n}{<-> mathx10}{}
\DeclareSymbolFont{mathx}{U}{mathx}{m}{n}
\DeclareMathAccent{\wb}{0}{mathx}{"73}

\DeclarePairedDelimiterX{\norm}[1]{\lVert}{\rVert}{#1}
\DeclarePairedDelimiterX{\seminorm}[1]{\lvert}{\rvert}{#1}

\DeclareFontFamily{U}{mathx}{\hyphenchar\font45}
\DeclareFontShape{U}{mathx}{m}{n}{
	<5> <6> <7> <8> <9> <10>
	<10.95> <12> <14.4> <17.28> <20.74> <24.88>
	mathx10
}{}
\DeclareSymbolFont{mathx}{U}{mathx}{m}{n}
\DeclareFontSubstitution{U}{mathx}{m}{n}
\DeclareMathAccent{\widecheck}{0}{mathx}{"71}

\newcommand{\Reals}{\mathbb{R}}

\newcommand{\abs}[1]{\left \lvert #1 \right \rvert}

\newcommand{\set}[1]{\left\{#1\right\}}

\newcommand{\floor}[1]{\left\lfloor #1 \right\rfloor}
\newcommand{\Var}{\mathrm{Var}}
\newcommand{\Cov}{\mathrm{Cov}}
\newcommand{\diam}{\mathrm{diam}}

\newcommand{\vol}{\text{vol}}

\newcommand{\1}{\mathbf{1}}

\DeclareMathOperator*{\argmin}{argmin}

\DeclareMathOperator*{\esssup}{ess\,sup}

\newcommand{\Rd}{\Reals^d}

\newcommand{\lambdavec}{\boldsymbol{\lambda}}

\newcommand{\Lap}{L}

\newcommand{\Id}{I}

\newcommand{\Xset}{\mathcal{X}}

\newcommand{\Leb}{L}
\newcommand{\mc}[1]{\mathcal{#1}}

\newcommand{\Pbb}{\mathbb{P}}
\newcommand{\Ebb}{\mathbb{E}}

\newcommand{\dive}{\mathrm{div}}

\newcommand{\dx}{\,dx}

\newcommand{\wt}[1]{\widetilde{#1}}
\newcommand{\wh}[1]{\widehat{#1}}
\newcommand{\wc}[1]{\widecheck{#1}}

\newtheorem{theorem}{Theorem}

\newtheorem{lemma}{Lemma}

\newtheorem{proposition}{Proposition}

\theoremstyle{definition}

\theoremstyle{remark}

\begin{document}

\begin{center} {\Large{\bf{Minimax Optimal Regression over Sobolev Spaces \\
\vspace{.2cm}
 via Laplacian Regularization on Neighborhood Graphs}}}
	
	\vspace*{.3cm}
	
	{\large{
			\begin{center}
				Alden Green~~~~~ Sivaraman Balakrishnan~~~~~ Ryan J. Tibshirani\\
				\vspace{.2cm}
			\end{center}

			\begin{tabular}{c}
				Department of Statistics and Data Science \\
				Carnegie Mellon University
			\end{tabular}
			
			\vspace*{.2in}
			
			\begin{tabular}{c}
				\texttt{\{ajgreen,siva,ryantibs\}@stat.cmu.edu}
			\end{tabular}
	}}
	
	\vspace*{.2in}
	
	\today
	\vspace*{.2in}

\end{center}

\begin{abstract}
	In this paper we study the statistical properties of Laplacian smoothing, a graph-based approach to nonparametric regression. Under standard regularity conditions, we establish upper bounds on the error of the Laplacian smoothing estimator \smash{$\wh{f}$}, and a goodness-of-fit test also based on \smash{$\wh{f}$}. These upper bounds match the minimax optimal estimation and testing rates of convergence over the first-order Sobolev class $H^1(\Xset)$, for $\Xset \subseteq \Rd$ and $1 \leq d < 4$; in the estimation problem, for $d = 4$, they are optimal modulo a $\log n$ factor. Additionally, we prove that Laplacian smoothing is manifold-adaptive: if $\Xset \subseteq \Reals^d$ is an $m$-dimensional manifold with $m < d$, then the error rate of Laplacian smoothing (in either estimation or testing) depends only on $m$, in the same way it would if $\Xset$ were a full-dimensional set in $\Reals^m$.
\end{abstract}

\section{Introduction}

We adopt the standard nonparametric regression setup, where we observe samples $(X_1,Y_1),\ldots,(X_n,Y_n)$ that are i.i.d.\ draws from the model 
\begin{equation}
\label{eqn:signal_plus_noise_model}
Y_i = f_0(X_i) + \varepsilon_i, \quad \varepsilon_i \sim N(0,1), 
\end{equation}
where $\varepsilon_i$ is independent of $X_i$. Our goal is to perform statistical inference on the unknown regression function $f_0$, by which we mean either \emph{estimating} $f_0$ or \emph{testing} whether $f_0 = 0$, i.e., whether there is any signal present. 

Laplacian smoothing \citep{smola2003} is a penalized least squares estimator, defined over a graph. Letting $G = (V,W)$ be a weighted undirected graph with vertices $V=\{1,\ldots,n\}$, associated with $\{X_1,\ldots,X_n\}$, and $W \in \mathbb{R}^{n \times n}$ is the (weighted) adjacency matrix of the graph. 
the Laplacian smoothing estimator \smash{$\wh{f}$} is given by
\begin{equation}
\label{eqn:laplacian_smoothing}
\wh{f} =  \argmin_{f \in \Reals^n} \; \sum_{i = 1}^{n}(Y_i - f_i)^2 + \rho \cdot f^\top \Lap f. 
\end{equation}
Here $\Lap$ is the graph Laplacian matrix (defined formally in Section~\ref{sec:problem_setup_and_background}), $G$ is typically a geometric graph (such as a $k$-nearest-neighbor or neighborhood graph), $\rho \geq 0$ is a tuning parameter, and the penalty
\begin{equation*}
f^\top \Lap f = \frac{1}{2} \sum_{i,j = 1}^{n} W_{ij}(f_i - f_j)^2
\end{equation*}
encourages \smash{$\wh{f}_i \approx \wh{f}_j$} when $X_i \approx X_j$. Assuming \eqref{eqn:laplacian_smoothing} is a reasonable estimator of $f_0$, the statistic
\begin{equation}
\label{eqn:laplacian_smoothing_test}
\wh{T} = \frac{1}{n} \| \wh{f} \|_2^2 
\end{equation}
is in turn a natural test statistic to test if $f_0 = 0$. 

Of course there are many methods for nonparametric regression (see, e.g., \citet{gyorfi2006,wasserman2006,tsybakov2008_book}), but Laplacian smoothing has its own set of advantages. For instance:
\begin{itemize}
	\item \emph{Computational ease.} Laplacian smoothing is fast, easy, and stable to compute. The estimate \smash{$\wh{f}$} can be computed by solving a symmetric diagonally dominant linear system. There are by now various nearly-linear-time solvers for this problem (see e.g., the seminal papers of \citet{spielman2011,spielman2013,spielman2014}, or the overview by \citet{vishnoi2012} and references therein).
	\item \emph{Generality.} Laplacian smoothing is well-defined whenever one can associate a graph with observed responses. This generality lends itself to many different data modalities, e.g., text and image classification, as in \citet{kondor2002,belkin03a,belkin2006}.
	\item \emph{Weak supervision.} Although we study Laplacian smoothing in the supervised problem setting \eqref{eqn:signal_plus_noise_model}, the method can be adapted to the semi-supervised or unsupervised settings, as in \citet{zhu2003semisupervised,zhou2005learning,nadler09}. 
\end{itemize}

For these reasons, a body of work has emerged that analyzes the statistical properties of Laplacian smoothing, and graph-based methods more generally. Roughly speaking, this work can be divided into two categories, based on the perspective they adopt. 
\begin{itemize}
	\item \emph{Fixed design perspective.} Here one treats the design points $X_1,\ldots,X_n$ and the graph $G$ as fixed, and carries out inference on $f_0(X_i)$, $i=1,\ldots,n$. In this problem setting, tight upper bounds have been derived on the error of various graph-based methods (e.g., \citet{wang2016,hutter2016,sadhanala16,sadhanala17,kirichenko2017,kirichenko2018}) and tests (e.g., \citet{sharpnack2010identifying,sharpnack2013b,sharpnack2013,sharpnack2015}), which certify that such procedures are \emph{optimal} over ``function'' classes (in quotes because these classes really model the $n$-dimensional vector of evaluations). The upside of this work is its generality: in this setting $G$ need not be a geometric graph, but in principle it could be any graph over $V=\{1,\ldots,n\}$. The downside is that, in the context of nonparametric regression, it is arguably not as natural to think of the evaluations of $f_0$ as exhibiting smoothness over some fixed pre-defined graph $G$, and more natural to speak of the smoothness of the function $f_0$ itself. 
	\item \emph{Random design perspective.} Here one treats the design points $X_1,\ldots,X_n$ as independent samples from some distribution $P$ supported on a domain $\Xset \subseteq \Rd$. Inference is drawn on the regression function $f_0: \Xset \to \Reals$, which is typically assumed to be smooth in some \emph{continuum} sense, e.g., it possesses a first derivative bounded in $L^{\infty}$ (H\"{o}lder) or $\Leb^2$ (Sobolev) norm. To conduct graph-based inference, the user first builds a neighborhood graph over the random design points---so that $W_{ij}$ is large when $X_i$ and $X_j$ are close in (say) Euclidean distance---and then computes e.g., \eqref{eqn:laplacian_smoothing} or \eqref{eqn:laplacian_smoothing_test}. In this context, various graph-based procedures have been shown to be \emph{consistent}: as $n \to \infty$, they converge to a continuum limit (see \citet{belkin07,vonluxburg2008,trillos2018} among others). However, until recently such statements were not accompanied by error rates, and even so, such error rates as have been proved \citep{lee2016,trillos2020} are not optimal over continuum function spaces, such as H\"{o}lder or Sobolev classes. 
\end{itemize}
The random design perspective bears a more natural connection with nonparametric regression (the focus in this paper), as it allows us to formulate smoothness based on $f_0$ itself (how it behaves as a continuum function, and not just its evaluations at the design points). In this paper, we will adopt the random design perspective, and seek to answer the following question:
\begin{quote}{When we assume the regression function $f_0$ is smooth in a continuum sense, does Laplacian smoothing achieve optimal performance for estimation and goodness-of-fit testing?} 
\end{quote}
This is no small question---arguably, it is \emph{the} central question of nonparametric regression---and without an answer one cannot fully compare the statistical properties of Laplacian smoothing to alternative methods. It also seems difficult to answer: as we discuss next, there is a fundamental gap between the \emph{discrete} smoothness imposed by the penalty $f^\top \Lap f$ in problem \eqref{eqn:laplacian_smoothing} and the \emph{continuum} smoothness assumed on $f_0$, and in order to obtain sharp upper bounds we will need to bridge this gap in a suitable sense.

\section{Summary of Results}
\label{sec:summary}

\paragraph{Advantages of the Discrete Approach.} 

In light of the potential difficulty in bridging the gap between discrete and continuum notions of smoothness, it is worth asking whether there is any \emph{statistical} advantage to solving a discrete problem such as \eqref{eqn:laplacian_smoothing} (setting aside computational considerations for the moment). After all, we could have instead solved the following variational problem:
\begin{equation}
\label{eqn:thin_plate_spline}
\wt{f} = \argmin_{f : \Xset \to \Reals} \; \sum_{i = 1}^{n} \bigl(Y_i - f(X_i)\bigr)^2 + \rho \hspace{-2pt} \int_{\Xset} \|\nabla f(x)\|_2^2 \,dx,
\end{equation}
where the optimization is performed over all continuous functions $f$ that have a weak derivative $\nabla f$ in $L^2(\Xset)$. Analogously, for testing, we could use:
\begin{equation}
\label{eqn:thin_plate_spline_test}
\wt{T} = \| \wt{f} \|_n^2 := \frac{1}{n} \sum_{i=1}^n \wt{f}(X_i)^2.
\end{equation}
The penalty term in \eqref{eqn:thin_plate_spline} leverages the assumption that $f_0$ has a smooth derivative in a seemingly natural way. Indeed, the estimator \smash{$\wt{f}$} and statistic \smash{$\wt{T}$} are well-known: for $d = 1$, \smash{$\wt{f}$} is the familiar \emph{smoothing spline}, and for $d>1$, it is a type of \emph{thin-plate spline}. The statistical properties of smoothing and thin-plate splines are well-understood \citep{vandergeer2000,liu2019}. As we discuss later, the Laplacian smoothing problem \eqref{eqn:laplacian_smoothing} can be viewed as a discrete and noisy approximation to \eqref{eqn:thin_plate_spline}. At first blush, this suggests that Laplacian smoothing should at best inherit the statistical properties of \eqref{eqn:thin_plate_spline}, and at worst may have meaningfully larger error.

However, as we shall see the actual story is quite different: remarkably, Laplacian smoothing enjoys optimality properties even in settings where the thin-plate spline estimator \eqref{eqn:thin_plate_spline} is not well-posed (to be explained shortly); Tables~\ref{tbl:estimation_rates} and~\ref{tbl:testing_rates} summarize. As we establish in Theorems~\ref{thm:laplacian_smoothing_estimation1}-\ref{thm:laplacian_smoothing_testing_manifold}, when computed over an appropriately formed neighborhood graph, Laplacian smoothing estimators and tests are minimax optimal over first-order \emph{continuum} Sobolev balls. This holds true either when $\Xset \subseteq \Rd$ is a full-dimensional domain and $d = 1,2,$ or $3$, or when $\Xset$ is a manifold embedded in $\Rd$ of intrinsic dimension $m = 1,2,$ or $3$. Additionally, the estimator \smash{$\wh{f}$} is nearly minimax optimal (to within a \smash{$(\log n)^{1/3}$} factor) when $d=4$ (or $m=4$ in the manifold case). 

By contrast, smoothing splines are optimal only when $d=1$. When $d>1$, the thin-plate spline estimator \eqref{eqn:thin_plate_spline} is not even well-posed, in the following sense: for any $(X_1,Y_1),\ldots,(X_n,Y_n)$ and any $\delta>0$, there exists (e.g., \citet{green93} give a construction using ``bump'' functions) a differentiable function $f$ such that $f(X_i) = Y_i$, $i=1,\ldots,n$, and
\begin{equation*}
\int_{\Xset} \|\nabla f(x)\|_2^2 \leq \delta.
\end{equation*}
In other words, $f$ achieves perfect (zero) data loss and arbitrarily small penalty in the problem \eqref{eqn:thin_plate_spline}. This will clearly not lead to a consistent estimator of $f_0$ across the design points (as it always yields $Y_i$ at each $X_i$). In this light, our results when $d>1$ favorably distinguish Laplacian smoothing from its natural variational analog.

\begin{table}
	\begin{center}
		\begin{tabular}{p{.1\textwidth} | p{.14\textwidth} p{.12\textwidth} }
			Dimension & Laplacian smoothing \eqref{eqn:laplacian_smoothing} & Thin-plate splines \eqref{eqn:thin_plate_spline} \\
			\hline
			$d = 1$ & $\bm{n^{-2/3}}$ & \textcolor{red}{$\bm{n^{-2/3}}$} \\
			$d = 2,3$ & $\bm{n^{-2/(2 + d)}}$ & \textcolor{red}{$1$} \\
			$d = 4$ & $\bm{n^{-1/3}} (\log n)^{1/3}$ & \textcolor{red}{$1$} \\
			$d \geq 5$  & $(\log n/n)^{4/(3d)}$ &\textcolor{red}{$1$} \\
		\end{tabular}
	\end{center}
	\caption{Summary of estimation rates over first-order Sobolev balls. Black font marks new results from this paper, red font marks previously-known results; bold font marks minimax optimal rates. Although we suppress it for simplicity, in all cases the dependence of the error rate on the radius of the Sobolev ball is also optimal. The rates for thin-plate splines with $d \geq 2$ assume the estimator \smash{$\wt{f}$} interpolates the responses, \smash{$\wt{f}(X_i) = Y_i$} for $i = 1,\ldots,n$; see the discussion in Section~\ref{sec:summary}. Here, we use ``$1$'' to indicate inconsistency (error not converging to 0). Lastly, when $\Xset$ is an $m$-dimensional manifold embedded in $\Rd$, all Laplacian smoothing results hold with $d$ replaced by $m$, without any change to the method itself.}
	\label{tbl:estimation_rates}
\end{table}

\begin{table}
	\begin{center}
		\begin{tabular}{p{.1\textwidth} | p{.14\textwidth} p{.12\textwidth} }
			Dimension & Laplacian smoothing \eqref{eqn:laplacian_smoothing_test} & Thin-plate splines \eqref{eqn:thin_plate_spline_test} \\
			\hline
			$d = 1$ & $\bm{n^{-4/5}}$ & \textcolor{red}{$\bm{n^{-4/5}}$} \\
			$d = 2,3$ & $\bm{n^{-4/(4 + d)}}$ & \textcolor{red}{$n^{-1/2}$} \\
			$d \geq 4$ & $\bm{n^{-1/2}}$ & \textcolor{red}{$\bm{n^{-1/2}}$}
		\end{tabular}
	\end{center}
	\caption{Summary of testing rates over first-order Sobolev balls; black, red, and bold fonts are used as in Table~\ref{tbl:estimation_rates}. The rates for thin-plate splines with $d \geq 2$ assume the test statistic \smash{$\wt{T}$} is computed using an \smash{$\wt{f}$} that interpolates the responses, $\wt{f}(X_i) = Y_i$ for $i = 1,\ldots,n$.  Rates for $d \geq 4$ assume that $f_0 \in \Leb^4(\Xset,M)$. Lastly, when $\Xset$ is an $m$-dimensional manifold embedded in $\Rd$, all rates hold with $d$ replaced by $m$.}
	\label{tbl:testing_rates}
\end{table}

\paragraph{Future Directions.} 

To be clear, there is still much left to be investigated. For one, the Laplacian smoothing estimator \smash{$\wh{f}$} is only defined at $X_1,\ldots,X_n$. In this work we study its in-sample mean squared error 
\begin{equation}
\label{eqn:empirical_norm_error}
\bigl\|\wh{f} - f_0\bigr\|_n^2 := \frac{1}{n}\sum_{i = 1}^{n} \Bigl(\wh{f}_i - f_0(X_i)\Bigr)^2.
\end{equation}
In Section~\ref{sec:minimax_optimal_laplacian_smoothing}, we discuss how to extend \smash{$\wh{f}$} to a function over all $\Xset$, in such a way that the out-of-sample mean squared error \smash{$\|\wh{f} - f_0\|_{\Leb^2(\Xset)}^2$} should remain small, but leave a formal analysis to future work.

In a different direction, problem \eqref{eqn:thin_plate_spline} is only a special, first-order case of thin-plate splines. In general, the $k$th order thin-plate spline estimator is defined as
\begin{equation*}
\wt{f} = \argmin_{f : \Xset \to \Rd} \; \sum_{i = 1}^{n} \bigl(Y_i - f(X_i)\bigr)^2 + \rho \hspace{-2pt} \sum_{|\alpha| = k} \int_{\Xset} \bigl(D^\alpha f(x)\bigr)^2 \,dx,
\end{equation*}
where for each multi-index $\alpha=(\alpha_1,\ldots,\alpha_d)$ we write $D^\alpha f(x) = \partial^kf/\partial x_{1}^{\alpha_1} \cdots \partial x_{d}^{\alpha_d}$. This problem is in general well-posed whenever $2k > d$. In this regime, assuming that the $k$th order partial derivatives $D^\alpha f_0$ are all $\Leb^2(\Xset)$ bounded, the degree $k$ thin-plate spline has error on the order of \smash{$n^{-2k/(2k + d)}$} \citep{vandergeer2000}, which is minimax rate-optimal for such functions. Of course, assuming $f_0$ has $k$ bounded derivatives for some $2k > d$ is a very strong condition, but at present we do not know if (adaptations of) Laplacian smoothing on neighborhood graphs achieve these rates.

\paragraph{Notation.}

For an integer $p \geq 1$, we use $\Leb^p(\Xset)$ for the set of functions $f$ such that 
\begin{equation*}
\norm{f}_{\Leb^p(\Xset)}^p := \int_{\Xset} |f(x)|^p \,dx < \infty,
\end{equation*}
and $C^p(\Xset)$ for the set of functions that are $p$ times continuously differentiable. For sequences $a_n,b_n$, we write $a_n \lesssim b_n$ to mean $a_n \leq Cb_n$ for a constant $C>0$ and large enough $n$, and $a_n \asymp b_n$ to mean $a_n \lesssim b_n$ and $b_n \lesssim a_n$. Lastly, we use $a \wedge b = \min\{a,b\}$.

\section{Background}
\label{sec:problem_setup_and_background}

Before we present our main results in Section~\ref{sec:minimax_optimal_laplacian_smoothing}, we define neighborhood graph Laplacians, and review known minimax rates over first-order Sobolev spaces.

\paragraph{Neighborhood Graph Laplacians.}

In the graph-based approach to nonparametric regression, we first build a neighborhood graph \smash{$G_{n,r} = (V,W)$}, for $V=\{1,\ldots,n\}$, to capture the geometry of $P$ (the design distribution) and $\Xset$ (the domain) in a suitable sense. The $n \times n$ weight matrix $W = (W_{ij})$ encodes proximity between pairs of design points; for a kernel function $K: [0,\infty) \to \Reals$ and radius $r > 0$, we have
\begin{equation*}
\label{eqn:neighborhood_graph}
W_{ij} = K\Biggl(\frac{\|X_i - X_j\|_2}{r}\Biggr),
\end{equation*}
with $\|\cdot\|_2$ denoting the $\ell_2$ norm on $\Rd$. Defining $D$ as the $n \times n$ diagonal matrix with entries \smash{$D_{ii} = \sum_{j = 1}^{n} W_{ij}$}, the graph Laplacian can then be written as
\begin{equation}
\label{eqn:graph_Laplacian}
\Lap = D - W.
\end{equation}
We use \smash{$L = \sum_{k = 1}^{n} \lambda_k v_k v_k^\top$} for an eigendecomposition of $L$, and we always assume, by convention, ordered eigenvalues $0 = \lambda_1 \leq \cdots \leq \lambda_n$, and unit-norm eigenvectors.

\paragraph{Sobolev Spaces.}

We step away from graph-based methods for a moment, to briefly recall some classical results regarding minimax rates over Sobolev classes. We say that a function $f \in \Leb^2(\Xset)$ belongs to the \emph{first-order Sobolev space} $H^1(\Xset)$ if, for each $j = 1,\ldots,d$, the weak partial derivative $D^j f$ exists and belongs to $\Leb^2(\Xset)$. For such functions $f \in H^1(\Xset)$, the Sobolev seminorm \smash{$\seminorm{f}_{H^{1}(\Xset)}$} is the average size of the gradient $\nabla f = (D^1 f, \ldots, D^d f)$, 
\begin{equation*}
\seminorm{f}_{H^1(\Xset)}^2 := \int_{\Xset} \bigl\|\nabla f(x)\bigr\|_2^2 \,dx,
\end{equation*}
with corresponding Sobolev norm 
\begin{equation*}
\norm{f}_{H^1(\Xset)} := \norm{f}_{\Leb^2(\Xset)} + |f|_{H^1(\Xset)}.
\end{equation*}
The Sobolev ball $H^1(\Xset,M)$ for $M > 0$ is
\begin{equation*}
H^1(\Xset,M) := \Bigl\{f \in H^1(\Xset): \norm{f}_{H^1(\Xset)}^2 \leq M^2\Bigr\}.
\end{equation*}
For further details regarding Sobolev spaces see, e.g., \citet{evans10,leoni2017}.

\paragraph{Minimax Rates.}

To carry out a minimax analysis of regression in Sobolev spaces, one must impose regularity conditions on the design distribution $P$. We shall assume the following.
\begin{enumerate}[label=(P\arabic*)]
	\item
	\label{asmp:domain}
	$P$ is supported on a domain $\Xset \subseteq \Rd$, which is an open, connected set with Lipschitz boundary.
	\item
	\label{asmp:density} 
	$P$ admits a density $p$ such that
	\begin{equation*}
	0 < p_{\min} \leq p(x) \leq p_{\max} < \infty, ~~\textrm{for all $x \in \Xset$}.
	\end{equation*}
	Additionally, $p$ is Lipschitz on $\Xset$, with Lipschitz constant $L_p$.
\end{enumerate}

Under conditions~\ref{asmp:domain}, \ref{asmp:density}, the minimax estimation rate over a Sobolev ball of radius $M \geq n^{-1/2}$ is (e.g., \citet{tsybakov2008_book}):
\begin{equation}
\label{eqn:sobolev_space_estimation_minimax_rate}
\inf_{\wh{f}} \sup_{f_0 \in H^1(\Xset, M)} \Ebb\Bigl[\norm{\wh{f} - f_0}_{L^2(\Xset)}^2\Bigr] \asymp M^{2d/(2 + d)} n^{-2/(2 + d)}.
\end{equation}
(Throughout we assume $M \geq n^{-1/2}$, as otherwise the trivial estimator \smash{$\wh{f} = 0$} achieves smaller error than the parametric rate $n^{-1}$, and the problem does not fit well within the nonparametric setup.) 

As minimax rates in nonparametric hypothesis testing are (comparatively) less familiar than those in nonparametric estimation, we briefly summarize the main idea before stating the optimal error rate. In the goodness-of-fit testing problem, we ask for a test function---formally, a Borel measurable function $\phi$ taking values in $\{0,1\}$---which can distinguish between the hypotheses
\begin{equation}
\mathbf{H}_0: f_0 = f_0^{\star}, ~~\textrm{versus}~~ \mathbf{H}_a: f_0 \in \mc{F} \setminus \{f_0^{\star}\}.
\end{equation} 
Typically, the null hypothesis $f_0 = f_0^{\star} \in \mc{F}$ reflects the absence of interesting structure, and $\mc{F} \setminus  \{f_0^{\star}\}$ is a set of smooth departures from this null. In this paper, as in \citet{ingster2009}, we focus on the problem of \emph{signal detection} in Sobolev spaces, where $f_0^{\star} = 0$ and $\mc{F} = H^1(\Xset,M)$ is a first-order Sobolev ball. This is without loss of generality since our test statistic and its analysis are easily modified to handle the case when $f_0^{\star}$ is not $0$, by simply subtracting $f_0^{\star}(X_i)$ from each observation $Y_i$.



The Type I error of a test $\phi$ is $\mathbb{E}_0[\phi]$, and if $\mathbb{E}_0[\phi] \leq \alpha$ for a given $\alpha \in (0,1)$ we refer to $\phi$ as a level-$\alpha$ test. The worst-case risk of $\phi$ over $\mc{F}$ is
\begin{equation*}
R_n(\phi, \mc{F}, \epsilon) := \sup\Bigl\{\mathbb{E}_{f_0}[1 - \phi]: f_0 \in \mc{F}, \|f_0\|_{\Leb^2(\mc{X})} > \epsilon\Bigr\},
\end{equation*}
and for a given constant $b \geq 1$, the minimax critical radius $\epsilon(\mc{F})$ is the smallest value of $\epsilon$ such that some level-$\alpha$ test has worst-case risk of at most $1/b$. Formally,
\begin{equation*}
\epsilon(\mc{F}) := \inf\Bigl\{\epsilon > 0: \inf_{\phi} R_n(\phi,\mc{F},\epsilon) \leq 1/b \Bigr\},
\end{equation*} 
where in the above the infimum is over all level-$\alpha$ tests $\phi$, and $\Ebb_{f_0}[\cdot]$ is the expectation operator under the regression function $f_0$.\footnote{Clearly, the minimax critical radius $\epsilon$ depends on $\alpha$ and $b$. However, we adopt the typical convention of treating $\alpha \in (0,1)$ and $b \geq 1$ as small but fixed positive constants; hence they will not affect the testing error rates, and we suppress them notationally.} 

The classical approach to hypothesis testing typically focuses on designing test statistics, and studying their (limiting) distribution in order to ensure control of the Type I error. In many cases the Type II error (or risk in our terminology) is not emphasized, or the risk of the test against fixed or directional alternatives (i.e. alternatives which deviate from the null in a fixed direction) is studied. In contrast, in the minimax paradigm the (uniform or worst-case) risk against a large collection of alternatives is the central focus. See \citet{ingster82,ingster87,ingster2012,ariascastro2018,balakrishnan2019,balakrishnan2018hypothesis} for a more extended treatment of the minimax paradigm in nonparametric testing, and for a discussion of its advantages (and disadvantages) over other approaches to studying hypothesis tests.

Testing $f_0 = 0$ is an easier problem than estimating $f_0$, and hence the minimax testing critical radius over $H^1(\Xset,M)$ is smaller than the minimax estimation rate, for $1 \leq d < 4$ (see \citet{ingster2009}):
\begin{equation}
\label{eqn:sobolev_space_testing_critical_radius}
\epsilon^2\bigl(H^1(\Xset,M)\bigr) \asymp M^{2d/(4 + d)}n^{-4/(4 + d)}.
\end{equation}
When $d \geq 4$ the functions in $H^1(\Xset)$ are very irregular; formally speaking $H^1(\Xset)$ does not continuously embed into $\Leb^4(\Xset)$ when $d \geq 4$, and the minimax testing rates in this regime are unknown. 

\section{Minimax Optimality of Laplacian Smoothing}
\label{sec:minimax_optimal_laplacian_smoothing}

We now formalize the main conclusions of this paper: that Laplacian smoothing methods on neighborhood graphs are minimax rate-optimal over first-order continuum Sobolev classes. We will assume \ref{asmp:domain}, \ref{asmp:density} on $P$, and the following condition on the kernel $K$.
\begin{enumerate}[label=(K\arabic*)]
	\item
	\label{asmp:kernel}
	$K:[0,\infty) \to [0,\infty)$ is a nonincreasing function supported on $[0,1]$, its restriction to $[0,1]$ is Lipschitz, and $K(1) > 0$. Additionally, it is normalized so that
	\begin{equation*}
	\int_{\Rd} K(\|z\|_2) \,dz = 1.
	\end{equation*}
	We assume \smash{$\sigma_K = \frac{1}{d} \int_{\Rd} \|x\|_2^2 K(\|x\|_2) \,dx < \infty$}.
\end{enumerate}
This is a mild condition: recall the choice of kernel is under the control of the user, and moreover \ref{asmp:kernel} covers many common kernel choices.

\paragraph{Estimation Error of Laplacian Smoothing.} 

Under these conditions, the Laplacian smoothing estimator \smash{$\wh{f}$} achieves an error rate that matches the minimax lower bound over $H^1(\Xset,M)$. This statement will hold whenever the graph $G_{n,r}$ is computed with radius $r$ in the following range.
\begin{enumerate}[label=(R\arabic*)]
	\setcounter{enumi}{0}
	\item 
	\label{asmp:ls_kernel_radius_estimation}
	For constants $C_0,c_0>0$, the neighborhood graph radius $r$ satisfies
	\begin{equation*}
	C_0 \biggl(\frac{\log n}{n}\biggr)^{\frac{1}{d}} \leq r \leq c_0 \wedge M^{\frac{d - 4}{4 + 2d}} n^{-\frac{3}{4 + 2d}}.  
	\end{equation*}
\end{enumerate}

Next we state Theorem~\ref{thm:laplacian_smoothing_estimation1}, our main estimation result. Its proof, as with all proofs of results in this paper, can be found in the appendix. 
\begin{theorem}
	\label{thm:laplacian_smoothing_estimation1}
	Given i.i.d.\ draws $(X_i,Y_i)$, $i=1,\ldots,n$ from \eqref{eqn:signal_plus_noise_model}, assume $f_0 \in H^1(\Xset,M)$ where $\Xset \subseteq \Rd$ has dimension $d < 4$ and $M \leq n^{1/d}$. Assume \ref{asmp:domain}, \ref{asmp:density} on the design distribution $P$, and assume the neighborhood graph \smash{$G_{n,r}$} is computed with a kernel $K$ satisfying \ref{asmp:kernel}. There are constants $N,C,C_1,c,c_1>0$ (not depending on $f_0$) such that for any $n \geq N$, and any radius $r$ as in \ref{asmp:ls_kernel_radius_estimation}, the Laplacian smoothing estimator \smash{$\wh{f}$} in \eqref{eqn:laplacian_smoothing} with \smash{$\rho = M^{-4/(2 + d)} (nr^{d + 2})^{-1} n^{-2/(2 + d)}$} satisfies
	\begin{equation*}
	\bigl\|\wh{f} - f_0\bigr\|_n^2 \leq \frac{C}{\delta} M^{2d/(2 + d)} n^{-2/(2 + d)},
	\end{equation*}
	with probability at least $1 - \delta -  C_1n\exp(-c_1nr^d) - \exp(-c (M^2n)^{d/(2 + d)})$.
\end{theorem}

To summarize: for $d = 1,2,$ or $3$, with high probability, the Laplacian smoothing estimator \smash{$\wh{f}$} has in-sample mean squared error that is within a constant factor of the minimax error. Some remarks:
\begin{itemize}
	\item The first-order Sobolev space $H^1(\Xset)$ does not continuously embed into $C^0(\Xset)$ when $d>1$ (in general, the $k$th order space $H^k(\Xset)$ does not continuously embed into $C^0(\Xset)$ except if $2k>d$). For this reason, one really cannot speak of pointwise evaluation of a Sobolev function $f_0 \in H^1(\Xset)$ when $d>1$ (as we do in Theorem~\ref{thm:laplacian_smoothing_estimation1} by defining our target of estimation to be $f_0(X_i)$, $i=1,\ldots,n$). We can resolve this by appealing to what are known as \emph{Lebesgue points}, as explained in Appendix~\ref{app:preliminaries}.
	\item The assumption $M \leq n^{1/d}$ ensures that the upper bound provided in the theorem is meaningful (i.e., ensures it is of at most a constant order).
	\item The lower bound on $r$ imposed in condition \ref{asmp:ls_kernel_radius_estimation} is compatible with practice, where by far the most common choice of radius is the connectivity threshold \smash{$r \asymp (\log(n)/n)^{1/d}$}, which makes \smash{$G_{n,r}$} as sparse as possible while still being connected, for maximum computational efficiency. The upper bound may seem a bit more mysterious---we need it for technical reasons to ensure that \smash{$\wh{f}$} does not overfit, but we note that as a practical matter one rarely chooses $r$ to be so large anyway.
	\item It is possible to extend \smash{$\wh{f}$} to be defined on all of $\Xset$ and then evaluate the error of such an extension (as measured against $f_0$) in $\Leb^2(\Xset)$ norm. When \smash{$\wh{f}$} and $f_0$ are suitably smooth, tools from empirical process theory (see e.g., Chapter 14 of \citet{wainwright2019}) or approximation theory (e.g., Section 15.5 of \citet{johnstone2011}) guarantee that the $\Leb^2(\Xset)$ error is not too much greater than its in-sample counterpart. In fact, as shown in Appendix~\ref{subsec:laplacian_smoothing_estimation1_pf}, if $f_0$ is Lipschitz smooth and we extend \smash{$\wh{f}$} to be piecewise constant over the Voronoi tessellation induced by $X_1,\ldots,X_n$, then the out-of-sample error \smash{$\|\wh{f}-f_0\|_{\Leb^2(\Xset)}$} is within a negligible factor of the in-sample error \smash{$\|\wh{f}-f_0\|_n$}. We leave analysis of the Sobolev case to future work.
	\item When $f_0$ is Lipschitz smooth, we can also replace the factor of $\delta$ in the high probability bound by a factor of $\delta^2/n$, which is always smaller than $\delta$ when $\delta \in (0,1)$.
\end{itemize} 

When $d = 4$, our analysis results in an upper bound for the error of Laplacian smoothing that is within a \smash{$(\log n)^{1/3}$} factor of the minimax error rate. But when $d \geq 5$, our upper bounds do not match the minimax rates.
\begin{theorem}
	\label{thm:laplacian_smoothing_estimation2}
	Under the assumptions of Theorem~\ref{thm:laplacian_smoothing_estimation1}, if instead $\Xset$ has dimension $d = 4$, \smash{$r \asymp (\log n/n)^{1/4}$} and \smash{$\rho = M^{-2/3}(nr^6)^{-1}(\log n/n)^{1/3}$}, then we obtain 
	\begin{equation*}
	\bigl\|\wh{f} - f_0\bigr\|_n^2 \leq \frac{C}{\delta} M^{4/3} \biggl(\frac{\log n}{n}\biggr)^{1/3},
	\end{equation*}
	with the same probability guarantee as in Theorem~\ref{thm:laplacian_smoothing_estimation1}. If the dimension of $\Xset$ is $d \geq 5$, \smash{$r \asymp (\log n/n)^{1/d}$} and \smash{$\rho =  M^{-2/3}(nr^{2 + d})^{-1}n^{-4/(3d)}$}, then
	\begin{equation*}
	\bigl\|\wh{f} - f_0\bigr\|_n^2 \leq \frac{C}{\delta} M^{4/3} \biggl(\frac{\log n}{n}\biggr)^{4/(3d)},
	\end{equation*}
	again with the same probability guarantee.
\end{theorem}
This mirrors the conclusions of \citet{sadhanala16} who investigate estimation rates of Laplacian smoothing over the $d$-dimensional grid graph. These authors argue that their analysis is tight, and that it is likely the estimator, not the analysis, that is deficient when $d \geq 5$. Formalizing such a claim turns out to be harder in the random design setting than in the fixed design setting, and we leave it for future work. 

However, we do investigate the matter empirically. In Figure~\ref{fig:fig1}, we study the (in-sample) mean squared error of the Laplacian smoothing estimator as the dimension $d$ grows. Here $X_1,\ldots,X_n$ are sampled uniformly over $\Xset = [-1,1]^d$, and the regression function is taken as \smash{$f_0(x) \propto \Pi_{i = 1}^{d} \cos(a \pi x_i)$}, where $a = 2$ for $d = 2$, and $a = 1$ for $d \geq 3$. This regression function $f_0$ is quite smooth, and for $d = 2$ and $d = 3$ Laplacian smoothing appears to achieve or exceed the minimax rate. When $d = 4$, Laplacian smoothing appears modestly suboptimal; this fits with our theoretical upper bound, which includes a \smash{$(\log n)^{1/3}$} factor that plays a non-negligible role for these problem sizes ($n=1000$ to $n=10000$). On the other hand, when $d = 5$, Laplacian smoothing seems to be decidedly suboptimal. 

\begin{figure*}[tb]
	\includegraphics[width=.245\textwidth]{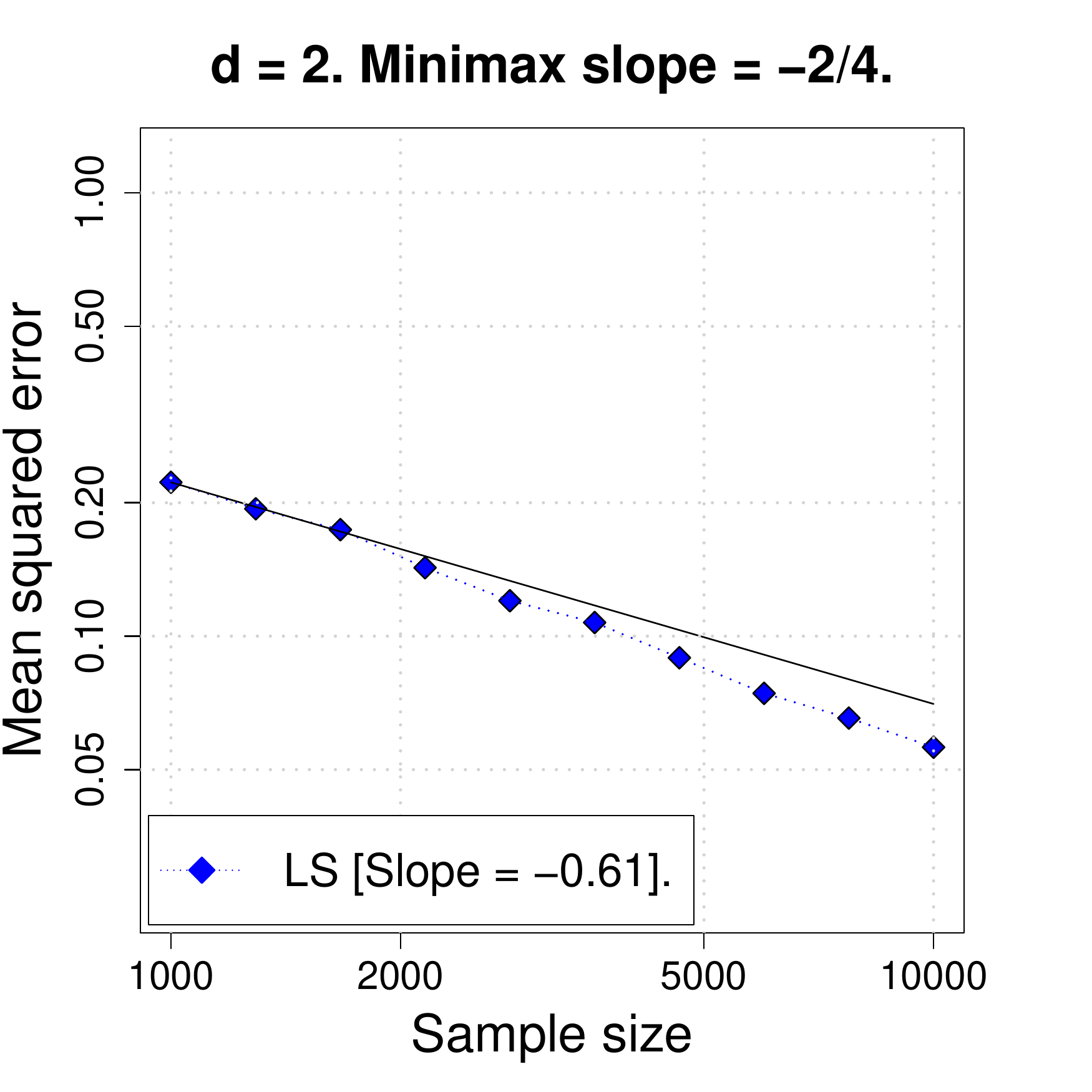}
	\includegraphics[width=.245\textwidth]{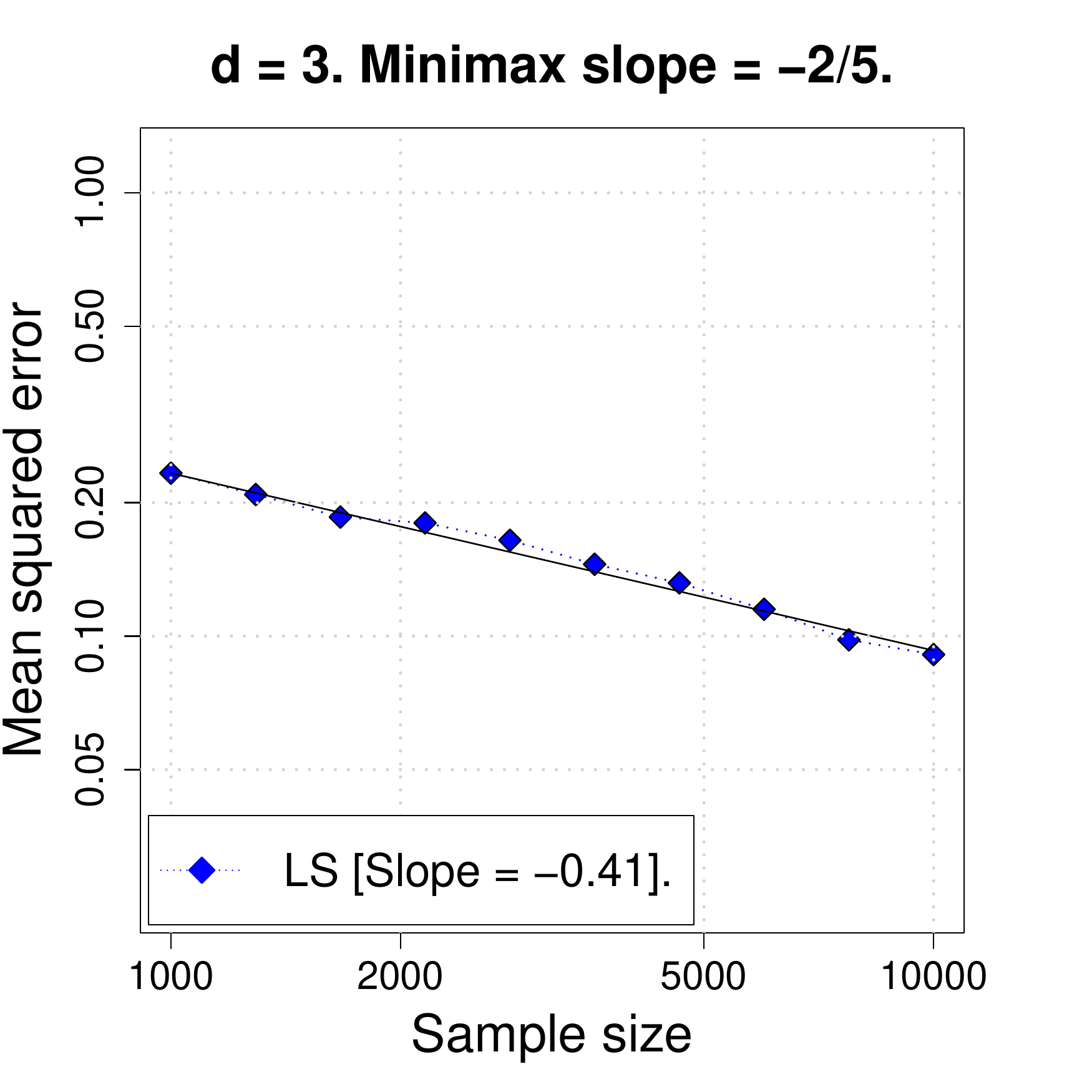} 
	\includegraphics[width=.245\textwidth]{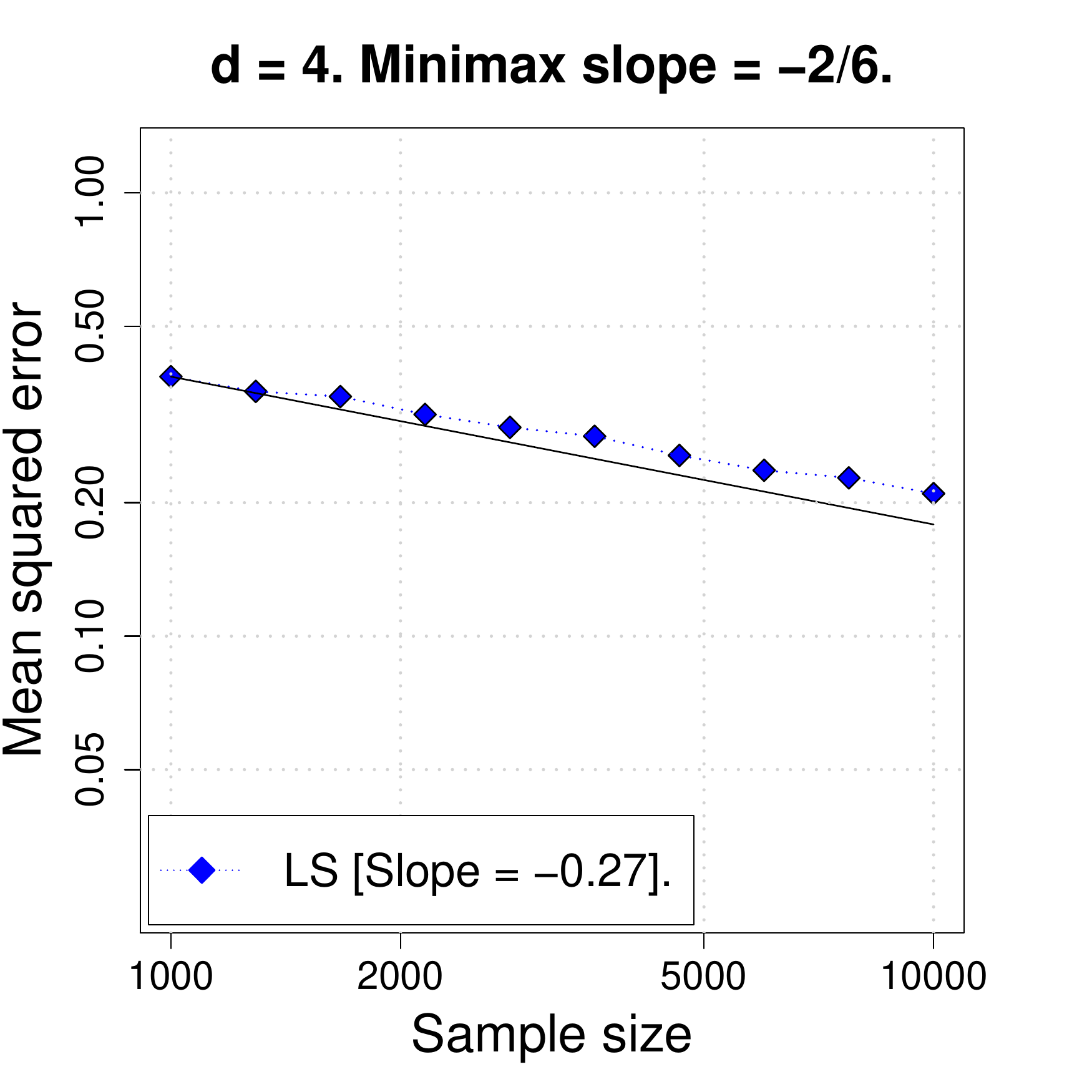}
	\includegraphics[width=.245\textwidth]{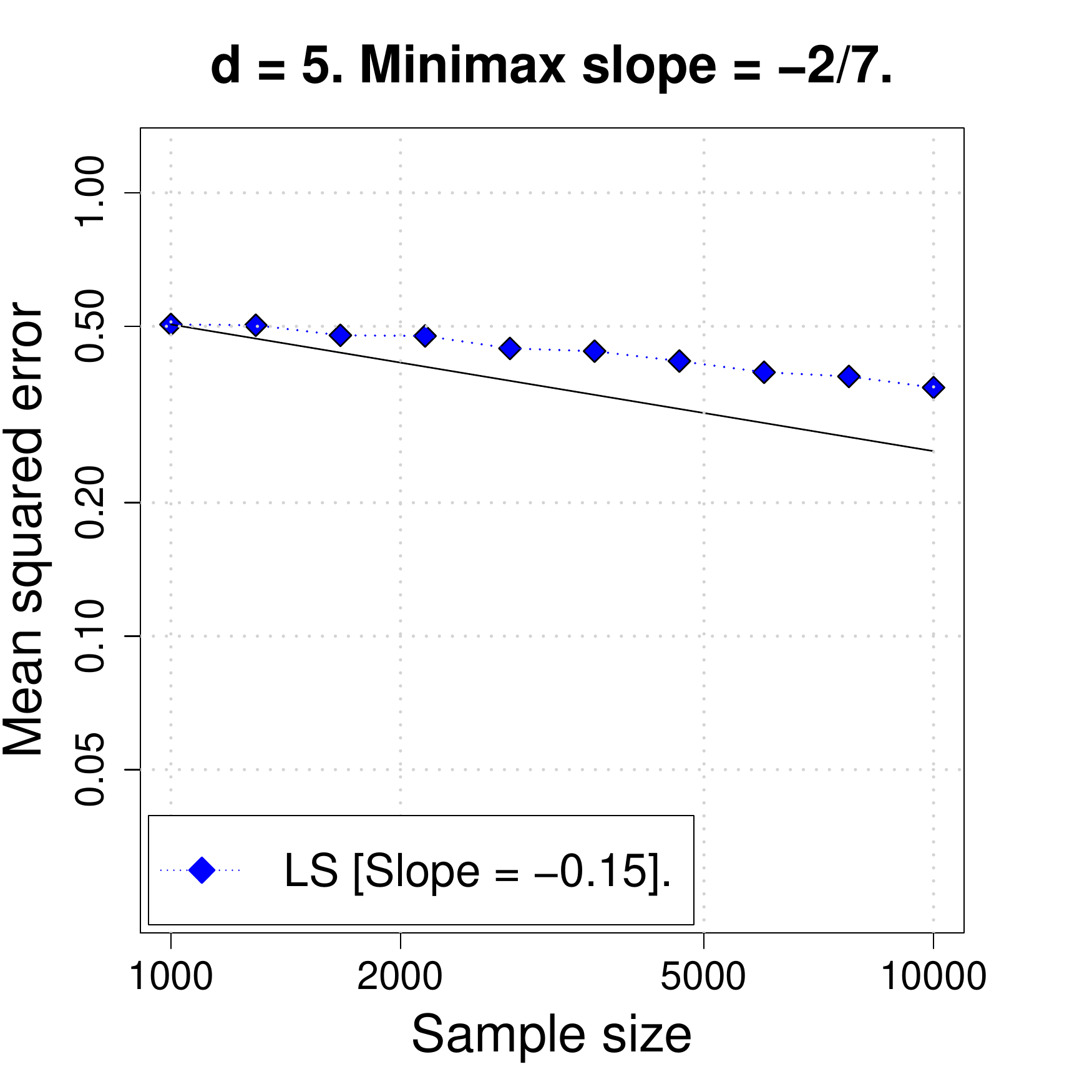}
	\caption{Mean squared error of Laplacian smoothing (\texttt{LS}) as a function of sample size $n$. Each plot is on the log-log scale, and the results are averaged over 5 repetitions, with Laplacian smoothing tuned for optimal average mean squared error. The black line shows the minimax rate (in slope only; the intercept is chosen to match the observed error).}
	\label{fig:fig1}
\end{figure*}

\paragraph{Testing Error of Laplacian Smoothing.}

For a given $0 < \alpha < 1$, define a threshold \smash{$\wh{t}_{\alpha}$} as
\begin{equation*}
\wh{t}_{\alpha} = \frac{1}{n}\sum_{k = 1}^{n} \frac{1}{(\rho \lambda_k + 1)^2} + \frac{1}{n}\sqrt{\frac{2}{\alpha} \sum_{k = 1}^{n} \frac{1}{(\rho \lambda_k + 1)^4}},
\end{equation*}
where we recall $\lambda_k$ is the $k$th smallest eigenvalue of \smash{$L$}. The Laplacian smoothing test is then simply
\begin{equation*}
\wh{\varphi} = \1\bigl\{\wh{T} > \wh{t}_{\alpha}\bigr\}.
\end{equation*} 
We show in Appendix~\ref{sec:fixed_graph_error_bounds} that \smash{$\wh{f}$} is a level-$\alpha$ test. In the next theorem, we upper bound the worst-case risk \smash{$R_n(\wh{\varphi},H^1(\mc{X},M),\epsilon)$} of \smash{$\wh{\varphi}$}, whenever $\epsilon$ is at least (a constant times) the critical radius given in~\eqref{eqn:sobolev_space_testing_critical_radius}. For this to hold, we will require a tighter range of scalings for the graph radius $r$.
\begin{enumerate}[label=(R\arabic*)]
	\setcounter{enumi}{1}
	\item 
	\label{asmp:ls_kernel_radius_testing}
	For constants $C_0,c_0>0$, the neighborhood graph radius $r$ satisfies
	\begin{equation*}
	C_0\biggl(\frac{\log n}{n}\biggr)^{\frac{1}{d}} \leq r \leq c_0 \wedge M^{\frac{(d - 8)}{8 + 2d}} n^{\frac{d - 20}{32 + 8d}}.
	\end{equation*}
\end{enumerate}
We will also require that the radius of the Sobolev class not be too large. Precisely, we will require $M \leq M_{\max}(d)$, where we define
\begin{equation*}
M_{\max}(d) :=
\begin{cases*}
n^{1/8} & \textrm{$d = 1$} \\
n^{(4 - d)/(4d)} & \textrm{$d \geq 2$}.
\end{cases*}
\end{equation*}

We now give Theorem~\ref{thm:laplacian_smoothing_testing}, our main testing result.
\begin{theorem}
	\label{thm:laplacian_smoothing_testing}
	Given i.i.d.\ draws $(X_i,Y_i)$, $i=1,\ldots,n$ from \eqref{eqn:signal_plus_noise_model}, assume $f_0 \in H^1(\Xset,M)$ where $\Xset \subseteq \Rd$ with $d < 4$, and $M \leq M_{\max}(d)$. Assume \ref{asmp:domain}, \ref{asmp:density} on the design distribution $P$, and assume \smash{$G_{n,r}$} is computed with a kernel $K$ satisfying \ref{asmp:kernel}. There exist constants $N,C,C_1,c_1>0$ such that for any $n \geq N$, and any radius $r$ as in \ref{asmp:ls_kernel_radius_testing}, the Laplacian smoothing test \smash{$\wh{\varphi}$} based on the estimator \smash{$\wh{f}$} in \eqref{eqn:laplacian_smoothing}, with \smash{$\rho = (nr^{d + 2})^{-1} n^{-4/(4 + d)} M^{-8/(4 + d)}$}, satisfies the following: for any $b \geq 1$, if
	\begin{equation}
	\label{eqn:laplacian_smoothing_testing}
	\epsilon^2 \geq C M^{2d/(4 + d)} n^{-4/(4 + d)}\biggl(b^2 + b\sqrt{\frac{1}{\alpha}}\biggr) ,
	\end{equation} 
	then the worst-case risk satisfies the upper bound: \smash{$R_n(\wh{\varphi},H^1(\mc{X},M), \epsilon) \leq C/b + C_1n\exp(-c_1nr^d)\bigr)$}.
\end{theorem}

Some remarks:
\begin{itemize}
	\item As mentioned earlier, Sobolev balls $H^1(\Xset,M)$ for $d \geq 4$ include quite irregular functions $f \not\in \Leb^4(\Xset)$. Proving tight lower bounds in this case is nontrivial, and as far as we understand such an analysis remains outstanding. On the other hand, if we explicitly assume that $f_0 \in \Leb^4(\Xset,M)$, then \citet{guerre02} show that the testing problem is characterized by a dimension-free lower bound $\epsilon^{2}(\Leb^4(\Xset,M)) \gtrsim n^{-1/2}$. Moreover, by setting $\rho = 0$ so that the resulting estimator \smash{$\wh{f}$} interpolates the responses $Y_1,\ldots,Y_n$, the subsequent test \smash{$\wh{\varphi}$} will achieve (up to constants) this lower bound. That is, for any $f_0 \in \Leb^4(\Xset,M)$ such that \smash{$\|f_0\|_{\Leb^2(\Xset)}^2 \geq C(b^2 + \sqrt{1/\alpha})n^{-1/2}$}, we have that \smash{$\Ebb_0[\wh{\varphi}] \leq \alpha$} and
	\begin{equation}
	\label{eqn:laplacian_smoothing_testing_low_smoothness}
	\Ebb_{f_0}\bigl[1 - \wh{\varphi}\bigr] \leq \frac{C(1 + M^4)}{b^2}.
	\end{equation} 
	\item To compute the data-dependent threshold \smash{$\wh{t}_{\alpha}$}, one must know all of the eigenvalues $\lambda_1,\ldots,\lambda_n$. Computing all these eigenvalues is far more expensive (cubic-time) than computing \smash{$\wh{T}$} in the first place (nearly-linear-time). But in practice we would not recommend using \smash{$\wh{t}_{\alpha}$} anyway, and would instead we make the standard recommendation to calibrate via a permutation test \citep{hoeffding1952}. Recent work \cite{kim2020minimax}, has shown that in a variety of closely related settings, calibration of a test statistic via the permutation test often retains minimax-optimal power, and we expect similar results to hold for the Laplacian smoothing-based test statistic.

\end{itemize}

\paragraph{More Discussion of Variational Analog.}

With some results in hand, let us pause to offer some explanation of why Laplacian smoothing can be optimal in settings where thin-plate splines are not even consistent. First, we elaborate on why this difference in performance is so surprising. As mentioned previously, the penalties in \eqref{eqn:laplacian_smoothing}, \eqref{eqn:thin_plate_spline} can be closely tied together: \citet{bousquet03} show that for $f \in C^2(\Xset)$, 
\begin{equation}
\label{eqn:seminorm_consistency}
\begin{aligned}
\lim \frac{1}{n^2 r^{d + 2}} f^\top L f & = \int_{\Xset} f(x) \cdot \Delta_Pf(x) p(x) \,dx \\
& = \int_{\Xset} \|\nabla f(x)\|_2^2 p^2(x) \,dx.
\end{aligned}
\end{equation}
In the above, the limit is as $n \to \infty$ and $r \to 0$, $\Delta_P$ is the (weighted) Laplace-Beltrami operator
\begin{equation*}
\Delta_Pf := -\frac{1}{p} \dive\bigl(p^2\nabla f),
\end{equation*}
and the second equality follows using integration by parts.\footnote{Assuming $f$ satisfies Neumann boundary conditions.} To be clear, this argument does not formally imply that the Laplacian eigenmaps estimator \smash{$\wh{f}$} and the thin-plate spline estimator \smash{$\wt{f}$} are close (for one, note that \eqref{eqn:seminorm_consistency} holds for $f \in C^2(\Xset)$, whereas the optimization in \eqref{eqn:thin_plate_spline} considers a much broader set of continuous functions with weak derivatives in $\Leb^2(\Xset)$). But it does seem to suggest that the two estimators should behave somewhat similarly. 

Of course, we know this is not the case: \smash{$\wh{f}$} and \smash{$\wt{f}$} look very different when $d > 1$. What is driving this difference? The key point is that the discretization imposed by the graph $G_{n,r}$---which might seem problematic at first glance---turns out to be a blessing. The problem with~\eqref{eqn:thin_plate_spline} is that the class $H^1(\Xset)$, which fundamentally underlies the criterion, is far ``too big'' for $d > 1$. This is meant in various related senses. By the Sobolev embedding theorem, for $d>1$, the class $H^1(\Xset)$ does not continuously embed into any H\"{o}lder space; and in fact it does not even continuously embed into $C^0(\Xset)$. Thus we cannot really restrict the optimization to \emph{continuous} and weakly differentiable functions, as we could when $d=1$ (the smoothing spline case), without throwing out a substantial subset of functions in $H^1(\Xset)$. Even among continuous and differentiable functions $f$, as we explained previously, we can use ``bump'' functions (as in \citet{green93}) to construct $f$ that interpolates the pairs $(X_i,Y_i)$, $i=1,\ldots,n$ and achieves arbitrarily small penalty (and hence criterion) in \eqref{eqn:thin_plate_spline}. In this sense, any estimator resulting from solving \eqref{eqn:thin_plate_spline} will clearly be inconsistent. 

On the other hand, problem \eqref{eqn:laplacian_smoothing} is finite-dimensional. As a result \smash{$\wh{f}$} has far less capacity to overfit than does \smash{$\wt{f}$}, for any given sample size $n$. Discretization is not the only way to make the problem \eqref{eqn:thin_plate_spline} more tractable: for instance, one can replace the penalty \smash{$\int_{\Xset} \|\nabla f(x)\|_2^2 \,dx$} with a stricter choice like \smash{$\esssup_{x \in \Xset} \|\nabla f(x)\|_2$}, or conduct the optimization over some finite-dimensional linear subspace of $H^1(\Xset)$ (i.e., use a sieve). While these solutions do improve the statistical properties of \smash{$\wt{f}$} for $d > 1$ (see e.g., \citet{birge1993,birge1998,vandergeer2000}), Laplacian smoothing is generally speaking much simpler and more computationally friendly. In addition, the other approaches are usually specifically tailored to the domain $\Xset$, in stark contrast to~\smash{$\wh{f}$}.

\paragraph{Overview of Analysis.}

The comparison with thin-plate splines highlights some surprising differences between \smash{$\wh{f}$} and \smash{$\wt{f}$}. Such differences also preclude us from analyzing \smash{$\wh{f}$} by, say, using \eqref{eqn:seminorm_consistency} to establish a coupling between \smash{$\wh{f}$} and \smash{$\wt{f}$}---we know this cannot work, because we would like to prove meaningful error bounds on \smash{$\wh{f}$} in regimes where no such bounds exist for \smash{$\wt{f}$}.

Instead we take a different approach, and directly analyze the error of \smash{$\wh{f}$} and \smash{$\wh{T}$} using a bias-variance decomposition (conditional on $X_1,\ldots,X_n$). A standard calculation shows that
\begin{equation*}
\bigl\|\wh{f} - f_0\bigr\|_n^2 \leq \underbrace{\vphantom{\sum_{k=1}^n}\frac{2\rho}{n} \bigl(f_0^\top L f_0\bigr)}_{\textrm{bias}} + \underbrace{\frac{10}{n} \sum_{k = 1}^{n} \frac{1}{(\rho \lambda_k + 1)^2}}_{\textrm{variance}},
\end{equation*}
and likewise that \smash{$\wh{\varphi}$} has small risk whenever
\begin{equation*}
\bigl\|f_0\bigr\|_n^2 \geq \underbrace{\vphantom{\sqrt{\sum_{k=1}^n}}\frac{2\rho}{n} \bigl(f_0^\top L f_0\bigr)}_{\textrm{bias}} + \underbrace{\frac{2\sqrt{2/\alpha} + 2b}{n} \sqrt{\sum_{k = 1}^{n} \frac{1}{(\rho \lambda_k + 1)^4}}}_{\textrm{variance}}.
\end{equation*}
The bias and variance terms are each functions of the random graph \smash{$G_{n,r}$}, and hence are themselves random. To upper bound them, we build on some recent works \citep{burago2014,trillos2019,calder2019} regarding the consistency of neighborhood graphs to establish the following lemmas. These lemmas assume \ref{asmp:domain}, \ref{asmp:density} on the design distribution $P$, and \ref{asmp:kernel} on the kernel used to compute the neighborhood graph \smash{$G_{n,r}$}. 

\begin{lemma}
	\label{lem:graph_sobolev_seminorm}
	There are constants $N,C_2 > 0$ such that for $n \geq N$, $r \leq c_0$, and $f \in H^1(\Xset)$, with probability at least $1 - \delta$, it holds that
	\begin{equation}
	\label{eqn:graph_sobolev_seminorm}
	f^\top L f \leq \frac{C_2}{\delta} n^2 r^{d + 2} |f|_{H^1(\Xset)}^2.
	\end{equation}
\end{lemma}

\begin{lemma}
	\label{lem:neighborhood_eigenvalue} 
	There are constants $N,C_1,C_3,c_1,c_3 > 0$ such that for $n \geq N$ and $C_0(\log n/n)^{1/d} \leq r \leq c_0$, with probability at least $1 - C_1n\exp(-c_1nr^d)$, it holds that
	\begin{equation}
	\label{eqn:neighborhood_eigenvalue}
	c_3A_{n,r}(k) \leq \lambda_k \leq C_3A_{n,r}(k), ~~\textrm{for $2 \leq k \leq n$},
	\end{equation}
	where \smash{$A_{n,r}(k) = \min\{nr^{d + 2}k^{2/d},nr^d\}$}.
\end{lemma}

Lemma~\ref{lem:graph_sobolev_seminorm} gives a direct upper bound on the bias term. Lemma~\ref{lem:neighborhood_eigenvalue} leads to a sufficiently tight upper bound on the variance term whenever the radius $r$ is sufficiently small; precisely, when $r$ is upper bounded as in \ref{asmp:ls_kernel_radius_estimation} for estimation, or \ref{asmp:ls_kernel_radius_testing} for testing. The parameter $\rho$ is then chosen to minimize the sum of these upper bounds on bias and variance, as usual, and some straightforward calculations give Theorems~\ref{thm:laplacian_smoothing_estimation1}-\ref{thm:laplacian_smoothing_testing}.

It may be useful to give one more perspective on our approach. A common strategy in analyzing penalized least squares estimators is to assume two properties: first, that the regression function $f_0$ lies in (or near) a ball defined by the penalty operator; second, that this ball is reasonably small, e.g., as measured by metric entropy, or Rademacher complexity, etc. In contrast, in Laplacian smoothing, the penalty induces a ball
\begin{equation*}
H^1(G_{n,r},M) := \{f: f^\top L f \leq M^2\}
\end{equation*}
that is data-dependent and random, and so we do not have access to either of the aforementioned properties a priori, and instead, must prove they hold with high probability. In this sense, our analysis is different than the typical one in nonparametric regression.

\section{Manifold Adaptivity}
\label{sec:manifold_adaptivity}

The minimax rates $n^{-2/(2 + d)}$ and $n^{-4/(4 + d)}$, in estimation and testing, suffer from the curse of dimensionality. However, in practice it can be often reasonable to assume a \emph{manifold hypothesis}: that the data $X_1,\ldots,X_n$ lie on a manifold $\Xset$ of $\Rd$ that has intrinsic dimension $m < d$. Under such an assumption, it is known \citep{bickel2007,ariascastro2018} that the optimal rates over $H^1(\Xset)$ are now $n^{-2/(2 + m)}$ (for estimation) and $n^{-4/(4 + m)}$ (for testing), which are much faster than the full-dimensional error rates when $m \ll d$. 



On the other hand, a theory has been developed \citep{belkin03,belkin05,niyogi2008finding,niyogi2013,balakrishnan2012minimax,balakrishnan2013cluster} establishing that the neighborhood graph $G_{n,r}$ can ``learn'' the manifold $\Xset$ in various senses, so long as $\Xset$ is locally linear. We contribute to this line of work by showing that under the manifold hypothesis, Laplacian smoothing achieves the tighter minimax rates over $H^1(\Xset)$.

\paragraph{Error Rates Assuming the Manifold Hypothesis.}

The conditions and results presented here will be largely similar to the previous ones, except with the ambient dimension $d$ replaced by the intrinsic dimension $m$. For the remainder, we assume the following.
\begin{enumerate}[label=(P\arabic*)]
	\setcounter{enumi}{2}
	\item 
	\label{asmp:domain_manifold}
	$P$ is supported on a compact, connected, smooth manifold $\Xset$ embedded in $\Rd$, of dimension $m \leq d$. The manifold is without boundary and has positive reach \citep{federer1959}.
	\item 
	\label{asmp:density_manifold} 
	$P$ admits a density $p$ with respect to the volume form of $\Xset$ such that 
	\begin{equation*}
	0 < p_{\min} \leq p(x) \leq p_{\max} < \infty, ~~\textrm{for all $x \in \Xset$}. 
	\end{equation*}
	Additionally, $p$ is Lipschitz on $\Xset$, with Lipschitz constant $L_p$.
\end{enumerate}

Under the assumptions \ref{asmp:domain_manifold}, \ref{asmp:density_manifold}, and \ref{asmp:kernel}, and for a suitable range of $r$, the error bounds on the estimator \smash{$\wh{f}$} and test \smash{$\wh{\varphi}$} will depend on $m$ instead of $d$. 
\begin{enumerate}[label=(R\arabic*)]
	\setcounter{enumi}{3}
	\item 
	\label{asmp:ls_kernel_radius_estimation_manifold}
	For constants $C_0,c_0>0$, the neighborhood graph radius $r$ satisfies 
	\begin{equation*}
	C_0\biggl(\frac{\log n}{n}\biggr)^{\frac{1}{m}} \leq r \leq c_0 \wedge  M^{\frac{(m - 4)}{(4 + 2m)}} n^{\frac{-3}{(4 + 2m)}}.
	\end{equation*}
\end{enumerate}

\begin{theorem}
	\label{thm:laplacian_smoothing_estimation_manifold}
	As in Theorem \ref{thm:laplacian_smoothing_estimation1}, but where $\Xset \subseteq \Rd$ is a manifold with intrinsic dimension $m < 4$, the design distribution $P$ obeys \ref{asmp:domain_manifold}, \ref{asmp:density_manifold}, and $M \leq n^{1/m}$. There are constants $N,C,c>0$ (not depending on $f_0$) such that for any \smash{$n \geq N$}, and any $r$ as in \ref{asmp:ls_kernel_radius_estimation_manifold}, the Laplacian smoothing estimator \smash{$\wh{f}$} in \eqref{eqn:laplacian_smoothing}, with \smash{$L = L_{n,r}$} and \smash{$\rho = M^{-4/(2 + m)} (nr^{m + 2})^{-1} n^{-2/(2 + m)}$}, satisfies
	\begin{equation*}
	\bigl\|\wh{f} - f_0\bigr\|_n^2 \leq \frac{C}{\delta} M^{2m/(2 + m)} n^{-2/(2 + m)},
	\end{equation*}
	with probability at least $1 - \delta -  C n\exp(-cnr^m) - \exp(-c(M^2n)^{m/(2+m)})$.
\end{theorem}
In a similar vein, we obtain results for manifold adaptive testing under the following condition on the graph radius parameter.
\begin{enumerate}[label=(R\arabic*)]
	\setcounter{enumi}{4}
	\item 
	\label{asmp:ls_kernel_radius_testing_manifold}
	For constants $C_0,c_0>0$, the neighborhood graph radius $r$ satisfies 
	\begin{equation*}
	C_0\biggl(\frac{\log n}{n}\biggr)^{\frac{1}{m}} \leq r \leq c_0 \wedge M^{\frac{(m - 8)}{8 + 2m}} n^{\frac{m - 20}{32 + 8m}}.
	\end{equation*}
\end{enumerate}

\begin{theorem}
	\label{thm:laplacian_smoothing_testing_manifold}
	As in Theorem \ref{thm:laplacian_smoothing_testing}, but where $\Xset \subseteq \Rd$ is a manifold with intrinsic dimension $m < 4$, $M \leq M_{\max}(m)$, and the design distribution $P$ obeys \ref{asmp:domain_manifold}, \ref{asmp:density_manifold}.
	There are constants $N,C,c>0$ such that for any $n \geq N$, and any $r$ as in \ref{asmp:ls_kernel_radius_testing_manifold}, the Laplacian smoothing test \smash{$\wh{\varphi}$} based on the estimator \smash{$\wh{f}$} in \eqref{eqn:laplacian_smoothing}, with \smash{$\rho = (nr^{m + 2})^{-1} n^{-4/(4 + m)} M^{-8/(4 + m)}$}, satisfies the following: for any $b \geq 1$, if
	\begin{equation}
	\label{eqn:laplacian_smoothing_testing_manifold}
	\epsilon^2 \geq C M^{2m/(4 + m)} n^{-4/(4 + m)}\biggl(b^2 + b\sqrt{\frac{1}{\alpha}}\biggr) ,
	\end{equation} 
	then the worst-case risk satisfies the upper bound: \smash{$R_n(\wh{\varphi},H^1(\mc{X},M), \epsilon) \leq C/b + Cn\exp(-cnr^m)$}.
\end{theorem}

The proofs of Theorems \ref{thm:laplacian_smoothing_estimation_manifold} and \ref{thm:laplacian_smoothing_testing_manifold} proceed in a similar manner to that of Theorems \ref{thm:laplacian_smoothing_estimation1} and \ref{thm:laplacian_smoothing_testing}. The key difference is that in the manifold setting, the equations \eqref{eqn:graph_sobolev_seminorm} and \eqref{eqn:neighborhood_eigenvalue} used to upper bound bias and variance will hold with $d$ replaced by $m$.

We emphasize that little about $\Xset$ need be known for Theorems~\ref{thm:laplacian_smoothing_estimation_manifold} and \ref{thm:laplacian_smoothing_testing_manifold} to hold. Indeed, all that is needed is the intrinsic dimension $m$, to properly tune $r$ and $\rho$ (from a theoretical point of view), and otherwise \smash{$\wh{f}$} and \smash{$\wh{\varphi}$} are computed without regard to $\Xset$. In contrast, the penalty in \eqref{eqn:thin_plate_spline} would have to be specially tailored to work in this setting, revealing another advantage of the discrete approach over the variational one.

\section{Discussion}
\label{sec:discussion}

We have shown that Laplacian smoothing, computed over a neighborhood graph, can be optimal for both estimation and goodness-of-fit testing over Sobolev spaces. There are many extensions worth pursuing, and several have already been mentioned. We conclude by mentioning a couple more. In practice, it is more common to use a $k$-nearest-neighbor (kNN) graph than a neighborhood graph, due to the guaranteed connectivity and sparsity of the former; we suspect that by building on the work of \citet{calder2019}, one can show that our main results all hold under the kNN graph as well. In another direction, one can also generalize Laplacian smoothing by replacing the penalty \smash{$f^\top L f$} with \smash{$f^\top L^{s} f$}, for an integer $s > 1$. The hope is that this would then achieve minimax optimal rates over the higher-order Sobolev class $H^s(\Xset)$. 

\subsection*{Acknowledgments}
AG and RJT were supported by ONR grant N00014-20-1-2787. AG and SB were supported by NSF grants DMS-1713003 and CCF-1763734.

\bibliographystyle{plainnat}
\bibliography{../../../graph_regression_bibliography} 

\appendix

\section{Preliminaries}
\label{app:preliminaries}
\noindent In the appendix, we provide complete proofs of all results. 
Our main theorems (Theorems~\ref{thm:laplacian_smoothing_estimation1}-\ref{thm:laplacian_smoothing_testing_manifold}) all follow the same general proof strategy of first establishing bounds in the fixed-design setup. In Section~\ref{sec:fixed_graph_error_bounds}, we establish (estimation or testing) error bounds which hold for any graph $G$; these bounds are stated with respect to (functionals of) the graph $G$, and allow us to upper bound the error of $\wh{f}$ and $\wh{\varphi}$ conditional on the design $\{X_1,\ldots,X_n\} = \{x_1,\ldots,x_n\}$. In Sections~\ref{sec:graph_sobolev_seminorm}, \ref{sec:graph_eigenvalues}, \ref{sec:empirical_norm}, and \ref{sec:manifold} we develop all the necessary probabilistic estimates on these functionals, for the particular random neighborhood graph $G = G_{n,r}$. It is in these sections where we invoke our various assumptions on the distribution $P$ and regression function $f_0$. In Section~\ref{sec:main_results}, we prove our main theorems and some other results. In Section~\ref{sec:concentration}, we state a few concentration bounds that we use repeatedly in our proofs.

\paragraph{Pointwise evaluation of Sobolev functions.}
First, however, as promised in our main text we clarify what is meant by pointwise evaluation of the regression function $f_0$. Strictly speaking, each $f \in H^1(X)$ is really an equivalence class, defined only up to sets of Lebesgue measure 0. In order to make sense of the evaluation $x \mapsto f(x)$, one must therefore pick a representative $f^{\star} \in f$. When $d = 1$, this is resolved in a standard way---since $H^1(\mc{X})$ embeds continuously into $C^0(\mc{X})$, there exists a continuous version of every $f \in H^1(\mc{X})$, and we take this continuous version as the representative $f^{\star}$. On the other hand, when $d \geq 2$, the Sobolev space $H^1(\mc{X})$ does not continuously embed into $C^0(\mc{X})$, and we must choose representatives in a different manner. In this case we let $f^{\star}$ be the precise representative \citep{evans15}, defined pointwise at points $x \in \mc{X}$ as
\begin{equation*}
f^{\star}(x) = 
\begin{dcases*}
\lim_{\varepsilon \to 0} \frac{1}{\nu(B(x,\varepsilon))} \int_{B(x,\varepsilon)} f(z) dz,&~~\textrm{if the limit exists,} \\
0,&~~\textrm{otherwise.}
\end{dcases*}
\end{equation*}
Note that when $d = 1$, the precise representative of any $f \in H^1(\mc{X})$ is continuous. 

Now we explain why the particular choice of representative is not crucial, using the notion of a Lebesgue point. Recall that for a locally Lebesgue integrable function $f$, a given point $x \in \mc{X}$ is a \emph{Lebesgue point} of $f$ if the limit of $1/(\nu(B(x,\varepsilon)))\int_{B(x,\varepsilon)} f(x) dx$ as $\varepsilon \to 0$ exists, and satisfies
\begin{equation*}
\lim_{\varepsilon \to 0} \frac{1}{\nu\bigl(B(x,\varepsilon)\bigr)} \int_{B(x,\varepsilon)} f(x) dx = f(x).
\end{equation*}
Let $E$ denote the set of Lebesgue points of $f$. By the Lebesgue differentiation theorem \citep{evans15}, if $f \in L^1(\mc{X})$ then almost every $x \in \mc{X}$ is a Lebesgue point, $\nu(\mc{X}\setminus E) = 0$. Since $f_0 \in H^1(\mc{X}) \subseteq \Leb^1(\mc{X})$, we can conclude that any function $g_0 \in f_0$ disagrees with the precise representative $f_0^{\star}$ only on a set of Lebesgue measure 0. Moreover, since we always assume the design distribution $P$ has a continuous density, with probability $1$ it holds that $g_0(X_i) = f_0^{\star}(X_i)$ for all $i = 1,\ldots,n$. This justifies the notation $f_0(X_i)$ used in the main text.

\section{Graph-dependent error bounds}
\label{sec:fixed_graph_error_bounds}

In this section, we adopt the fixed design perspective; or equivalently, condition on $X_i = x_i$ for $i = 1,\ldots,n$. Let $G = \bigl([n],W\bigr)$ be a fixed graph on $\{1,\ldots,n\}$ with Laplacian matrix $\Lap = D - W$. The randomness thus all comes from the responses 
\begin{equation}
\label{eqn:fixed_graph_regression_model}
Y_i = f_{0}(x_i) + \varepsilon_i
\end{equation}
where the noise variables $\varepsilon_i$ are independent $N(0,1)$. In the rest of this section, we will mildly abuse notation and write $f_0 = (f_0(x_1),\ldots,f_0(x_n)) \in \Reals^n$. We will also write ${\bf Y} = (Y_1,\ldots,Y_n)$.

Recall \eqref{eqn:laplacian_smoothing} and \eqref{eqn:laplacian_smoothing_test}: the Laplacian smoothing estimator of $f_0$ on $G$ is
\begin{equation*}
\label{eqn:ls_G}
\wh{f} := \argmin_{f \in \Reals^n} \biggl\{ \sum_{i = 1}^{n}(Y_i - f_i)^2 + \rho \cdot f^{\top} \Lap f \biggr\} = (\rho \Lap + \Id)^{-1}{\bf Y}.
\end{equation*}
and the Laplacian smoothing test statistic is 
\begin{equation*}
\label{eqn:ls_ts_G}
\wh{T} := \frac{1}{n} \|\wh{f}\|_2^2.
\end{equation*}
We note that in this section, many of the derivations involved in upper bounding the estimation error of $\wh{f}$ are similar to those of \cite{sadhanala16}, with the difference being that we seek bounds in high probability rather than in expectation. We keep the work here self-contained for purposes of completeness.

\subsection{Error bounds for linear smoothers}

Let $S \in \Reals^{n \times n}$ be a fixed square, symmetric matrix, and let 
\begin{equation*}
\wc{f} := SY
\end{equation*}
be a linear estimator of $f_0$. In  Lemma~\ref{lem:linear_smoother_fixed_graph_estimation} we upper bound the error $\frac{1}{n}\|\wc{f} - f_0\|_2^2$ as a function of the eigenvalues of $S$. Let $\lambdavec(S) = (\lambda_1(S),\ldots,\lambda_n(S)) \in \Reals^n$ denote these eigenvalues, and let $v_k(S)$ denote the corresponding unit-norm eigenvectors, so that $S = \sum_{k = 1}^{n} \lambda_k(S) \cdot v_k(S) v_k(S)^{\top}$. Denote $Z_k = v_k(S)^{\top} \varepsilon$, and observe that ${\bf Z} = (Z_1,\ldots,Z_n) \sim N(0,\Id)$. 

\begin{lemma}
	\label{lem:linear_smoother_fixed_graph_estimation}
	Let $\wc{f} = SY$ for a square, symmetric matrix, $S \in \Reals^{n \times n}$. Then
	\begin{equation*}
	\Pbb_{f_0}\biggl(\frac{1}{n}\bigl\|\wc{f} - f_0\bigr\|_2^2 \geq \frac{10}{n} \bigl\|\lambdavec(S)\bigr\|_2^2 + \frac{2}{n}\bigl\|(S - I)f_0\bigr\|_2^2\biggr) \leq 1 - \exp\Bigl(-\bigl\|\lambdavec(S)\bigr\|_2^2\Bigr)
	\end{equation*}
\end{lemma}
Here we have written $\Pbb_{f_0}(\cdot)$ for the probability law under the regression ``function'' $f_0 \in \Reals^n$. 

In Lemma~\ref{lem:linear_smoother_fixed_graph_testing}, we upper bound the error of a test involving the statistic $\|\wc{f}\|_2^2 = {\bf Y}^{\top} S^2 {\bf Y}$. We will require that $S$ be a \emph{contraction}, meaning that it has operator norm no greater than $1$, $\|Sv\|_2 \leq \|v\|_2$ for all $v \in \Reals^n$.
\begin{lemma}
	\label{lem:linear_smoother_fixed_graph_testing}
	Let $\wc{T} = {\bf Y}^{\top} S^2 {\bf Y}$ for a square, symmetric matrix $S \in \Reals^{n \times n}$. Suppose $S$ is a contraction. Define the threshold $\wc{t}_{\alpha}$ to be 
	\begin{equation}
	\label{eqn:linear_smoother_fixed_graph_threshold}
	\wc{t}_{\alpha} := \norm{\lambdavec(S)}_2^2 + \sqrt{\frac{2}{\alpha}} \norm{\lambdavec(S)}_4^2.
	\end{equation}
	It holds that:
	\begin{itemize}
		\item \textbf{Type I error.}
		\begin{equation}
		\label{eqn:linear_smoother_fixed_graph_testing_typeI}
		\Pbb_0\bigl(\wc{T} > \wc{t}_{\alpha}\bigr) \leq \alpha.
		\end{equation}
		\item \textbf{Type II error.} Under the further assumption
		\begin{equation}
		\label{eqn:linear_smoother_fixed_graph_testing_critical_radius}
		f_0^{\top} S^2 f_0 \geq \Bigl(2\sqrt{\frac{2}{\alpha}} + 2b\Bigr) \cdot \norm{\lambdavec(S)}_4^2,
		\end{equation}
		then
		\begin{equation}
		\label{eqn:linear_smoother_fixed_graph_testing_typeII}
		\Pbb_{f_0}\bigl(\wc{T} \leq \wc{t}_{\alpha}\bigr) \leq \frac{1}{b^2} + \frac{16}{b \norm{\lambdavec(S)}_4^{2}}.
		\end{equation}
	\end{itemize}
\end{lemma}

\paragraph{Proof of Lemma~\ref{lem:linear_smoother_fixed_graph_estimation}.}
The expectation $\Ebb_{f_0}[\wc{f}] = Sf_0$, and by the triangle inequality,
\begin{align*}
\frac{1}{n}\bigl\|\wc{f} - f_0\bigr\|_2^2 & \leq \frac{2}{n}\Bigl(\bigl\|\wc{f} - \Ebb_{f_0}[\wc{f}]\bigr\|_2^2 + \bigl\|\Ebb_{f_0}[\wc{f}] - f_0\bigr\|_2^2\Bigr) \\ 
& = \frac{2}{n}\Bigl(\bigl\|S\varepsilon\bigr\|_2^2 + \bigl\|(S - I)f_0\bigr\|_2^2\Bigr).
\end{align*}
Writing $\norm{S\varepsilon}_2^2 = \sum_{k = 1}^{n} \lambda_k(S)^2 Z_k^2$, the claim follows from the result of \citet{laurent00} on concentration of $\chi^2$-random variables, which for completeness we restate in Lemma~\ref{lem:chi_square_bound}. To be explicit, taking $t = \norm{\lambdavec(S)}_2^2$ in Lemma~\ref{lem:chi_square_bound} completes the proof of Lemma~\ref{lem:linear_smoother_fixed_graph_estimation}. 

\paragraph{Proof of Lemma~\ref{lem:linear_smoother_fixed_graph_testing}.}
We compute the mean and variance of $T$ as a function of $f_0$, then apply Chebyshev's inequality.

\textit{Mean.} We make use of the eigendecomposition $S = \sum_{k = 1}^{n} \lambda_k(S) \cdot v_k(S) v_k(S)^{\top}$ to obtain
\begin{equation}
\label{pf:linear_smoother_fixed_graph_testing1}
\begin{aligned}
\wc{T} & = f_0^{\top} S^2 f_0 + 2 f_0^{\top} S^2 \varepsilon + \varepsilon^{\top} S^2 \varepsilon \\
& = f_0^{\top} S^2 f_0 + 2 f_0^{\top} S^2 \varepsilon + \sum_{k = 1}^{n}  \bigl(\lambda_k(S)\bigr)^2 (\varepsilon^{\top} v_k(S))^2 \\
& = f_0^{\top} S^2 f_0 + 2 f_0^{\top} S^2 \varepsilon + \sum_{k = 1}^{n}  \bigl(\lambda_k(S)\bigr)^2 Z_k^2,
\end{aligned}
\end{equation}
implying
\begin{equation}
\label{pf:linear_smoother_fixed_graph_testing_mean}
\Ebb_{f_0}\bigl[\wc{T}\bigr] = f_0^{\top} S^2 f_0 + \sum_{k = 1}^{n} \bigl(\lambda_k(S)\bigr)^2.
\end{equation}

\textit{Variance.} We start from \eqref{pf:linear_smoother_fixed_graph_testing1}. Recalling that $\Var(Z_k^2) = 2$, it follows from the Cauchy-Schwarz inequality that
\begin{equation}
\label{pf:linear_smoother_fixed_graph_testing_var}
\Var_{f_0}\bigl[\wc{T}\bigr] \leq 8 f_0^{\top} S^4 f_0 + 4 \sum_{k = 1}^{n} \bigl(\lambda_k(S)\bigr)^4.
\end{equation}

\textit{Bounding Type I and Type II error.} The upper bound \eqref{eqn:linear_smoother_fixed_graph_testing_typeI} on Type I error follows immediately from \eqref{pf:linear_smoother_fixed_graph_testing_mean}, \eqref{pf:linear_smoother_fixed_graph_testing_var}, and Chebyshev's inequality.

We now establish the upper bound~\eqref{eqn:linear_smoother_fixed_graph_testing_typeII} on Type II error. From assumption~\eqref{eqn:linear_smoother_fixed_graph_testing_critical_radius}, we see that $f_0^{\top} S^2 f_0^{\top} - \wc{t}_{\alpha} \leq 0$. As a result,
\begin{align*}
\Pbb_{f_0}\Bigl(\wc{T} \leq \wc{t}_{\alpha}\Bigr) & = \Pbb_{f_0}\Bigl(\wc{T} - \Ebb_{f_0}\bigl[\wc{T}\bigr] \leq \wc{t}_{\alpha} - \Ebb_{f_0}\bigl[\wc{T}\bigr]\Bigr) \\ 
& \leq \Pbb_{f_0}\Bigl(\Bigl|\wc{T} - \Ebb_{f_0}\bigl[\wc{T}\bigr]\Bigr| \geq \Bigl|\wc{t}_{\alpha} - \Ebb_{f_0}\bigl[\wc{T}\bigr]\Bigr|\Bigr) \\ 
& \leq \frac{\Var_{f_0}\bigl[\wc{T}\bigr]}{\bigl(\wc{t}_{\alpha} - \Ebb_{f_0}\bigl[\wc{T}\bigr]\bigr)^2},
\end{align*}
where the last line follows from Chebyshev's inequality. Plugging in the expressions~\eqref{pf:linear_smoother_fixed_graph_testing_mean} and~\eqref{pf:linear_smoother_fixed_graph_testing_var} for the mean and variance of $\wc{T}$, as well as the definition of $\wc{t}_{\alpha}$ in~\eqref{eqn:linear_smoother_fixed_graph_threshold}, we obtain that
\begin{equation}
\label{pf:linear_smoother_fixed_graph_testing2}
\Pbb_{f_0}\Bigl(\wc{T} \leq \wc{t}_{\alpha}\Bigr) \leq \frac{4\|\lambdavec(S)\|_4^4}{\bigl(f_0^{\top}S^2f_0 - \sqrt{2/\alpha} \norm{\lambdavec(S)}_4^2\bigr)^2} + \frac{8f_0^{\top}S^4f_0 }{\bigl(f_0^{\top}S^2f_0 - \sqrt{2/\alpha} \norm{\lambdavec(S)}_4^2\bigr)^2}.
\end{equation}
We now use the assumed lower bound $f_0^{\top} S^2 f_0 \geq (2\sqrt{2/\alpha} + 2b) \norm{\lambdavec(S)}_4^2$ to separately upper bound each of the two terms on the right hand side of~\eqref{pf:linear_smoother_fixed_graph_testing2}. It follows immediately that
\begin{equation}
\label{pf:linear_smoother_fixed_graph_testing2.5}
\frac{4 \|\lambdavec(S)\|_4^4}{\bigl(f_0^{\top}S^2f_0 - \sqrt{2/\alpha} \norm{\lambdavec(S)}_4^2\bigr)^2} \leq \frac{1}{b^2},
\end{equation}
giving a sufficient upper bound on the first term. Now we upper bound the second term,
\begin{equation}
\label{pf:linear_smoother_fixed_graph_testing3}
\frac{8f_0^{\top}S^4f_0 }{\bigl(f_0^{\top}S^2f_0 - \sqrt{2/\alpha} \norm{\lambdavec(S)}_4^2\bigr)^2} \leq \frac{32 f_0^{\top}S^4f_0 }{\bigl(f_0^{\top}S^2f_0\bigr)^2} \leq \frac{16}{b \|\lambdavec(S)\|_4^2} \frac{f_0^{\top}S^4f_0 }{f_0^{\top}S^2f_0} \leq \frac{16}{b \|\lambdavec(S)\|_4^2},
\end{equation}
where the final inequality is satisfied because $S$ is a contraction. Plugging~\eqref{pf:linear_smoother_fixed_graph_testing2.5} and~\eqref{pf:linear_smoother_fixed_graph_testing3} back into~\eqref{pf:linear_smoother_fixed_graph_testing2} then gives the desired result.\\
\subsection{Analysis of Laplacian smoothing}

Upper bounds on the mean squared error of $\wh{f}$, and Type I and Type II error of $\wh{T}$, follow from setting $S = (\rho L + I)^{-1}$ in Lemmas~\ref{lem:linear_smoother_fixed_graph_estimation} and~\ref{lem:linear_smoother_fixed_graph_testing}. We give these results in Lemma~\ref{lem:ls_fixed_graph_estimation} and~\ref{lem:ls_fixed_graph_testing}, and prove them immediately. Recall that $\lambda_1,\ldots,\lambda_n$ are the $n$ eigenvalues of $\Lap$ (sorted in ascending order).
\begin{lemma}
	\label{lem:ls_fixed_graph_estimation}
	For any $\rho > 0$,
	\begin{equation}
	\label{eqn:ls_fixed_graph_estimation_prob}
	\frac{1}{n}\bigl\|\wh{f} - f_0\bigr\|_2^2 \leq \frac{2\rho}{n} \bigl(f_0^{\top} \Lap f_0\bigr) + \frac{10}{n}\sum_{k = 1}^{n} \frac{1}{\bigl(\rho \lambda_{k} + 1\bigr)^2},
	\end{equation}
	with probability at least $1 - \exp\Bigl(-\sum_{k = 1}^{n}\bigl(\rho \lambda_k + 1\bigr)^{-2}\Bigr)$.
\end{lemma}
Recall that 
\begin{equation*}
\wh{t}_{\alpha} := \frac{1}{n}\sum_{k = 1}^{n} \frac{1}{\bigl(\rho \lambda_k + 1\bigr)^2} + \frac{1}{n}\sqrt{\frac{2}{\alpha}\sum_{k = 1}^{n} \frac{1}{\bigl(\rho \lambda_k + 1\bigr)^4}}.
\end{equation*}
\begin{lemma}
	\label{lem:ls_fixed_graph_testing}
	For any $\rho > 0$ and any $b \geq 1$, it holds that:
	\begin{itemize}
		\item \textbf{Type I error.}
		\begin{equation}
		\label{eqn:ls_fixed_graph_testing_typeI}
		\Pbb_0\Bigl(\wh{T} > \wh{t}_{\alpha}\Bigr) \leq \alpha.
		\end{equation}
		\item \textbf{Type II error.} If
		\begin{equation}
		\label{eqn:ls_fixed_graph_testing_critical_radius}
		\frac{1}{n}\norm{f_0}_2^2 \geq \frac{2 \rho}{n} \bigl(f_0^{\top} \Lap f_0\bigr) + \frac{2\sqrt{2/\alpha} + 2b}{n} \Biggl(\sum_{k = 1}^{n} \frac{1}{(\rho\lambda_k + 1)^4} \Biggr)^{1/2},
		\end{equation}
		then
		\begin{equation}
		\label{eqn:ls_fixed_graph_testing_typeII}
		\Pbb_{f_0}\Bigl(\wh{T}(G) \leq \wh{t}_{\alpha}\Bigr) \leq \frac{1}{b^2} + \frac{16}{b} \Biggl(\sum_{k = 1}^{n} \frac{1}{(\rho\lambda_k + 1)^4} \Biggr)^{-1/2}.
		\end{equation}
	\end{itemize}
\end{lemma}

\paragraph{Proof of Lemma~\ref{lem:ls_fixed_graph_estimation}.}
Let $\wh{S} = (\Id + \rho \Lap)^{-1}$, the estimator $\wh{f} = \wh{S}Y$, and
\begin{equation*}
\bigl\|\lambdavec(\wh{S})\bigr\|_2^2 = \sum_{k = 1}^{n} \frac{1}{\bigl(1 + \rho \lambda_k\bigr)^2}.
\end{equation*} 
We deduce the following upper bound on the bias term,
\begin{equation*}
\begin{aligned}
\bigl\|(\wh{S} - I)f_0\bigr\|_2^2 & = f_0^{\top} \Lap^{1/2} \Lap^{-1/2}\bigl(\wh{S} - \Id\bigr)^2 \Lap^{-1/2} \Lap^{1/2} f_0 \\
& \leq f_0^{\top} \Lap f_0 \cdot \lambda_n\Bigl(\Lap^{-1/2}\bigl(\wh{S} - \Id\bigr)^2 \Lap^{-1/2}\Bigr) \\
& = f_0^{\top} \Lap f_0 \cdot \max_{k \in [n]} \biggl\{\frac{1}{\lambda_k} \Bigl(1 - \frac{1}{\rho\lambda_k + 1}\Bigr)^2 \biggr\} \\
& \leq f_0^{\top} \Lap f_0 \cdot \rho.
\end{aligned}
\end{equation*} 
In the above, we have written $\Lap^{-1/2}$ for the square root of the pseudoinverse of $\Lap$, the maximum is over all indices $k$ such that $\lambda_k > 0$, and the last inequality follows from the basic algebraic identity $1 - 1/(1 + x)^2 \leq 2 x$ for any $x > 0$. The claim of the Lemma then follows from Lemma~\ref{lem:linear_smoother_fixed_graph_estimation}.

\paragraph{Proof of Lemma~\ref{lem:ls_fixed_graph_testing}.}
Let $\wh{S} := (I + \rho L)^{-1}$, so that $\wh{T} = \frac{1}{n}{\bf Y}^{\top} \wh{S}^2 {\bf Y}$. Note that $\wh{S}$ is a contraction, so that we may invoke Lemma~\ref{lem:linear_smoother_fixed_graph_testing}. The bound on Type I error \eqref{eqn:ls_fixed_graph_testing_typeI} follows immediately from \eqref{eqn:linear_smoother_fixed_graph_testing_typeI}. 
To establish the bound on Type II error, we must lower bound $f_0^{\top} \wh{S}^2 f_0$. We first note that by assumption \eqref{eqn:ls_fixed_graph_testing_critical_radius},
\begin{align*}
f_0^{\top} \wh{S}^2 f_0 & = \bigl\|f_0\bigr\|_2^2 - f_0^{\top}(I - \wh{S}^2)f_0 \\
& \geq 2 \rho \bigl(f_0^{\top} \Lap f_0\bigr) - f_0^{\top}\bigl(I - \wh{S}^2\bigr)f_0 +  \Bigl(2\sqrt{\frac{2}{\alpha}} + 2b\Bigr) \cdot \Biggl(\sum_{k = 1}^{n} \frac{1}{(\rho\lambda_k + 1)^4} \Biggr)^{-1/2}.
\end{align*}
Upper bounding $f_0^{\top}\bigl(I - \wh{S}^2\bigr)f_0$ as follows:
\begin{equation*}
\begin{aligned}
f_0^{\top} \Bigl(\Id - \wh{S}^2\Bigr) f_0  & = f_0^{\top} \Lap^{1/2} \Lap^{-1/2}\Bigl(\Id - \wh{S}^2\Bigr) \Lap^{-1/2} \Lap^{1/2} f_0 \\ 
& \leq f_0^{\top} \Lap f_0 \cdot  \lambda_{n}\biggl(\Lap^{-1/2}\Bigl(\Id - \wh{S}^2\Bigr) \Lap^{-1/2}\biggr) \\ 
& = f_0^{\top} \Lap f_0 \cdot \max_{k} \biggl\{ \frac{1}{\lambda_k} \Bigl(1 - \frac{1}{(\rho \lambda_k + 1)^2}\Bigr) \biggr\} \\
& \leq f_0^{\top} \Lap f_0 \cdot 2\rho,
\end{aligned}
\end{equation*}
---where in the above the maximum is over all indices $k$ such that $\lambda_k > 0$---we deduce that
\begin{equation*}
f_0^{\top} \wh{S}^2 f_0 \geq  \Bigl(2\sqrt{\frac{2}{\alpha}} + 2b\Bigr) \cdot \biggl(\sum_{k = 1}^{n} \frac{1}{(\rho\lambda_k + 1)^4} \biggr)^{-1/2}.
\end{equation*} 
The upper bound on Type II error \eqref{eqn:ls_fixed_graph_testing_typeII} then follows from Lemma~\ref{lem:linear_smoother_fixed_graph_testing}.

\section{Neighborhood graph Sobolev semi-norm}
\label{sec:graph_sobolev_seminorm}

In this section, we prove Lemma~\ref{lem:graph_sobolev_seminorm}, which states an upper bound on $f^{\top} L f$ that holds when $f$ is bounded in Sobolev norm. We also establish stronger bounds in the case when $f$ has a bounded Lipschitz constant; this latter result justifies one of our remarks after Theorem~\ref{thm:laplacian_smoothing_estimation1}. 

Throughout this proof, we will assume that $f \in H^1(\Xset)$ has zero-mean, meaning $\int_{\Xset} f(x) \,dx = 0$. This is without loss of generality---assuming for the moment that \eqref{eqn:graph_sobolev_seminorm} holds for zero-mean functions, for any $f \in H^1(\Xset)$, taking $a = \int_{\Xset} f(x) \,dx$ and $g = f - a$, we have that
\begin{equation*}
f^{\top} L f = g^{\top} L g \leq \frac{C_2}{\delta} n^2 r^{d + 2} |g|_{H^1(\mc{X})}^2 = \frac{C_2}{\delta} n^2 r^{d + 2} |f|_{H^1(\mc{X})}^2.
\end{equation*} 
Now, for any zero-mean function $f \in H^1(\mc{X})$ it follows by the Poincare inequality (see Section 5.8, Theorem 1 of \citet{evans10}) that $\|f\|_{H^1(\mc{X})}^2 \leq C_8 |f|_{H^1(\mc{X})}^2$, for some constant $C_8$ that does not depend on $f$. Therefore, to prove Lemma~\ref{lem:graph_sobolev_seminorm}, it suffices to show that
\begin{equation*}
\Ebb\Bigl[f^{\top} L f\Bigr] \leq C n^2 r^{d + 2} \|f\|_{H^1(\Xset)}^2,
\end{equation*}
since the high-probability upper bound then follows immediately by Markov's inequality. (Recall that $L$ is positive semi-definite, and therefore $f^{\top} L f$ is a non-negative random variable).

Since
\begin{equation*}
f^{\top} L f = \frac{1}{2}\sum_{i, j = 1}^{n} \bigl(f(X_i) - f(X_j)\bigr)^2 W_{ij},
\end{equation*}
it follows that
\begin{equation}
\label{pf:first_order_graph_sobolev_seminorm_1}
\Ebb\Bigl[f^{\top} L f\Bigr] = \frac{n(n - 1)}{2} \Ebb\biggl[\Bigl(f(X') - f(X)\Bigr)^2 K\biggl(\frac{\|X' - X\|}{r}\biggr)\biggr],
\end{equation}
where $X$ and $X'$ are random variables independently drawn from $P$.

Now, take $\Omega$ to be an arbitrary bounded open set such that $B(x,c_0) \subseteq \Omega$ for all $x \in \mc{X}$. For the remainder of this proof, we will assume that (i) $f \in H^1(\Omega)$ and additionally (ii) $\|f\|_{H^1(\Omega)} \leq C_5 \|f\|_{H^1(\mc{X})}$ for a constant $C_5$ that does not depend on $f$. This is without loss of generality, since by Theorem~1 in Chapter 5.4 of~\citet{evans10} there exists an extension operator $E: H^1(\mc{X}) \to H^1(\Omega)$ for which the extension $Ef$ satisfies both (i) and (ii). Additionally, we will assume $f \in C^{\infty}(\Omega)$. Again, this is without loss of generality, as $C^{\infty}(\Omega)$ is dense in $H^1(\Omega)$ and the expectation on the right hand side of \eqref{pf:first_order_graph_sobolev_seminorm_1} is continuous in $H^1(\Omega)$. The reason for dealing with a smooth extension $f \in C^{\infty}(\Omega)$ is so that we can make sense of the following equality for any $x$ and $x'$ in $\mc{X}$:
\begin{equation}
\label{pf:first_order_graph_sobolev_seminorm_1.5}
f(x') - f(x) = \int_{0}^{1} \nabla f\bigl(x + t(x' - x)\bigr)^{\top} (x' - x) \,dt.
\end{equation}

Obviously
\begin{equation}
\Ebb\biggl[\bigl(f(X') - f(X)\bigr)^2K\biggl(\frac{\|X' - X\|}{r}\biggr)\biggr] \leq p_{\max}^2 \int_{\Xset} \int_{\Xset} \bigl(f(x') - f(x)\bigr)^2 K\biggl(\frac{\|x' - x\|}{r}\biggr) \,dx' \,dx, \label{pf:first_order_graph_sobolev_seminorm_2}
\end{equation}
so that it remains now to bound the double integral. Replacing difference by integrated derivative as in \eqref{pf:first_order_graph_sobolev_seminorm_1.5}, we obtain
\begin{align}
\int_{\Xset} \int_{\Xset} \bigl(f(x') - f(x)\bigr)^2 K\biggl(\frac{\|x' - x\|}{r}\biggr) \,dx' \,dx & = \int_{\Xset} \int_{\Xset} \biggl[\int_{0}^{1} \nabla f\bigl(x + t(x' - x)\bigr)^{\top} (x' - x)\,dt\biggr]^2 K\biggl(\frac{\|x' - x\|}{r}\biggr) \,dx' \,dx \nonumber \\
& \overset{(i)}{\leq} \int_{\Xset} \int_{\Xset} \int_{0}^{1} \biggl[\nabla f\bigl(x + t(x' - x)\bigr)^{\top} (x' - x)\biggr]^2 K\biggl(\frac{\|x' - x\|}{r}\biggr) \,dt\,dx' \,dx \nonumber \\
& \overset{(ii)}{\leq} r^{d + 2} \int_{\Xset} \int_{B(0,1)} \int_{0}^{1} \biggl[\nabla f\bigl(x + trz\bigr)^{\top} z\biggr]^2 K\bigl(\|z\|\bigr) \,dt\,dz \,dx \nonumber \\
&  \overset{(iii)}{\leq} r^{d + 2} \int_{\Omega} \int_{B(0,1)} \int_{0}^{1} \Bigl[\nabla f\bigl(\wt{x}\bigr)^{\top} z\Bigr]^2 K\bigl(\|z\|\bigr) \,dt \,dz \,d\wt{x}, \label{pf:first_order_graph_sobolev_seminorm_3} 
\end{align}
where $(i)$ follows by Jensen's inequality, $(ii)$ follows by substituting $z = (x' - x)/r$ and~\ref{asmp:kernel}, and $(iii)$ by exchanging integrals, substituting $\wt{x} = x + trz$, and noting that $x \in \mc{X}$ implies that $\wt{x} \in \Omega$.

Now, writing $\bigl(\nabla f(\wt{x}) ^{\top} z\bigr)^2 = \bigl(\sum_{i = 1}^{d} z_{i} f^{(e_i)}(x) \bigr)^2$, expanding the square and integrating, we have that for any $\wt{x} \in \Xset$,
\begin{align*}
\int_{B(0,1)} \Bigl[\nabla f\bigl(\wt{x}\bigr)^{\top} z\Bigr]^2 K\bigl(\|z\|\bigr) \,dz & = \sum_{i,j = 1}^{d} f^{(e_i)}(\wt{x}) f^{(e_j)}(\wt{x}) \int_{\Rd} z_iz_jK(\|z\|) \,dz \\
& = \sum_{i = 1}^{d} \bigl(f^{(e_i)}(\wt{x})\bigr)^2 \int_{B(0,1)} z_i^2 K\bigl(\|z\|\bigr) \,dz \\
& = \sigma_K \|\nabla f(\wt{x})\|^2,
\end{align*}
where the last equality follows from the rotational symmetry of $K(\|z\|)$. Plugging back into \eqref{pf:first_order_graph_sobolev_seminorm_3}, we obtain
\begin{equation*}
\int_{\Xset} \int_{\Xset} \bigl(f(x') - f(x)\bigr)^2 K\biggl(\frac{\|x' - x\|}{r}\biggr) \,dx' \,dx \leq r^{d + 2} \sigma_K \|f\|_{H^1(\Omega)}^2 \leq C_5 r^{d + 2} \sigma_K \|f\|_{H^1(\mc{X})}^2,
\end{equation*}
proving the claim of Lemma~\ref{lem:graph_sobolev_seminorm} upon taking $C_2 := C_8C_5 \sigma_K p_{\max}^2$ in the statement of the lemma.

\subsection{Stronger bounds under Lipschitz assumption}

Suppose $f$ satisfies $|f(x') - f(x)| \leq M \|x - x\|$ for all $x,x' \in \mc{X}$. Then we can strengthen the high probability bound in Lemma~\ref{lem:graph_sobolev_seminorm} from $1 - \delta$ to $1 - \delta^2/n$, at the cost of only a constant factor in the upper bound on $f^{\top} L f$.
\begin{proposition}
	\label{prop:graph_sobolev_seminorm_lipschitz}
	Let $r \geq C_0(\log n/n)^{1/d}$. For any $f$ such that $|f(x') - f(x)| \leq M \|x - x\|$ for all $x,x' \in \mc{X}$, with probability at least $1 - C\delta^2/n$ it holds that
	\begin{equation*}
	f^{\top} L f \leq \biggl(\frac{1}{\delta} + C_2\biggr) n^2 r^{d + 2} M^2.
	\end{equation*}
\end{proposition}

\paragraph{Proof of Proposition~\ref{prop:graph_sobolev_seminorm_lipschitz}.}
We will prove Proposition~\ref{prop:graph_sobolev_seminorm_lipschitz} using Chebyshev's inequality, so the key step is to upper bound the variance of $f^{\top}Lf$. Putting $\varDelta_{ij} := K(\|X_i - X_j\|/r) \cdot (f(X_i) - f(X_j))^2$, we can write the variance of $f^{\top} L f$ as a sum of covariances,
\begin{equation*}
\Var\bigl[f^{\top} L f\bigr] = \frac{1}{4} \sum_{i, j = 1}^{n} \sum_{\ell,m = 1}^{n} \Cov\bigl[\varDelta_{ij},\varDelta_{\ell m}\bigr].
\end{equation*}
Clearly $\Cov\bigl[\varDelta_{ij},\varDelta_{\ell m}\bigr]$ depends on the cardinality of $I := \{i,j,k,\ell\}$; we divide into cases, and upper bound the covariance in each case.
\begin{enumerate}
	\item[$\bigl|I\bigr| = 4$.] In this case $\varDelta_{ij}$ and $\varDelta_{\ell m}$ are independent, and $\Cov\bigl[\varDelta_{ij},\varDelta_{\ell m}\bigr] = 0$.
	\item[$\bigl|I\bigr| = 3$.] Taking $i = \ell$ without loss of generality, and noting that the expectation of $\Delta_{ij}$ and $\Delta_{im}$ is non-negative, we have
	\begin{align*}
	\Cov\bigl[\varDelta_{ij},\varDelta_{i m}\bigr] & \leq \Ebb\bigl[\varDelta_{ij} \varDelta_{i m} \bigr] \\
	& = \int_{\Xset} \int_{\Xset} \int_{\Xset} \bigl(f(z) - f(x)\bigr)^2 \bigl(f(x') - f(x)\bigr)^2 K\biggl(\frac{\|x' - x\|}{r}\biggr) K\biggl(\frac{\|z - x\|}{r}\biggr) p(z) p(x') p(x) \,dz \,dx' \,dx \\
	& \leq p_{\max}^3 M^4 r^4 \int_{\Xset} \int_{\Xset} \int_{\Xset} K\biggl(\frac{\|x' - x\|}{r}\biggr) K\biggl(\frac{\|z - x\|}{r}\biggr) \,dz \,dx' \,dx \\
	& \leq p_{\max}^3 M^4 r^{4 + 2d}.
	\end{align*}
	\item[$\bigl|I\bigr| = 2$.] Taking $i = \ell$ and $j = m$ without loss of generality, we have
	\begin{align*}
	\Var\bigl[\varDelta_{ij}\bigr] & \leq \Ebb\bigl[\varDelta_{ij}^2\bigr] \\
	& \leq \int_{\Xset} \int_{\Xset} \bigl(f(x') - f(x)\bigr)^4 \biggl[K\biggl(\frac{\|x' - x\|}{r}\biggr)\biggr]^2 p(x') p(x) \,dx' \,dx \\
	& \leq p_{\max}^2 M^4 r^4 K(0) \int_{\Xset} \int_{\Xset} K\biggl(\frac{\|x' - x\|}{r}\biggr) \,dx' \,dx \\
	& \leq p_{\max}^2 M^4 r^{4 + d} K(0).
	\end{align*}
	\item[$\bigl|I\bigr| = 1$.] In this case $\varDelta_{i j} = \varDelta_{\ell m} = 0$.
\end{enumerate}
Therefore
\begin{equation*}
\Var\bigl[f^{\top} L f\bigr] \leq n^3 p_{\max}^3 M^4 r^{4 + 2d} + n^2 p_{\max}^2 M^4 r^{4 + d} K(0) \leq C M^4 n^3r^{4 + 2d},
\end{equation*}
where the latter inequality follows since $nr^d \gg 1$. For any $\delta > 0$, it follows from Chebyshev's inequality that
\begin{equation*}
\Pbb\biggl(\Bigl|f^{\top} L f - \Ebb\bigl[f^{\top} L f\bigr]\Bigr| \geq \frac{M^2}{\delta}n^2 r^{d + 2}\biggr) \leq C\frac{\delta^2}{n},
\end{equation*}
and since we have already upper bounded $\Ebb\bigl[f^{\top} L f\bigr] \leq C_2 M^2 n^2 r^{d + 2}$, the proposition follows. 

Note that the bound on $\Var[\varDelta_{i j}]$ follows as long as we can control $\|\nabla f\|_{\Leb^4(\Xset)}$; this implies the Lipschitz assumption---which gives us control of $\|\nabla f\|_{\Leb^{\infty}(\Xset)}$---can be weakened. However, the Sobolev assumption---which gives us control only over $\|\nabla f\|_{\Leb^2(\Xset)}$---will not do the job. 

\section{Bounds on neighborhood graph eigenvalues}
\label{sec:graph_eigenvalues}

In this section, we prove Lemma~\ref{lem:neighborhood_eigenvalue}, following the lead of~\citet{burago2014,trillos2019,calder2019}, who establish similar results with respect to a manifold without boundary. To prove this lemma, in  Theorem~\ref{thm:neighborhood_eigenvalue} we give estimates on the difference between eigenvalues of the graph Laplacian $L$ and eigenvalues of the weighted Laplace-Beltrami operator $\Delta_P$. We recall $\Delta_P$ is defined as
\begin{equation*}
\Delta_Pf(x) := -\frac{1}{p(x)} \dive\bigl(p^2\nabla f\bigr)(x).
\end{equation*}
To avoid confusion, in this section we write $\lambda_k(G_{n,r})$ for the $k$th smallest eigenvalue of the graph Laplacian matrix $L$ and $\lambda_k(\Delta_P)$ for the $k$th smallest eigenvalue of $\Delta_P$ \footnote{Under the assumptions~\ref{asmp:domain} and~\ref{asmp:density}, the operator $\Delta_P$ has a discrete spectrum; see \citet{garciatrillos18} for more details.}. Some other notation: throughout this section, we will write $A, A_0, A_1,\ldots$ and $a,a_0,a_1,\ldots$ for constants which may depend on $\Xset$, $d$, $K$, and $p$, but do not depend on $n$; we keep track of all such constants explicitly in our proofs. We let $L_K$ denote the Lipschitz constant of the kernel $K$. Finally, for notational ease we set $\theta$ and $\wt{\delta}$ to be the following (small) positive numbers:
\begin{equation}
\label{asmp:smallness}
\wt{\delta} := \max\Biggl\{n^{-1/d}, \min\biggl\{\frac{1}{2^{d + 3}A_0}, \frac{1}{A_3}, \frac{K(1)}{8L_KA_0}, \frac{1}{8\max\{A_1,A\}c_0}\biggr\}r\Biggr\},~~\textrm{and}~~
\theta := \frac{1}{8\max\{A_1,A\}}.
\end{equation} 
We note that each of $\wt{\delta}, \theta$ and $\wt{\delta}/r$ are of at most constant order. 

\begin{theorem}
	\label{thm:neighborhood_eigenvalue}
	For any $\ell \in \mathbb{N}$ such that
	\begin{equation}
	\label{eqn:neighborhood_eigenvalue_1}
	1 - A\Biggl(r \sqrt{\lambda_{\ell}(\Delta_P)} + \theta + \wt{\delta}\Biggr)\geq \frac{1}{2}
	\end{equation}
	with probability at least $1 - A_0n\exp(-a_0n\theta^2\wt{\delta}^{d})$, it holds that
	\begin{equation}
	\label{eqn:eigenvalue_bound}
	a \lambda_k(G_{n,r}) \leq nr^{d+2} \lambda_k(\Delta_P) \leq A \lambda_k(G_{n,r}),~~\textrm{for all $1 \leq k \leq \ell$}
	\end{equation}
\end{theorem}
Before moving forward to the proofs of Lemma~\ref{lem:neighborhood_eigenvalue} and Theorem~\ref{thm:neighborhood_eigenvalue}, it is worth being clear about the differences between Theorem~\ref{thm:neighborhood_eigenvalue} and the results of \citet{burago2014,trillos2019,calder2019}. First of all, the reason we cannot directly use the results of these works in the proof of Lemma~\ref{lem:neighborhood_eigenvalue} is that they all assume the domain $\mc{X}$ is without boundary, whereas for our results in Section~\ref{sec:minimax_optimal_laplacian_smoothing} we instead assume $\mc{X}$ has a (Lipschitz smooth) boundary. Fortunately, in this setting the high-level strategy shared by \citet{burago2014,trillos2019,calder2019} can still be used---indeed we follow it closely, as we summarize in Section~\ref{subsec:neighbrhood_eigenvalue_pf}. However, many calculations need to be redone, in order to account for points $x$ which are on or sufficiently close to the boundary of $\mc{X}$. For completeness and ease of reading, we provide a self-contained proof of Theorem~\ref{thm:neighborhood_eigenvalue}, but we comment where appropriate on connections between the technical results we use in this proof, and those derived in \citet{burago2014,trillos2019,calder2019}.

On the other hand, we should also point out that unlike the results of~\citet{burago2014,trillos2019,calder2019}, Theorem~\ref{thm:neighborhood_eigenvalue} does not imply that $\lambda_{k}(G_{n,r})$ is a consistent estimate of $\lambda_k(\Delta_P)$, i.e. it does not imply that $|(nr^{d + 2})^{-1}\lambda_{k}(G_{n,r}) - \lambda_k(\Delta_P)| \to 0$ as $n \to \infty, r \to 0$. The key difficulty in proving consistency when $\mc{X}$ has a boundary can be summarized as follows: while at points $x \in \mc{X}$ satisfying $B(x,r) \subseteq \mc{X}$, the graph Laplacian $L$ is a reasonable approximation of the operator $\Delta_P$, at points $x$ near the boundary $L$ is known to approximate a different operator altogether \citep{belkin2012}. This is reminiscent of the boundary effects present in the analysis of kernel smoothing. We believe a more subtle analysis might imply convergence of eigenvalues in this setting. However, the conclusion of Theorem~\ref{thm:neighborhood_eigenvalue}---that $\lambda_k(G_{n,r})/(nr^{d + 2}\lambda_k(\Delta_P))$ is bounded above and below by constants that do not depend on $k$---suffices for our purposes.  

The bulk of the remainder of this section is devoted to the proof of Theorem~\ref{thm:neighborhood_eigenvalue}. First, however, we show that under our regularity conditions on $p$ and $\Xset$, Lemma~\ref{lem:neighborhood_eigenvalue} is a simple consequence of Theorem~\ref{thm:neighborhood_eigenvalue}. The link between the two is Weyl's Law.
\begin{proposition}[Weyl's Law]
	\label{prop:weyl}
	Suppose the density $p$ and the domain $\Xset$ satisfy~\ref{asmp:domain} and~\ref{asmp:density}. Then there exist constants $a_2$ and $A_2$ such that
	\begin{equation}
	\label{eqn:weyls_law}
	a_2k^{2/d} \leq \lambda_k(\Delta_P) \leq A_2k^{2/d}~~\textrm{for all $k \in \mathbb{N}, k > 1$}.
	\end{equation}
\end{proposition}
See Lemma 28 of~\citet{dunlop2020} for a proof that~\ref{asmp:domain} and~\ref{asmp:density} imply Weyl's Law.

\paragraph{Proof of Lemma~\ref{lem:neighborhood_eigenvalue}.}
Put
\begin{equation*}
\ell_{\star} = \floor{\biggl(\frac{\bigl(1/(2A) - (\theta + \wt{\delta})\bigr)}{rA_2^{1/2}}\biggr)^d}.
\end{equation*}
Let us verify that $\lambda_{\ell_{\star}}(\Delta_P)$ satisfies the condition~\eqref{eqn:neighborhood_eigenvalue_1} of Theorem~\ref{thm:neighborhood_eigenvalue}. Setting $c_0 := 1/(2^{1/d}4A_2^{1/2})$, the assumed upper bound on the radius $r \leq c_0$ guarantees that $\ell_{\star} \geq 2$. Therefore, by Proposition~\ref{prop:weyl} we have that
\begin{align*}
\sqrt{\lambda_{\ell_{\star}}(\Delta_P)} \leq A_2^{1/2}\ell_{\star}^{1/d} \leq  \frac{1}{r}\biggl(\frac{1}{2A} - (\theta + \wt{\delta})\biggr).
\end{align*}
Rearranging the above inequality shows that condition \eqref{eqn:neighborhood_eigenvalue_1} is satisfied. 

It is therefore the case that the inequalities in \eqref{eqn:eigenvalue_bound} hold with probability at least $1 - A_0n\exp(-a_0n\theta^2\wt{\delta}^{d})$. Together, \eqref{eqn:eigenvalue_bound} and \eqref{eqn:weyls_law} imply the following bounds on the graph Laplacian eigenvalues:
\begin{equation*}
\frac{a}{A_2} nr^{d + 2} k^{2/d} \leq \lambda_k(G_{n,r}) \leq \frac{A}{a_2} nr^{d + 2} k^{2/d}~~\textrm{for all $2 \leq k \leq \ell_{\star}$}.
\end{equation*}
It remains to bound $\lambda_k(G_{n,r})$ for those indices $k$ which are greater than $\ell_{\star}$. On the one hand, since the eigenvalues are sorted in ascending order, we can use the lower bound on $\lambda_{\ell_{\star}}(G_{n,r})$ that we have just derived:
\begin{equation*}
\lambda_k(G_{n,r}) \geq \lambda_{\ell_{\star}}(G_{n,r}) \geq \frac{a_2}{A}nr^{d + 2}\ell_{\star}^{2/d} \geq \frac{a_2}{64A^3 A_2} nr^{d}.
\end{equation*}
On the other hand, for any graph $G$ the maximum eigenvalue of the Laplacian is upper bounded by twice the maximum degree \citep{chung97}. Writing $D_{\max}(G_{n,r})$ for the maximum degree of $G_{n,r}$, it is thus a consequence of Lemma~\ref{lem:max_degree} that
\begin{equation*}
\lambda_k(G_{n,r}) \leq 2D_{\max}(G_{n,r}) \leq 4p_{\max}nr^d,
\end{equation*}
with probability at least $1 - 2n\exp\Bigl(-nr^dp_{\min}/(3K(0)^2)\Bigr)$. In sum, we have shown that with probability at least $1 - A_0n\exp(-a_0n\theta^2\wt{\delta}^{d}) - 2n\exp\Bigl(-nr^dp_{\min}/(3K(0)^2)\Bigr)$,
\begin{equation*}
\min\biggl\{\frac{a_2}{A}nr^{d + 2}k^{2/d}, \frac{a_2}{A^3 64 A_3} nr^d\biggr\} \leq \lambda_k(G_{n,r}) \leq \min\biggl\{\frac{A_2}{a}nr^{2 + d}k^{2/d}, 4p_{\max}nr^d\biggr\}~~\textrm{for all $2 \leq k \leq n$.}
\end{equation*}
Lemma~\ref{lem:neighborhood_eigenvalue} then follows upon setting
\begin{align*}
C_1 & := \max\{2A_0,4\},~~ && c_1 := \min\Biggl\{\frac{p_{\min}}{3K(0)^2}, \frac{\theta^2\wt{\delta}}{r}\Biggr\} \\
C_3 & := \max\biggl\{\frac{A_2}{a}, 4p_{\max}\biggr\},~~ && c_3 := \min\biggl\{\frac{a_2}{A}, \frac{a_2}{A^3 64 A_3} \biggr\}.
\end{align*}
in the statement of that Lemma.

\subsection{Proof of Theorem~\ref{thm:neighborhood_eigenvalue}}
\label{subsec:neighbrhood_eigenvalue_pf}

In this section we prove Theorem~\ref{thm:neighborhood_eigenvalue}, following closely the approach of \citet{burago2014,trillos2019,calder2019}. As in these works, we relate $\lambda_k(\Delta_P)$ and $\lambda_k(G_{n,r})$ by means of the Dirichlet energies
\begin{equation*}
b_r(u) := \frac{1}{n^2 r^{d+ 2}}u^{\top} L u 
\end{equation*}
and
\begin{equation*}
D_2(f) :=
\begin{cases*}
\int_{\Xset} \|\nabla f(x)\|^2 p^2(x) \,dx~~ &\textrm{if $f \in H^1(\Xset)$} \\
\infty~~ & \textrm{otherwise,}
\end{cases*}
\end{equation*}
Let us pause briefly to motivate the relevance of $b_r(u)$ and $D_2(f)$. In the following discussion, recall that for a function $u: \{X_1,\ldots,X_n\} \to \Reals$, the empirical norm is defined as $\|u\|_n^2 := \frac{1}{n} \sum_{i = 1}^{n} (u(X_i))^2$, and the class $\Leb^2(P_n)$ consists of those $u \in \Reals^n$ for which $\|u\|_{n} < \infty$. Similarly, for a function $f: \Xset \to \Reals$, the $L^2(P)$ norm of $f$ is
\begin{equation*}
\|f\|_{P}^2 := \int_{\Xset} \bigl|f(x)\bigr|^2 p(x) \,dx,
\end{equation*} 
and the class $\Leb^2(P)$ consists of those $f$ for which $\|f\|_P < \infty$. Now, suppose one could show the following two results: 
\begin{enumerate}[(1)]
	\item an upper bound of $b_r(u)$ by $D_2\bigl(\mc{I}(u)\bigr)$ for an appropriate choice of interpolating map $\mc{I}: \Leb^2(P_n) \to \Leb^2(\mc{X})$, and vice versa an upper bound of $D_2(f)$ by $b_r(\mc{P}(f))$ for an appropriate choice of discretization map $\mc{P}: \Leb^2(\mc{X}) \to \Leb^2(P_n)$,
	\item that $\mc{I}$ and $\mc{P}$ were near-isometries, meaning $\|\mc{I}(u)\|_{P} \approx \|u\|_{n}$ and $\|\mc{P}(f)\|_{P} \approx \|f\|_{n}$.
\end{enumerate}
Then, by using the variational characterization of eigenvalues $\lambda_k(\Delta_P)$ and $\lambda_k(G_{n,r})$---i.e. the Courant-Fischer Theorem---one could obtain estimates on the error $\bigl|nr^{d + 2}\lambda_k(\Delta_P) - \lambda_k(G_{n,r})\bigr|$.

We will momentarily define particular maps $\wt{\mc{I}}$ and $\wt{\mc{P}}$, and establish that they satisfy both (1) and (2). In order to define these maps, we must first introduce a particular probability measure $\wt{P}_n$ that, with high probability, is close in transportation distance to both $P_n$ and $P$. This estimate on the transportation distance---which we now give---will be the workhorse that allows us to relate $b_r$ to $D_2$, and $\|\cdot\|_n$ to $\|\cdot\|_P$.

\paragraph{Transportation distance between $P_n$ and $P$.}
For a measure $\mu$ defined on $\Xset$ and map $T: \Xset \to \Xset$, let $T_{\sharp}\mu$ denote the \emph{push-forward} of $\mu$ by $T$, i.e the measure for which
\begin{equation*}
\bigl(T_{\sharp}\mu\bigr)(U) := \mu\bigl(T^{-1}(U)\bigr)
\end{equation*}
for any Borel subset $U \subseteq \Xset$. Suppose $T_{\sharp}\mu = P_n$; then the map $T$ is referred to as transportation map between $\mu$ and $P_n$. The  $\infty$-transportation distance between $\mu$ and $P_n$ is then
\begin{equation}
\label{eqn:optimal_transport}
d_{\infty}(\mu,P_n) := \inf_{T: T_{\sharp} \mu = P_n} \|T - \mathrm{Id}\|_{L^{\infty}(\mu)}
\end{equation}
where $\mathrm{Id}(x) = x$ is the identity mapping.

\citet{calder2019} take $\Xset$ to be a smooth submanifold of $\Rd$ without boundary, i.e. they assume $\Xset$ satisfies~\ref{asmp:domain_manifold}. In this setting, they exhibit an absolutely continuous measure $\wt{P}_n$ with density $\wt{p}_n$ that with high probability is close to $P_n$ in transportation distance, and for which $\|p - \wt{p}_n\|_{\Leb^\infty}$ is also small. In Proposition~\ref{prop:optimal_transport}, we adapt this result to the setting of full-dimensional manifolds with boundary.  
\begin{proposition}
	\label{prop:optimal_transport}
	Suppose $\Xset$ satisfies~\ref{asmp:domain} and $p$ satisfies~\ref{asmp:density}. Then with probability at least $1 - A_0 n \exp\bigl\{-a_0 n\theta^2\wt{\delta}^d\bigr\}$, the following statement holds: there exists a probability measure $\wt{P}_n$ with density $\wt{p}_n$ such that:
	\begin{equation}
	\label{eqn:optimal_transport_1}
	d_{\infty}(\wt{P}_n, P_n) \leq A_0 \wt{\delta}
	\end{equation}
	and
	\begin{equation}
	\label{eqn:optimal_transport_2}
	\|\wt{p}_n - p\|_{\infty} \leq A_0\bigl(\wt{\delta} + \theta\bigr).
	\end{equation}
\end{proposition}
For the rest of this section, we let $\wt{P}_n$ be a probability measure with density $\wt{p}_n$, that satisfies the conclusions of Proposition~\ref{prop:optimal_transport}. Additionally we denote by $\wt{T}_n$ an \emph{optimal transport map} between $\wt{P}_n$ and $P_n$, meaning a transportation map which achieves the infimum in \eqref{eqn:optimal_transport}. Finally, we write $U_1,\ldots,U_n$ for the preimages of $X_1,\ldots,X_n$ under $\wt{T}_n$, meaning $U_i = \wt{T}_n^{-1}(X_i)$. 

\paragraph{Interpolation and discretization maps.}
The discretization map  $\wt{\mathcal{P}}: \Leb^2(\Xset) \to \Leb^2(P_n)$ is given by averaging over the cells $U_1,\ldots,U_n$, 
\begin{equation*}
\bigl(\wt{\mathcal{P}}f\bigr)(X_i) := n \cdot \int_{U_i} f(x) \wt{p}_n(x) \,dx.
\end{equation*}
On the other hand, the interpolation map $\wt{\mc{I}}: \Leb^2(P_n) \to \Leb^2(\Xset)$ is defined as $\wt{\mc{I}}u := \Lambda_{r - 2A_0\wt{\delta}}(\wt{\mc{P}}^{\star}u)$. Here, $\wt{\mc{P}}^{\star} = u \circ \wt{T}$ is the adjoint of $\wt{\mc{P}}$, i.e.
\begin{equation*}
\bigl(\wt{\mc{P}}^{\star}u\bigr)(x) = \sum_{j = 1}^{n} u(x_i) \1\{x \in U_i\} 
\end{equation*} 
and $\Lambda_{r - 2A_0\wt{\delta}}$ is a kernel smoothing operator, defined with respect to a carefully chosen kernel $\psi$. To be precise, for any $h > 0$,
\begin{equation*}
\Lambda_h(f) := \frac{1}{h^d\tau_h(x)}\int_{\Xset} \eta_h(x',x) f(x') \,dx',~~ \eta_h(x',x) := \psi\biggl(\frac{\|x' - x\|}{r}\biggr)
\end{equation*}
where $\psi(t) := (1/\sigma_K)\int_{t}^{\infty} s K(s) \,ds$ and $\tau_h(x) := (1/h^d)\int_{\Xset} \eta_h(x',x) \,dx'$ is a normalizing constant.

Propositions~\ref{prop:dirichlet_energies} and~\ref{prop:isometry} establish our claims regarding $\wt{\mc{P}}$ and $\wt{\mc{I}}$: first, that they approximately preserve the Dirichlet energies $b_r$ and $D_2$, and second that they are near-isometries for functions $u \in \Leb^2(P_n)$ (or $f \in \Leb^2(P)$) of small Dirichlet energy $b_r(u)$ (or $D_2(f)$).

\begin{proposition}[\textbf{cf. Proposition 4.1 of \citet{calder2019}}]
	\label{prop:dirichlet_energies}
	With probability at least $1 - A_0n\exp(-a_0n\theta^2\wt{\delta}^{d})$, we have the following.
	\begin{enumerate}[(1)]
		\item For every $u \in \Leb^2(P_n)$,
		\begin{equation}
		\label{eqn:dirichlet_energies_1}
		\sigma_{K} D_2(\wt{\mc{I}}u) \leq A_8 \Bigl(1 + A_1(\theta + \wt{\delta})\Bigr) \cdot \biggl(1 + A_3\frac{\wt{\delta}}{r}\biggr) b_r(u).
		\end{equation}
		\item For every $f \in \Leb^2(\Xset)$,
		\begin{equation}
		\label{eqn:dirichlet_energies_2}
		b_r(\wt{\mc{P}}f) \leq \Bigl(1 + A_1(\theta + \wt{\delta})\Bigr) \cdot \Bigl(1 + A_9\frac{\wt{\delta}}{r}\Bigr) \cdot \Bigl(\frac{C_5 p_{\max}^2}{p_{\min}^2}\Bigr) \cdot \sigma_{K} D_2(f).
		\end{equation}
	\end{enumerate}
\end{proposition}

\begin{proposition}[\textbf{cf. Proposition 4.2 of \citet{calder2019}}]
	\label{prop:isometry}
	With probability at least $1 - A_0n\exp(-a_0n\theta^2\wt{\delta}^{d})$, we have the following.
	\begin{enumerate}[(1)]
		\item For every $f \in \Leb^2(\Xset)$,
		\begin{equation}
		\label{eqn:isometry_1}
		\Bigl|\|f\|_{P}^2 - \|\wt{\mc{P}}f\|_{n}^2\Bigr| \leq A_5 r \|f\|_{P} \sqrt{D_2(f)} + A_1\bigl(\theta + \wt{\delta}\bigr)\|f\|_{P}^2.
		\end{equation}
		\item For every $u \in \Leb^2(P_n)$,
		\begin{equation}
		\label{eqn:isometry_2}
		\Bigl|\|\wt{\mc{I}}u\|_{P}^2 - \|u\|_{n}^2\Bigr| \leq A_6 r \|u\|_{n} \sqrt{b_r(u)} + A_7\bigl(\theta + \wt{\delta}\bigr) \|u\|_{n}^2.
		\end{equation}
	\end{enumerate}
\end{proposition}

We will devote most of the rest of this section to the proofs of Propositions~\ref{prop:optimal_transport},~\ref{prop:dirichlet_energies}, and~\ref{prop:isometry}. First, however, we use these propositions to prove Theorem~\ref{thm:neighborhood_eigenvalue}.

\paragraph{Proof of Theorem~\ref{thm:neighborhood_eigenvalue}.}
Throughout this proof, we assume that inequalities \eqref{eqn:dirichlet_energies_1}-\eqref{eqn:isometry_2} are satisfied. We take $A$ and $a$ to be positive constants such that
\begin{align*}
\frac{1}{a} & \geq 2\Bigl(1 + A_1(\theta + \wt{\delta})\Bigr)\Bigl(1 + A_9\frac{\wt{\delta}}{r}\Bigr)\Bigl(\frac{C_5 p_{\max}^2}{p_{\min}^2}\Bigr) ,~~\textrm{and}~~A \geq \max\Biggl\{A_1,A_5, \frac{1}{\sqrt{a}}A_6, A_7, 2A_8\biggl(1 + A_1(\theta + \wt{\delta})\biggr) \biggl(1 + A_3\frac{\wt{\delta}}{r}\biggr) \Biggr\}.
\end{align*}

Let $k$ be any number in $1,\ldots,\ell$. We start with the upper bound in \eqref{eqn:eigenvalue_bound}, proceeding as in Proposition 4.4 of \citet{burago2014}. Let $f_1,\ldots,f_{k}$ denote the first $k$ eigenfunctions of $\Delta_P$ and set $W := \mathrm{span}\{f_1,\ldots,f_k\}$, so that by the Courant-Fischer principle $D_2(f) \leq \lambda_k(\Delta_P) \|f\|_{P}^2$ for every $f \in W$. As a result, by Part (1) of Proposition~\ref{prop:isometry} we have that for any $f \in W$,
\begin{equation*}
\bigl\|\wt{\mc{P}}f\bigr\|_{n}^2 \geq \biggl(1 - A_5r \sqrt{\lambda_{k}(\Delta_P)} - A_1(\theta + \wt{\delta})\biggr)\|f\|_{P}^2  \geq \frac{1}{2} \|f\|_{P}^2,
\end{equation*}
where the second inequality follows by assumption \eqref{eqn:neighborhood_eigenvalue_1}. 

Therefore $\wt{\mc{P}}$ is injective over $W$, and $\wt{\mc{P}}W$ has dimension $\ell$. This means we can invoke the Courant-Fischer Theorem, along with Proposition \ref{prop:dirichlet_energies}, and conclude that
\begin{align*}
\frac{\lambda_k(G_{n,r})}{nr^{d + 2}} & \leq \max_{\substack{u \in \wt{\mc{P}}W \\ u \neq 0} } \frac{b_r(u)}{\|u\|_{n}^2} \\
& = \max_{\substack{f \in W \\ f \neq 0} } \frac{b_r(\wt{\mc{P}}f)}{\bigl\|\wt{\mc{P}}f\bigr\|_{n}^2} \\
& \leq 2\Bigl(1 + A_1(\theta + \wt{\delta})\Bigr) \cdot \Bigl(1 + A_9\frac{\wt{\delta}}{r}\Bigr) \cdot \Bigl(\frac{C_5 p_{\max}^2}{p_{\min}^2}\Bigr) \sigma_K \lambda_{k}(\Delta_P),
\end{align*}
establishing the lower bound in \eqref{eqn:eigenvalue_bound}.

The upper bound follows from essentially parallel reasoning. Recalling that $v_1,\ldots,v_k$ denote the first $k$ eigenvectors of $L$, set $U := \mathrm{span}\{v_1,\ldots,v_k\}$, so that $nr^{d + 2} b_r(u) \leq \lambda_k(G_{n,r}) \|u\|_n^2$. By Proposition~\ref{prop:isometry}, Part (2), we have that for every $u \in U$,
\begin{align*}
\bigl\|\wt{\mc{I}}u\bigr\|_{P}^2 & \geq \|u\|_n^2 - A_6 r \|u\|_n \sqrt{b_r(u)} - A_7\bigl(\theta + \wt{\delta}\bigr)\|u\|_n^2 \\
& \geq \|u\|_n^2 - A_6 r \|u\|_n^2 \sqrt{\frac{\lambda_{k}(G_{n,r})}{nr^{d + 2}}} - A_7\bigl(\theta + \wt{\delta}\bigr)\|u\|_n^2 \\
& \geq \|u\|_n^2 - A_6 r \|u\|_n^2 \sqrt{\frac{1}{a}\lambda_k(\Delta_P)} - A_7\bigl(\theta + \wt{\delta}\bigr)\|u\|_n^2 \\
& \geq \frac{1}{2}\|u\|_n^2,
\end{align*}
where the second to last inequality follows from the lower bound $a \lambda_k(G_{n,r}) \leq nr^{d + 2}\lambda_k(\Delta_P)$ that we just derived, and the last inequality from assumption \eqref{eqn:neighborhood_eigenvalue_1}.

Therefore $\wt{\mc{I}}$ is injective over $U$, $\wt{\mc{I}}U$ has dimension $k$, and by Proposition~\ref{prop:dirichlet_energies} we conclude that
\begin{align*}
\lambda_k(\Delta_P) & \leq \max_{u \in U} \frac{D_2(\wt{\mc{I}}u)}{\|u\|_P^2} \\
& \leq 2A_8\biggl(1 + A_1(\theta + \wt{\delta})\biggr) \biggl(1 + A_3 \frac{\wt{\delta}}{r}\biggr)\max_{u \in U} \frac{b_r(u)}{\|u\|_n^2} \\
& \leq 2A_8\biggl(1 + A_1(\theta + \wt{\delta})\biggr) \biggl(1 + A_3\frac{\wt{\delta}}{r}\biggr) \frac{\lambda_k(G_{n,r})}{nr^{d + 2}},
\end{align*}
establishing the upper bound in \eqref{eqn:eigenvalue_bound}.

\paragraph{Organization of this section.}
The rest of this section will be devoted to proving Propositions~\ref{prop:optimal_transport},~\ref{prop:dirichlet_energies} and~\ref{prop:isometry}. To prove the latter two propositions, it will help to introduce the intermediate energies
\begin{equation*}
\wt{E}_r(f,\eta,V) := \frac{1}{r^{d + 2}}\int_{V} \int_{\Xset} \bigl(f(x') - f(x)\bigr)^2 \eta\biggl(\frac{\|x' - x\|}{r}\biggr) \wt{p}_n(x') \wt{p}_n(x) \,dx' \,dx
\end{equation*}
and
\begin{equation*}
{E}_r(f,\eta,V) := \frac{1}{r^{d + 2}}\int_{V} \int_{\Xset} \bigl(f(x') - f(x)\bigr)^2 \eta\biggl(\frac{\|x' - x\|}{r}\biggr) p(x') p(x) \,dx' \,dx.
\end{equation*}
Here $\eta: [0,\infty) \to [0,\infty)$ is an arbitrary kernel, and $V \subseteq \Xset$ is a measurable set. We will abbreviate $\wt{E}_r(f,\eta,\Xset)$ as $\wt{E}_r(f,\eta)$ and $\wt{E}_r(f,K) = \wt{E}_r(f)$ (and likewise with $E_r$.)

The proof of Proposition~\ref{prop:optimal_transport} is given in Section~\ref{subsec:proof_proposition_optimal_transport}. In Section~\ref{subsec:integrals}, we establish relationships between the (non-random) functionals $E_r(f)$ and $D_2(f)$, as well as providing estimates on some assorted integrals. In Section~\ref{subsec:random_functionals}, we establish relationships between the stochastic functionals $\wt{E}_r(f)$ and $E_r(f)$,  between $\wt{E}_r\bigl(\wt{\mc{I}}(u)\bigr)$ and $b_r\bigl(u\bigr)$, and between $\wt{E}_r\bigl(f\bigr)$ and $b_r\bigl(\wt{\mc{P}}f\bigr)$. Finally, in Section~\ref{subsec:proof_of_prop_dirichlet_energies_and_isometry} we use these various relationships to prove Propositions~\ref{prop:dirichlet_energies} and~\ref{prop:isometry}.

\subsection{Proof of Proposition~\ref{prop:optimal_transport}}
\label{subsec:proof_proposition_optimal_transport}

We start by defining the density $\wt{p}_n$, which will be piecewise constant over a particular partition $\mc{Q}$ of $\Xset$. Specifically, for each $Q$ in $\mc{Q}$ and every $x \in Q$, we set
\begin{equation}
\label{pf:prop_optimal_transport_0}
\wt{p}_n(x) := \frac{P_n(Q)}{\vol(Q)},
\end{equation}
where $\vol(\cdot)$ denotes the Lebesgue measure. Then $\wt{P}_n(U) = \int_{U} \wt{p}_n(x) \dx$.

We now construct the partition $\mc{Q}$, in progressive degrees of generality on the domain $\Xset.$
\begin{itemize}
	\item In the special case of the unit cube $\Xset = (0,1)^d$, the partition will simply be a collection of cubes,
	\begin{equation*}
	\mc{Q} = \set{Q_k: k \in [\wt{\delta}^{-1}]^d},
	\end{equation*}
	where $Q_k = \wt{\delta}\Bigl([k_1 - 1,k_1] \otimes \cdots \otimes [k_d - 1,k_d]\Bigr)$ and we assume without loss of generality that $\wt{\delta}^{-1} \in \mathbb{N}$.
	\item If $\mc{\Xset}$ is an open, connected set with smooth boundary, then by Proposition 3.2 of \citet{trillos2015}, there exist a finite number $N(\Xset) \in \mathbb{N}$ of disjoint polytopes which cover $\Xset$. Moreover, letting $U_j$ denote the intersection of the $j$th of these polytopes with $\wb{\Xset}$, this proposition establishes that for each $j$ there exists a bi-Lipschitz homeomorphism $\Phi_j: U_j \to [0,1]^d$. We take the collection
	\begin{equation*}
	\mc{Q} = \biggl\{\Phi_j^{-1}(Q_k): j = 1,\ldots,N(\Xset)~~\textrm{and}~~k \in [\wt{\delta}^{-1}]^d\biggr\}
	\end{equation*}
	to be our partition. Denote by $L_{\Phi}$ the maximum of the bi-Lipschitz constants of $\Phi_1,\ldots,\Phi_{N(\Xset)}$.
	\item Finally, in the general case where $\Xset$ is an open, connected set with Lipschitz boundary, then there exists a bi-Lipschitz homeomorphism $\Psi$ between $\Xset$ and a smooth, open, connected set with Lipschitz boundary. Letting $\Phi_j$ and $\wt{Q}_{j,k}$ be as before, we take the collection
	\begin{equation*}
	\mc{Q} = \biggl\{\wt{Q}_{j,k} = \Bigl(\Psi^{-1} \circ \Phi_j^{-1}\Bigr)(Q_k): j = 1,\ldots,N(\Xset)~~\textrm{and}~~k \in [\wt{\delta}^{-1}]^d\biggr\}
	\end{equation*}
	to be our partition. Denote by $L_{\Psi}$ the bi-Lipschitz constant of $\Psi$.
\end{itemize}
Let us record a few facts which hold for all $\wt{Q}_{j,k} \in \mc{Q}$, and which follow from the bi-Lipschitz properties of $\Phi_j$ and $\Psi$: first that
\begin{equation}
\label{pf:prop_optimal_transport_1}
\diam(\wt{Q}_{j,k}) \leq L_{\Psi} \L_{\Phi} \wt{\delta},
\end{equation}
and second that
\begin{equation}
\label{pf:prop_optimal_transport_2}
\vol(\wt{Q}_{j,k}) \geq \biggl(\frac{1}{L_{\Psi} L_{\Phi}}\biggr)^d \wt{\delta}^d.
\end{equation}
We now use these facts to show that $\wt{P}_n$ satisfies the claims of Proposition~\ref{prop:optimal_transport}. On the one hand for every $Q \in \mc{Q}$, letting $N(Q)$ denote the number of design points $\{X_1,\ldots,X_k\}$ which fall in $Q$, we have
\begin{equation*}
\wt{P}_n(Q) = \int_{Q} \wt{p}_n(x) \,dx = P_n(Q) = \frac{N(Q)}{n}.
\end{equation*}
Moreover, ignoring those cells for which $N(Q) = 0$ (since $\wt{P}_n(Q) = 0$ for such $Q$, and so they do not contribute to the essential supremum in \eqref{eqn:optimal_transport}), appropriately dividing each remaining cell $Q \in \mc{Q}$ into $N(Q)$ subsets $S_1,\ldots,S_{N(Q)}$ of equal volume, and mapping each $S_{\ell}$ to a different design point $X_i \in Q$, we can exhibit a transport map $T$ from $\wt{P}_n$ to $P_n$ for which
\begin{equation*}
\|T - \mathrm{Id}\|_{L^{\infty}(\wt{P}_n)} \leq \max_{Q \in \mc{Q}} \diam(Q) \leq   L_{\Psi} \L_{\Phi} \wt{\delta}.
\end{equation*}
On the other hand, applying the triangle inequality we have that for $x \in \wt{Q}_{j,k}$
\begin{align*}
|\wt{p}_n(x) - p(x)| \leq \biggl|\frac{P_n(\wt{Q}_{j,k}) - P(\wt{Q}_{j,k})}{\vol(\wt{Q}_{j,k})}\biggr| + \frac{1}{\vol(\wt{Q}_{j,k})} \int_{\wt{Q}_{j,k}} |p(x') - p(x)| \,dx,
\end{align*}
and using the Lipschitz property of $p$ we find that 
\begin{equation}
\label{pf:prop_optimal_transport_3}
\|\wt{p}_n - p\|_{\Leb^{\infty}} \leq \max_{j,k} \biggl|\frac{P_n(\wt{Q}_{j,k}) - P(\wt{Q}_{j,k})}{\vol(\wt{Q}_{j,k})}\biggr| + L_p L_{\Phi} L_{\Psi} \wt{\delta}.
\end{equation}
From Hoeffding's inequality and a union bound, we obtain that 
\begin{align*}
\mathbb{P}\biggl( \bigl|P_n(\wt{Q}) - P(\wt{Q})\bigr| & \leq \theta P(\wt{Q}) \quad \forall \wt{Q} \in \mc{Q} \biggr) \geq 1 - 2 \sharp(\mc{Q}) \cdot \exp\biggl\{-\frac{\theta^2 n \min \{P(\wt{Q})\}}{3}\biggr\} \\
& \geq 1 - \frac{2 N(\Xset)}{\wt{\delta}^d} \cdot \exp\biggl\{-\frac{\theta^2 n p_{\min} \wt{\delta}^d }{3\bigl(L_{\Psi} L_{\Phi}\bigr)^d}\biggr\}.
\end{align*}
Noting that by assumption $P(\wt{Q}) \leq p_{\max} \vol(\wt{Q})$ and $\wt{\delta}^{-d} \leq n$, the claim follows upon plugging back into \eqref{pf:prop_optimal_transport_3}, and setting
\begin{equation*}
a_0 := \frac{1}{3\bigl(L_{\Psi} L_{\Phi}\bigr)^d}~~\textrm{and}~~A_0 := \max \Bigl\{2N(\Xset),L_p L_{\Psi} L_{\Phi}, L_{\Psi} L_{\Phi}\Bigr\}
\end{equation*}
in the statement of the proposition.

\subsection{Non-random functionals and integrals}
\label{subsec:integrals}

Let us start by making the following observation, which we make use of repeatedly in this section. Let $\eta: [0,\infty) \to [0,\infty)$ be an otherwise arbitrary function. As a consequence of ~\ref{asmp:domain}, there exist constants $c_0$ and $a_3$ which depend on $\Xset$, such that for any $0 < \varepsilon \leq c_0$ it holds that 
\begin{equation}
\label{eqn:integral_boundary}
\int_{B(x,\varepsilon) \cap \Xset} \eta\biggl(\frac{\|x' - x\|}{\varepsilon}\biggr) \,dx' \geq a_3 \cdot \int_{B(x,\varepsilon)} \eta\biggl(\frac{\|x' - x\|}{\varepsilon}\biggr) \,dx'
\end{equation}
As a special case: when $\eta(x) = 1$, this implies $\vol\bigl(B(x,\varepsilon) \cap \Xset\bigr) \geq a_3 \nu_d \varepsilon^d$ for any $0 < \varepsilon \leq c_0$.

We have already upper bounded $E_r(f)$ by (a constant times) $D_2(f)$ in the proof of Lemma~\ref{lem:graph_sobolev_seminorm}. In Lemma~\ref{lem:first_order_graph_sobolev_seminorm_expected_lb}, we establish the reverse inequality.
\begin{lemma}[\textbf{cf. Lemma 9 of \citet{trillos2019}, Lemma 5.5 of \citet{burago2014}}]
	\label{lem:first_order_graph_sobolev_seminorm_expected_lb}
	For any $f \in \Leb^2(\Xset)$, and any $0 < h \leq c_0$, it holds that
	\begin{equation*}
	\sigma_KD_2(\Lambda_hf) \leq A_8 E_h(f).
	\end{equation*}
\end{lemma}

To prove Lemma~\ref{lem:first_order_graph_sobolev_seminorm_expected_lb}, we require upper and lower bounds on $\tau_h(x)$, as well as an upper bound on the gradient of $\tau_h$. The lower bound here---$\tau_h(x) \geq a_3$---is quite a bit a looser than what can be shown when $\Xset$ has no boundary. The same is the case regarding the upper bound of the size of the gradient $\|\nabla \tau_h(x)\|$. However, the bounds as stated here will be sufficient for our purposes.
\begin{lemma}
	\label{lem:tau_bound}
	For any $0 < h \leq c_0$, for all $x \in \Xset$ it holds that
	\begin{equation*}
	a_3 \leq \tau_h(x) \leq 1.
	\end{equation*}
	and 
	\begin{equation*}
	\|\nabla \tau_h(x)\| \leq \frac{1}{\sqrt{d\sigma_K} h}.
	\end{equation*}
\end{lemma}

Finally, to prove part (2) of Proposition~\ref{prop:isometry}, we require Lemma~\ref{lem:smoothening_error}, which gives an estimate on the error $\Lambda_h f - f$ in $\|\cdot\|_P^2$ norm.
\begin{lemma}[\textbf{c.f Lemma 8 of \citet{trillos2019}, Lemma 5.4 of \citet{burago2014}}]
	\label{lem:smoothening_error}
	For any $0 < h \leq c_0$, 
	\begin{equation}
	\label{eqn:smoothening_error_norm}
	\bigl\|\Lambda_hf\bigr\|_{P}^2 \leq \frac{p_{\max}}{a_3p_{\min}} \bigl\|f\bigr\|_{P}^2,
	\end{equation}
	and
	\begin{equation}
	\label{eqn:smoothening_error_energy}
	\bigl\|\Lambda_hf - f\bigr\|_{P}^2 \leq \frac{1}{a_3\sigma_Kp_{\min}} h^2 E_h(f),
	\end{equation}
	for all $f \in \Leb^2(\Xset)$.
\end{lemma}

\paragraph{Proof of Lemma~\ref{lem:first_order_graph_sobolev_seminorm_expected_lb}.}
For any $a \in \Reals$, $\Lambda_hf$ satisfies the identity
\begin{equation*}
\Lambda_hf(x) = a + \frac{1}{h^d\tau_h(x)}\int_{\Xset} \eta_h(x',x)\bigl(f(x') - a\bigr)\,dx',
\end{equation*}
and by differentiating with respect to $x$,  we obtain
\begin{equation*}
\bigl(\nabla \Lambda_hf\bigr)(x)= \frac{1}{h^d\tau_h(x)}\int_{\Xset} \bigl(\nabla \eta_h(x',\cdot)\bigr)(x)\bigl(f(x') - a\bigr)\,dx' + \nabla\biggl(\frac{1}{\tau_h}\biggr)(x)\cdot \frac{1}{h^d}\int_{\Xset} \eta_h(x',x)\bigl(f(x') - a\bigr)\,dx'
\end{equation*} 
Plugging in $a = f(x)$, we get $\nabla\Lambda_hf(x) = J_1(x)/\tau_h(x) + J_2(x)$ for
\begin{equation*}
J_1(x) := \frac{1}{h^d}\int_{\Xset} \bigl(\nabla \eta_h(x',\cdot)\bigr)(x)\bigl(f(x') - f(x)\bigr)\,dx',~~ J_2(x) := \nabla\biggl(\frac{1}{\tau_h}\biggr)(x)\cdot \frac{1}{h^d}\int_{\Xset} \eta_h(x',x)\bigl(f(x') - f(x)\bigr)\,dx'.
\end{equation*}
To upper bound $\bigl\|J_1(x)\bigr\|^2$, we first compute the gradient of $\eta_h(x',\cdot)$,
\begin{align*}
\bigl(\nabla\eta_h(x',\cdot)\bigr)(x) & = \frac{1}{h} \psi'\biggl(\frac{\|x'  - x\|}{h}\biggr) \frac{(x - x')}{\|x' - x\|} \\
& = \frac{1}{\sigma_Kh^{2}} K\biggl(\frac{\|x' - x\|}{h}\biggr) (x' - x),
\end{align*}
and additionally note that $\|J_1(x)\|^2 = \sup_{w}\bigl(\langle J_1(x), w \rangle\bigr)^2$ where the supremum is over unit norm vector. Taking $w$ to be a unit norm vector which achieves this supremum, we have that
\begin{align*}
\bigl\|J_1(x)\bigr\|^2 & = \frac{1}{\sigma_K^2 h^{4 + 2d}} \Biggl[\int_{\Xset} \bigl(f(x') - f(x)\bigr)K\biggl(\frac{\|x' - x\|}{h}\biggr)(x' - x)^{\top}w\,dx'\Biggr]^2 \\
& \leq \frac{1}{\sigma_K^2 h^{4 + 2d}} \biggl[\int_{\Xset}K\biggl(\frac{\|x' - x\|}{h}\biggr)\bigl((x' - x)^{\top} w\bigr)^2\,dx'\biggr] \biggl[\int_{\Xset}K\biggl(\frac{\|x' - x\|}{h}\biggr)\bigl(f(x') - f(x)\bigr)^2\,dx'\biggr].
\end{align*}
By a change of variables, we obtain
\begin{align*}
\int_{\Xset}K\biggl(\frac{\|x' - x\|}{h}\biggr)\bigl((x' - x)^{\top} w\bigr)^2\,dx' & \leq h^{d + 2} \int_{\Xset} K\bigl(\|z\|\bigr) \bigl(z^{\top} w\bigr)^2 \,dz \leq \sigma_K h^{d + 2},
\end{align*}
with the resulting upper bound
\begin{equation*}
\bigl\|J_1(x)\bigr\|^2 \leq \frac{1}{\sigma_K h^{2 + d}} \int_{\Xset}K\biggl(\frac{\|x' - x\|}{h}\biggr)\bigl(f(x') - f(x)\bigr)^2\,dx'.
\end{equation*}
To upper bound $\bigl\|J_2(x)\bigr\|^2$, we use the Cauchy-Schwarz inequality along with the observation $\eta_h(x',x) \leq \frac{1}{\sigma_K} K\bigl(\|x' - x\|/h\bigr)$ to deduce
\begin{align*}
\bigl\|J_2(x)\bigr\|^2 & \leq \Bigl\|\nabla\biggl(\frac{1}{\tau_h}\biggr)(x)\Bigr\|^2\frac{1}{h^{2d}} \biggl[\int_{\Xset}\eta_h(x',x) \,dx'\biggr] \cdot \biggl[\int_{\Xset} \eta_h(x',x)\bigl(f(x') - f(x)\bigr)^2 \,dx' \biggr] \\
& = \Bigl\|\nabla\biggl(\frac{1}{\tau_h}\biggr)(x)\Bigr\|^2\frac{\tau_h(x)}{h^d} \int_{\Xset} \eta_h(x',x)\bigl(f(x') - f(x)\bigr)^2 \,dx'\\ 
& \leq \Bigl\|\nabla\biggl(\frac{1}{\tau_h}\biggr)(x)\Bigr\|^2\frac{\tau_h(x)}{\sigma_K h^d}\int_{\Xset} K\biggl(\frac{\|x' - x\|}{h}\biggr)\bigl(f(x') - f(x)\bigr)^2 \,dx' \\
& \leq \frac{1}{da_3^2\sigma_K^2h^{2 + d}} \int_{\Xset} K\biggl(\frac{\|x' - x\|}{h}\biggr)\bigl(f(x') - f(x)\bigr)^2 \,dx',
\end{align*}
where the last inequality follows from the estimates on $\tau_h$ and $\nabla \tau_h$ provided in Lemma~\ref{lem:tau_bound}. Combining our bounds on $\bigl\|J_1(x)\bigr\|^2$ and $\bigl\|J_2(x)\bigr\|^2$ along with the lower bound on $\tau_h(x)$ in Lemma~\ref{lem:tau_bound} and integrating over $\Xset$, we have
\begin{align*}
\sigma_K D_2(\Lambda_h f) & = \sigma_K\int_{\Xset} \biggl\|\Bigl(\nabla \Lambda_hf)(x)\biggr\|^2 p^2(x) \,dx \\
& \leq 2 \sigma_K \int_{\Xset} \Biggl(\frac{\bigl\|J_1(x)\bigr\|^2}{\tau_h^2(x)} + \bigl\|J_2(x)\bigr\|^2\Biggr) p^2(x) \,dx \\
& \leq \biggl(\frac{1}{a_3^2} + \frac{1}{da_3^2\sigma_K}\biggr)\frac{2}{h^{d + 2}} \int_{\Xset} \int_{\Xset} K\biggl(\frac{\|x' - x\|}{h}\biggr)\bigl(f(x') - f(x)\bigr)^2 p^2(x) \,dx' \,dx \\
& \leq 2 \Bigl(1 + \frac{L_ph}{p_{\min}}\Bigr) \biggl(\frac{1}{a_3^2} + \frac{1}{da_3^2\sigma_K}\biggr) E_h(f),
\end{align*}
and taking $A_8 := 2 \Bigl(1 + \frac{L_pc_0}{p_{\min}}\Bigr) \biggl(\frac{1}{a_3^2} + \frac{1}{da_3^2\sigma_K}\biggr)$ completes the proof of Lemma~\ref{lem:first_order_graph_sobolev_seminorm_expected_lb}. 

\paragraph{Proof of Lemma~\ref{lem:tau_bound}.}
We first establish our estimates of $\tau_h(x)$, and then upper bound $\|\nabla\tau_h(x)\|$. Using \eqref{eqn:integral_boundary}, we have that
\begin{align*}
\tau_h(x) & = \frac{1}{h^d} \int_{\Xset \cap B(x,h)} \psi\biggl(\frac{\|x' - x\|}{h}\biggr) \,dx' \\
& \geq \frac{a_3}{h^d} \int_{B(x,h)} \psi\biggl(\frac{\|x' - x\|}{h}\biggr) \,dx' \\
& = a_3\int_{B(0,1)} \psi(\|z\|) \,dz,
\end{align*}
and it follows from similar reasoning that $\tau_h(x) \leq \int_{B(0,1)} \psi(\|z\|) \,dz$. 

We will now show that $\int_{B(0,1)} \psi(\|z\|) \,dz = 1$, from which we derive the estimates $a_3 \leq \tau_h(x) \leq 1$. To see the identity, note that on the one hand, by converting to polar coordinates and integrating by parts we obtain
\begin{align*}
\int_{B(0,1)} \psi\bigl(\|z\|\bigr) \,dz = d \nu_d \int_{0}^{1} \psi(t) t^{d - 1} \,dt = -\nu_d \int_{0}^{1} \psi'(t) t^{d} \,dt = \frac{\nu_d}{\sigma_K} \int_{0}^{1} t^{d + 1} K(t) \,dt;
\end{align*}
on the other hand, again converting to polar coordinates, we have
\begin{equation*}
\sigma_K = \frac{1}{d} \int_{\Reals^d} \|x\|^2 K(\|x\|) \,dx = \nu_d \int_{0}^{1}t^{d + 1} K(t) \,dt,
\end{equation*}
and so $\int_{B(0,1)} \psi(\|z\|) \,dz = 1$.

Now we upper bound $\|\nabla\tau_h(x)\|^2$. Exchanging derivative and integral, we have
\begin{align*}
\nabla\tau_h(x) = \frac{1}{h^d} \int_{\Xset} \bigl(\nabla \eta_h(x',\cdot)\bigr)(x) \,dx' = \frac{1}{\sigma_K h^{d + 2}} \int_{\Xset} K\biggl(\frac{\|x' - x\|}{h}\biggr)(x' - x)\,dx',
\end{align*}
whence by the Cauchy-Schwarz inequality,
\begin{equation*}
\|\nabla\tau_h(x)\|^2 \leq \frac{1}{\sigma_K^2 h^{2d + 4}} \biggl[\int_{\Xset} K\biggl(\frac{\|x' - x\|}{h}\biggr)\,dx'\biggr] \biggl[\int_{\Xset} K\biggl(\frac{\|x' - x\|}{h}\biggr)\|x' - x\|^2\,dx',\biggr] \leq \frac{1}{d\sigma_K h^{2}},
\end{equation*}
concluding the proof of Lemma~\ref{lem:tau_bound}. 

We remark that while $\nabla\tau(x) = 0$ when $B(x,r) \in \Xset$, near the boundary the upper bound we derived by using Cauchy-Schwarz appears tight. 

\paragraph{Proof of Lemma~\ref{lem:smoothening_error}.}
By Jensen's inequality and Lemma~\ref{lem:tau_bound},
\begin{align*}
\Bigl|\Lambda_hf(x)\Bigr|^2 & \leq \frac{1}{h^d\tau_h(x)}\int_{\Xset} \eta_h(x',x) \bigl[f(x')\bigr]^2 \,dx' \\
& \leq \frac{1}{a_3 h^dp_{\min}}\int_{\Xset} \eta_h(x',x) \bigl[f(x')\bigr]^2 p(x') \,dx'.
\end{align*}
Then, integrating over $x$, and recalling that$\int_{B(0,1)} \psi(\|z\|) = 1$ as shown in the proof of Lemma~\ref{lem:tau_bound}, we have
\begin{align*}
\bigl\|\Lambda_hf\bigr\|_{P}^2 & \leq \frac{1}{a_3h^d p_{\min}} \int_{\Xset} \int_{\Xset} \eta_h(x',x) \bigl[f(x')\bigr]^2 p(x') p(x) \,dx' \,dx \\ 
& \leq \frac{p_{\max}}{a_3h^dp_{\min}} \int_{\Xset} \bigl[f(x')\bigr]^2 p(x') \biggl(\int_{\Xset} \eta_h(x',x) \,dx\biggr) \,dx' \\
& \leq \frac{p_{\max}}{a_3p_{\min}} \int_{\Xset} \bigl[f(x')\bigr]^2 p(x') \biggl(\int_{B(0,1)} \psi(\|z\|) \,dz\biggr) \,dx' \\
& = \frac{p_{\max}}{a_3p_{\min}} \|f\|_{P}^2.
\end{align*}

To establish \eqref{eqn:smoothening_error_energy}, noting that $\Lambda_ha = a$ for any $a \in \Reals$, we have that
\begin{align*}
\bigl|\Lambda_rf(x) - f(x)\bigr|^2 & = \biggl[\frac{1}{h^d\tau_h(x)} \int_{\Xset} \eta_h(x',x) \bigl(f(x') - f(x)\bigr) \,dx'\biggr]^2 \\
& \leq \frac{1}{h^{2d} \tau_h^2(x)} \biggl[\int_{\Xset} \eta_h(x',x) \,dx'\biggr] \cdot \biggl[\int_{\Xset} \eta_h(x',x) \bigl(f(x') - f(x)\bigr)^2 \,dx'\biggr] \\
& = \frac{1}{h^d \tau_h(x)} \int_{\Xset} \eta_h(x',x) \bigl(f(x') - f(x)\bigr)^2 \,dx'. \\
& \leq \frac{1}{h^d \tau_h(x) p_{\min}} \int_{\Xset} \eta_h(x',x) \bigl(f(x') - f(x)\bigr)^2 p(x') \,dx'.
\end{align*}
From here, we can use the lower bound $\tau_h(x) \geq a_3$ stated in Lemma~\ref{lem:tau_bound}, as well as the upper bound $\eta_h(x',x) \leq (1/\sigma_K) K(\|x' - x\|/h)$, to deduce
\begin{equation*}
\bigl|\Lambda_rf(x) - f(x)\bigr|^2 \leq \frac{1}{h^{d} a_3 \sigma_K p_{\min}} \int_{\Xset} K\biggl(\frac{\|x' - x\|}{h}\biggr) \bigl(f(x') - f(x)\bigr)^2 p(x') \,dx'.
\end{equation*}
Then integrating over $\Xset$ with respect to $p$ yields \eqref{eqn:smoothening_error_energy}.

\subsection{Random functionals}
\label{subsec:random_functionals}

We will use Lemma~\ref{lem:poincare} in the proof of Proposition~\ref{prop:isometry}. 
\begin{lemma}[\textbf{cf. Lemma 3.4 of \citet{burago2014}}]
	\label{lem:poincare}
	Let $U \subseteq \Xset$ be a measurable subset such that $\mathrm{vol}(U) > 0$, and $\diam(U) \leq 2A_0\wt{\delta}$. Then, letting $a = (\wt{P}_n(U))^{-1} \cdot \int_{U} f(x) \wt{p}_n(x) \,dx$ be the average of $f$ over $U$, it holds that
	\begin{equation*}
	\int_{U} \Bigl|f(x)-a\Bigr|^2 \wt{p}_n(x)\,dx \leq A_3 r^2 \wt{E}_r(f,U).
	\end{equation*}
\end{lemma}

Now we relate $\wt{E}_r(f)$ and $E_r(f)$. Some standard calculations show that for $A_1 := 3A_0/p_{\min}$,
\begin{equation}
\label{eqn:calder19_1}
\bigl(1 - A_1(\theta + \wt{\delta})\bigr) E_r(f) \leq \wt{E}_r(f) \leq \bigl(1 + A_1(\theta + \wt{\delta})\bigr) E_r(f),
\end{equation}
as well as implying that the norms $\|f\|_{P}$ and $\|f\|_{n}$ satisfy
\begin{equation}
\label{eqn:calder19_2}
\bigl(1 - A_1(\theta + \wt{\delta})\bigr) \|f\|_{P}^2 \leq \|f\|_{\wt{P}_n}^2 \leq \bigl(1 + A_1(\theta + \wt{\delta})\bigr) \|f\|_{P}^2.
\end{equation}

Lemma~\ref{lem:first_order_graph_sobolev_seminorm_discretized} relates the graph Sobolev semi-norm $b_r(\wt{\mc{P}}f)$ to the non-local energy $\wt{E}_r(f)$. 
\begin{lemma}[\textbf{cf. Lemma 13 of \citet{trillos2019}, Lemma 4.3 of \citet{burago2014}}]
	\label{lem:first_order_graph_sobolev_seminorm_discretized}
	For any $f \in \Leb^2(\Xset)$,
	\begin{equation*}
	b_r(\wt{\mc{P}}f) \leq \Bigl(1 + A_9\frac{\wt{\delta}}{r}\Bigr) \wt{E}_{r + 2A_0\wt{\delta}}(f).
	\end{equation*}
\end{lemma}

In Lemma~\ref{lem:first_order_graph_sobolev_seminorm_discretized_lb}, we establish the reverse of Lemma~\ref{lem:first_order_graph_sobolev_seminorm_discretized}. 
\begin{lemma}[\textbf{cf. Lemma 14 of \citet{trillos2019}}]
	\label{lem:first_order_graph_sobolev_seminorm_discretized_lb}
	For any $u \in \Leb^2(P_n)$, 
	\begin{equation*}
	\wt{E}_{r - 2A_0\wt{\delta}}\bigl(\wt{\mc{P}}^{\star}u\bigr) \leq \biggl(1 + A_3\frac{\wt{\delta}}{r}\biggr) b_{r}(u).
	\end{equation*}
\end{lemma}

\paragraph{Proof of Lemma~\ref{lem:poincare}.}
A symmetrization argument implies that
\begin{equation}
\label{pf:poincare_1}
\int_{U} \Bigl|f(x)-a\Bigr|^2 \wt{p}_n(x)\,dx = \frac{1}{2\wt{P}_n(U)} \int_{U} \int_{U} \bigl|f(x') - f(x)\bigr|^2 \wt{p}_n(x') \wt{p}_n(x) \,dx' \,dx
\end{equation}
Now, since $x'$ and $x$ belong to $U$, we have that $\|x' - x\| \leq 2A_0\wt{\delta}$. Set $V = B(x,r) \cap B(x',r)$, and note that $B(x,r - 2A_0\wt{\delta}) \subseteq V$. Moreover, $r - 2A_0\wt{\delta} \leq r \leq c_0$ by assumption. Therefore by \eqref{eqn:integral_boundary},
\begin{equation*}
\vol\bigl(V \cap \Xset\bigr) \geq \vol\bigl(B(x,r - 2A_0\wt{\delta}) \cap \mc{X}\bigr) \geq a_3 \nu_d (r - 2A_0\wt{\delta})^d \geq \frac{a_3 \nu_d}{2^d}r^d 
\end{equation*}
where the last inequality follows since $\wt{\delta} \leq \frac{1}{4A_0}r$. Using the triangle inequality 
\begin{equation*}
\bigl|f(x') - f(x)\big|^2 \leq 2\bigl(\bigl|f(x') - f(z)\big|^2 + \bigl|f(z) - f(x)\big|^2\bigr)
\end{equation*}
we have that for any $x$ and $x'$ in $U$,
\begin{align}
\bigl|f(x') - f(x)\big|^2 & \leq \frac{2}{\vol(V \cap \Xset)} \int_{V \cap \Xset} \bigl|f(x') - f(z)\big|^2 + \bigl|f(z) - f(x)\big|^2 \,dz \nonumber \\
& \leq \frac{2^{d + 1}}{a_3 \nu_d r^d} \int_{V \cap \Xset} \bigl|f(x') - f(z)\big|^2 + \bigl|f(z) - f(x)\big|^2 \,dz \nonumber \\
& \leq \frac{2^{d + 2}}{K(1) a_3 \nu_d r^d p_{\min}} \Bigl(F(x') + F(x)\bigr), \label{pf:poincare_2}
\end{align}
where in the last inequality we set
\begin{equation*}
F(x) := \int_{\Xset} K\biggl(\frac{\|z - x\|}{r}\biggr) \bigl(f(z) - f(x)\bigr)^2 \wt{p}_n(x) \,dx,
\end{equation*}
and use the facts that $\wt{p}_n(x) \geq p_{\min}/2$, that $K(\|z  - x\|/r) \geq K(1)$ for all $z \in B(x,r)$. 

Plugging the upper bound \eqref{pf:poincare_2} back into \eqref{pf:poincare_1}, we have that
\begin{align*}
\int_{U} \Bigl|f(x)-a\Bigr|^2 \wt{p}_n(x)\,dx & \leq \frac{2^{d + 2}}{K(1) a_3 \nu_d r^d} \int_{U} F(x)\wt{p}_n(x) \,dx \\
& = \frac{2^{d + 2}}{K(1) a_3 \nu_d}r^2 \wt{E}_r(f,U),
\end{align*}
and Lemma~\ref{lem:poincare} follows by taking $A_3 := 2^{d + 2}/(K(1) a_3 \nu_d)$.

\paragraph{Proof of Lemma~\ref{lem:first_order_graph_sobolev_seminorm_discretized}.}
Recalling that $\bigl(\wt{\mc{P}}f\bigr)(X_i) = n \cdot \int_{U_i} f(x) \wt{p}_n(x) \,dx$, by Jensen's inequality,
\begin{equation*}
\biggl(\bigl(\wt{\mc{P}}f\bigr)(X_i) - \bigl(\wt{\mc{P}}f\bigr)(X_j)\biggr)^2 \leq n^2 \cdot \int_{U_i} \int_{U_j} \bigl(f(x') - f(x)\bigr)^2 \wt{p}_n(x') \wt{p}_n(x) \,dx' \,dx.
\end{equation*}
Additionally, the non-increasing and Lipschitz properties of $K$ imply that for any $x \in U_i$ and $x' \in U_j$, 
\begin{equation*}
K\biggl(\frac{\|X_i - X_j\|}{r}\biggr) \leq K\biggl(\frac{\bigl(\|x' - x\| - 2A_0\wt{\delta}\bigr)_{+}}{r}\biggr) \leq K\biggl(\frac{\|x' - x\|}{r + 2A_0\wt{\delta}}\biggr) + \frac{2L_KA_0\wt{\delta}}{r}\1\Bigl\{\|x' - x\| \leq r + 2A_0\wt{\delta}\Bigr\}.
\end{equation*}
As a result, the graph Dirichlet energy is upper bounded as follows:
\begin{align*}
b_r(\wt{\mc{P}}f) & = \frac{1}{n^2r^{d + 2}} \sum_{i,j = 1}^n \Bigl(\bigl(\wt{\mc{P}}f\bigr)(X_i) - \bigl(\wt{P}f\bigr)(X_j)\Bigr)^2 K\biggl(\frac{\|X_i - X_j\|}{r}\biggr) \\
& \leq \frac{1}{r^{d + 2}} \sum_{i,j = 1}^n \int_{U_i} \int_{U_j}  \bigl(f(x') - f(x)\bigr)^2 \wt{p}_n(x') \wt{p}_n(x) K\biggl(\frac{\|X_i - X_j\|}{r}\biggr) \,dx' \,dx \\
& \leq \frac{1}{r^{d + 2}} \sum_{i,j = 1}^n \int_{U_i} \int_{U_j}  \bigl(f(x') - f(x)\bigr)^2 \wt{p}_n(x') \wt{p}_n(x) \biggl[K\biggl(\frac{\|x' - x\|}{r + 2A_0\wt{\delta}}\biggr) + \frac{2L_KA_0\wt{\delta}}{r}\1\Bigl\{\|x' - x\| \leq r + 2\wt{\delta}\Bigr\}\biggr] \,dx' \,dx \\
& = \Bigl(1 + 2A_0\frac{\wt{\delta}}{r}\Bigr)^{d + 2}\biggl[\wt{E}_{r + 2A_0\wt{\delta}}(f) + \frac{2L_KA_0\wt{\delta}}{r}\wt{E}_{r + 2A_0\wt{\delta}}(f; \1_{[0,1]})\biggr],
\end{align*}
for $\1_{[0,1]}(t) = \1\{0 \leq t \leq 1\}$. But by assumption $\wt{E}_{r + 2A_0\wt\delta}(f; \1_{[0,1]}) \leq 1/(K(1))\wt{E}_{r + 2A_0\wt{\delta}}(f)$, and so we obtain
\begin{equation*}
b_r(\wt{\mc{P}}f) \leq \Bigl(1 + 2A_0\frac{\wt{\delta}}{r}\Bigr)^{d + 2} \Bigl(1 + \frac{2L_KA_0\wt{\delta}}{rK(1)}\Bigr) \wt{E}_{r + 2A_0\wt{\delta}}(f);
\end{equation*}
the Lemma follows upon choosing $A_9 := A_0(2^{d + 4} + \frac{4L_K}{K(1)})$.

\paragraph{Proof of Lemma~\ref{lem:first_order_graph_sobolev_seminorm_discretized_lb}.}
For brevity, we write $\wt{r} := r - 2A_0\wt{\delta}$. We begin by expanding the energy $\wt{E}_{\wt{r}}\bigl(\wt{\mc{P}}^{\star}u\bigr)$ as a double sum of double integrals,
\begin{align*}
\wt{E}_{\wt{r}}\bigl(\wt{\mc{P}}^{\star}u\bigr) & = \frac{1}{\wt{r}^{d + 2}} \sum_{i = 1}^{n} \sum_{j = 1}^{n} \int_{U_i} \int_{U_j} \Bigl(u(X_i) - u(X_j)\Bigr)^2 K\biggl(\frac{\|x' - x\|}{\wt{r}}\biggr) \wt{p}_n(x') \wt{p}_n(x) \,dx' \,dx.
\end{align*}
We next use the Lipschitz property of the kernel $K$---in particular that for $x \in U_i$ and $x' \in U_j$,
\begin{equation*}
K\biggl(\frac{\|x' - x\|}{\wt{r}}\biggr) \leq K\biggl(\frac{\|X_i - X_j\|}{r}\biggr) + \frac{2A_0L_K\wt{\delta}}{\wt{r}} \cdot \1\biggl\{\frac{\|x' - x\|}{\wt{r}} \leq 1\biggr\},
\end{equation*}
---to conclude that
\begin{align}
\wt{E}_{\wt{r}}\bigl(\wt{\mc{P}}^{\star}u\bigr) & \leq \frac{1}{n^2\wt{r}^{d + 2}} \sum_{i = 1}^{n} \sum_{j = 1}^{n} \Bigl(u(X_i) - u(X_j)\Bigr)^2 K\biggl(\frac{\|X_i - X_j\|}{r}\biggr) + \frac{2A_0L_K\wt{\delta}}{\wt{r}}\wt{E}_{\wt{r}}(\wt{\mc{P}}^{\star}u,\1_{[0,1]}\bigr) \nonumber \\
& \leq \biggl(1 + 2^{d + 2}A_0\frac{\wt{\delta}}{r}\biggr) b_r(u) + \frac{2A_0L_K\wt{\delta}}{\wt{r}}\wt{E}_{\wt{r}}(\wt{\mc{P}}^{\star}u,\1_{[0,1]}\bigr) \nonumber \\
& \leq \biggl(1 + 2^{d + 2}A_0\frac{\wt{\delta}}{r}\biggr) b_r(u) + \frac{4A_0L_K\wt{\delta}}{K(1)r} \wt{E}_{\wt{r}}(\wt{\mc{P}}^{\star}u\bigr). \nonumber 
\end{align}
In other words,
\begin{align*}
\wt{E}_{\wt{r}}\bigl(\wt{\mc{P}}^{\star}u\bigr) & \leq \biggl(1 - \frac{4A_0L_K\wt{\delta}}{K(1)r}\biggr)^{-1}\biggl(1 + 2^{d + 2}A_0\frac{\wt{\delta}}{r}\biggr) b_r(u) \\
& \leq \biggl(1 + \frac{\wt{\delta}}{r}\Bigl(\frac{8A_0L_K}{K(1)} + 2^{d + 3}\Bigr)\biggr) b_r(u),
\end{align*}
where the second inequality follows from the algebraic identities $(1 - t)^{-1} \leq (1 + 2t)$ for any $0 < t < 1/2$ and $(1 + s)(1 + t) < 1 + 2s + t$ for any $0 < t < 1$ and $s > 0$. The Lemma follows upon choosing $A_3 := \frac{8A_0L_K}{K(1)} + 2^{d + 3}$. 

\subsection{Proof of Propositions~\ref{prop:dirichlet_energies} and \ref{prop:isometry}}
\label{subsec:proof_of_prop_dirichlet_energies_and_isometry}

\paragraph{Proof of Proposition~\ref{prop:dirichlet_energies}.}
Part (1) of Proposition~\ref{prop:dirichlet_energies} follows from
\begin{align*}
\sigma_K D_2(\Lambda_{r - 2A_0\wt{\delta}} \wt{\mc{P}}^{\star}u) & \overset{(i)}{\leq} A_8 E_{r - 2A_0\wt{\delta}}(\wt{\mc{P}}^{\star}u) \\
& \overset{(ii)}{\leq} A_8 \Bigl(1 + A_1(\theta + \wt{\delta})\Bigr) \wt{E}_{r - 2A_0\wt{\delta}}(\wt{\mc{P}}^{\star}u) \\
& \overset{(iii)}{\leq} A_8 \Bigl(1 + A_1(\theta + \wt{\delta})\Bigr) \cdot \biggl(1 + A_3\frac{\wt{\delta}}{r}\biggr) b_r(u),
\end{align*}
where $(i)$ follows from Lemma~\ref{lem:first_order_graph_sobolev_seminorm_expected_lb}, $(ii)$ follows from \eqref{eqn:calder19_1}, and $(iii)$ follows from~Lemma~\ref{lem:first_order_graph_sobolev_seminorm_discretized_lb}.

Part (2) of Proposition~\ref{prop:dirichlet_energies} follows from
\begin{align*}
b_r(\wt{\mc{P}}f) & \overset{(iv)}{\leq} \Bigl(1 + A_9\frac{\wt{\delta}}{r}\Bigr)\wt{E}_{r + 2A_0\wt{\delta}}(f)\\
& \overset{(v)}{\leq} \Bigl(1 + A_1(\theta + \wt{\delta})\Bigr) \Bigl(1 + A_9\frac{\wt{\delta}}{r}\Bigr){E}_{r + 2A_0\wt{\delta}}(f) \\
& \overset{(vi)}{\leq} \Bigl(1 + A_1(\theta + \wt{\delta})\Bigr) \cdot \Bigl(1 + A_9\frac{\wt{\delta}}{r}\Bigr) \cdot \Bigl(\frac{C_5 p_{\max}^2}{p_{\min}^2}\Bigr) \cdot \sigma_{K} D_2(f),
\end{align*}
where $(iv)$ follows from Lemma~\ref{lem:first_order_graph_sobolev_seminorm_discretized}, $(v)$ follows from \eqref{eqn:calder19_1}, and $(vi)$ follows from the proof of Lemma~\ref{lem:graph_sobolev_seminorm}.

\paragraph{Proof of Proposition~\ref{prop:isometry}.}
\textit{Proof of (1).}
We begin by upper bounding $\bigl\|\wt{\mc{P}}f\bigr\|_{n}$. By the Cauchy-Schwarz inequality and the bound on $\|\wt{p}_n - p\|_{\infty}$ in \eqref{eqn:optimal_transport_2},
\begin{align*}
\Bigl|\wt{\mc{P}}f(X_i)\Bigr|^2 & = n^2 \Bigl|\int_{U_i} f(x) \wt{p}_n(x) \,dx\Bigr|^2 \\
& \leq n \int_{U_i} \bigl|f(x)\bigr|^2 \wt{p}_n(x) \,dx \\
& \leq n \Bigl(1 + A_1(\theta + \wt{\delta})\Bigr)\biggl[\int_{U_i} \bigl|f(x)\bigr|^2 p(x) \,dx + A_1(\theta + \wt{\delta}) \int_{U_i} \bigl|f(x)\bigr|^2 p(x) \,dx\biggr],
\end{align*}
and summing over $i = 1,\ldots,n$, we obtain
\begin{equation}
\label{pf:prop_isometry_1}
\bigl\|\wt{\mc{P}}f\bigr\|_{n}^2 \leq \biggl(1 + A_1(\theta + \wt{\delta})\biggr) \bigl\|f\bigr\|_{P}^2.
\end{equation}
Now, noticing that $\bigl\|\wt{\mc{P}}f\bigr\|_{n} = \bigl\|\wt{P}^{\star}\wt{P}f\bigr\|_{\wt{P}_n}$, we can use the upper bound \eqref{pf:prop_isometry_1} to show that
\begin{align}
\Bigl|\bigl\|\wt{\mc{P}}f\bigr\|_{n}^2 - \bigl\|f\bigr\|_{P}^2\Bigr| & \leq \Bigl|\bigl\|\wt{\mc{P}}f\bigr\|_{n}^2 - \bigl\|f\bigr\|_{\wt{P}_n}^2\Bigr| + \Bigl|\bigl\|f\bigr\|_{\wt{P}_n}^2 - \bigl\|f\bigr\|_{P}^2\Bigr| \nonumber \\
& \overset{(i)}{\leq} \Bigl|\bigl\|\wt{\mc{P}}f\bigr\|_{n}^2 - \bigl\|f\bigr\|_{\wt{P}_n}^2\Bigr|  + A_1(\theta + \wt{\delta}) \bigl\|f\bigr\|_{P}^2 \\
& \overset{(ii)}{\leq} 2 \sqrt{1 + A_1(\theta + \wt{\delta})} \Bigl|\bigl\|\wt{\mc{P}}f\bigr\|_{n} - \bigl\|f\bigr\|_{\wt{P}_n}\Bigr| \cdot \bigl\|f\bigr\|_{P} + A_1(\theta + \wt{\delta}) \bigl\|f\bigr\|_{P}^2 \nonumber \\
& \leq 2 \sqrt{1 + A_1(\theta + \wt{\delta})} \bigl\|\wt{\mc{P}}^{\star}\wt{\mc{P}}f - f\bigr\|_{\wt{P}_n} \cdot \bigl\|f\bigr\|_{P} + A_1(\theta + \wt{\delta}) \bigl\|f\bigr\|_{P}^2, \label{pf:prop_isometry_2}
\end{align}
where $(i)$ follows from \eqref{eqn:calder19_2} and $(ii)$ follows from \eqref{eqn:calder19_2} and  \eqref{pf:prop_isometry_1}. 

It remains to upper bound $\bigl\|\wt{\mc{P}}^{\star}\wt{\mc{P}}f - f\bigr\|_{\wt{P}_n}^2$. Noting that $\wt{\mc{P}}^{\star}\wt{\mc{P}}f$ is piecewise constant over the cells $U_i$, we have
\begin{equation*}
\bigl\|\wt{\mc{P}}^{\star}\wt{\mc{P}}f - f\bigr\|_{\wt{P}_n}^2 = \sum_{i = 1}^{n} \int_{U_i} \biggl(f(x) - n\cdot\int_{U_i} f(x') \wt{p}_n(x') \,dx'\biggr)^2 \wt{p}_n(x) \,dx.
\end{equation*}
From Lemma~\ref{lem:poincare}, we have that for each $i = 1,\ldots,n$,
\begin{equation*}
\int_{U_i} \biggl(f(x) - n\cdot\int_{U_i} f(x') \wt{p}_n(x') \,dx'\biggr)^2 \wt{p}_n(x) \,dx \leq A_3r^2 \wt{E}_r(f,U_i).
\end{equation*}
Summing up over $i$ on both sides of the inequality gives
\begin{equation*}
\bigl\|\wt{\mc{P}}^{\star}\wt{\mc{P}}f - f\bigr\|_{\wt{P}_n}^2 \leq A_3r^2 \wt{E}_r(f,\Xset) \leq  A_3 \Bigl(1 + A_1(\theta + \wt{\delta})\Bigr) \cdot \Bigl(\frac{C_5 p_{\max}^2}{p_{\min}^2}\Bigr) \cdot \sigma_{K}r^2 D_2(f),
\end{equation*}
where the latter inequality follows from the proof of Proposition~\ref{prop:dirichlet_energies}, Part (2). Then Proposition~\ref{prop:isometry}, Part (1) follows by plugging this inequality into \eqref{pf:prop_isometry_2} and taking 
\begin{equation*}
A_5 := 2\sqrt{A_3} \Bigl(1 + A_1(\theta + \wt{\delta})\Bigr) \Bigl(\frac{\sqrt{C_5} p_{\max}}{p_{\min}}\Bigr) \cdot \sqrt{\sigma_{K}}.
\end{equation*}

\textit{Proof of (2).}
By the triangle inequality and \eqref{eqn:calder19_2},
\begin{align}
\Bigl|\|\wt{\mc{I}}u\|_{P}^2 - \|u\|_{n}^2\Bigr| & \leq \Bigl|\|\wt{\mc{I}}u\|_{P}^2 - \|\wt{\mc{I}}u\|_{\wt{P}_n}^2\Bigr| + \Bigl|\|\wt{\mc{I}}u\|_{\wt{P}_n}^2 - \|u\|_{n}^2\Bigr| \nonumber \\
& \leq A_1(\theta + \wt{\delta}) \|\wt{\mc{I}}u\|_{\wt{P}_n}^2 + \Bigl|\|\wt{\mc{I}}u\|_{\wt{P}_n}^2 - \|u\|_{n}^2\Bigr| \nonumber\\
& = A_1(\theta + \wt{\delta}) \|\wt{\mc{I}}u\|_{\wt{P}_n}^2 + \Bigl(\|\wt{\mc{I}}u\|_{\wt{P}_n} + \|u\|_{n}\Bigr) \cdot \Bigl|\|\wt{\mc{I}}u\|_{\wt{P}_n} - \|u\|_{n}\Bigr| \label{pf:prop_isometry_4}
\end{align}
To upper bound the second term in the above expression, we first note that~$\|u\|_{n} = \|\wt{\mc{P}}^{\star}u\|_{\wt{P}_n}$, and thus
\begin{align}
\Bigl|\|\wt{\mc{I}}u\|_{\wt{P}_n} - \|u\|_{n} \Bigr| & = \Bigl|\|\wt{\mc{I}}u\|_{\wt{P}_n} - \|\wt{\mc{P}}^{\star}u\|_{\wt{P}_n}\Bigr| \nonumber \\
& \overset{(iii)}{\leq} \|\Lambda_{\wt{r}}\wt{\mc{P}}^{\star}u - \wt{\mc{P}}^{\star}u\|_{\wt{P}_n} \nonumber \\
& \overset{(iv)}{\leq} \wt{r} \sqrt{\frac{1}{a_3\sigma_K p_{\min}} E_{\wt{r}}(\wt{\mc{P}}^{\star}u)} \nonumber \\
& \overset{(v)}{\leq} \wt{r} \sqrt{\frac{1 + A_1(\theta + \wt{\delta})}{a_3\sigma_K p_{\min}} \Bigl(1 + A_3\frac{\wt{\delta}}{r}\Bigr) b_r(u)},\label{pf:prop_isometry_5}
\end{align}
where $(iii)$ follows by the triangle inequality, $(iv)$ follows from Lemma~\ref{lem:smoothening_error}, and $(v)$ follows from \eqref{eqn:calder19_1} and Lemma~\ref{lem:first_order_graph_sobolev_seminorm_discretized_lb}. On the other hand, by \eqref{eqn:calder19_2} and Lemma~\ref{lem:smoothening_error},
\begin{align*}
\|\wt{\mc{I}}u\|_{\wt{P}_n}^2 & \leq \Bigl(1 + A_1(\theta + \wt{\delta})\Bigr) \|\wt{\mc{I}}u\|_{P}^2 \\
& \leq \frac{p_{\max}}{a_3p_{\min}} \cdot \Bigl(1 + A_1(\theta + \wt{\delta})\Bigr) \|\wt{\mc{P}}^{\star}u\|_{P}^2 \\
& \leq \frac{p_{\max}}{a_3p_{\min}} \cdot \Bigl(1 + A_1(\theta + \wt{\delta})\Bigr)^2 \|\wt{\mc{P}}^{\star}u\|_{\wt{P}_n}^2 \\
& = \frac{p_{\max}}{a_3p_{\min}} \cdot \Bigl(1 + A_1(\theta + \wt{\delta})\Bigr)^2 \|u\|_{n}^2.
\end{align*}
Plugging this estimate along with \eqref{pf:prop_isometry_5} back into \eqref{pf:prop_isometry_4}, we obtain part (2) of Proposition~\ref{prop:isometry}, upon choosing
\begin{equation*}
A_6 := \biggl(3\sqrt{\frac{2p_{\max}}{p_{\min}}} + 1\biggr)\sqrt{\frac{4}{a_3\sigma_Kp_{\min}}},~~A_7:=4A_1\frac{p_{\max}}{a_3p_{\min}}.
\end{equation*}

\section{Bound on the empirical norm}
\label{sec:empirical_norm}

In Lemma~\ref{lem:empirical_norm_sobolev}, we lower bound $\norm{f_0}_n^2$ by (a constant times) the $\Leb^2(\Xset)$ norm of $f$.

\begin{lemma}
	\label{lem:empirical_norm_sobolev}
	Fix $\delta \in (0,1)$ Suppose $P$ satisfies~\ref{asmp:density}. If $f \in H^1(\Xset,M)$ is lower bounded in $\Leb^2(\Xset)$ norm,
	\begin{equation}
	\label{eqn:empirical_norm_sobolev_1}
	\norm{f}_{\Leb^2(\Xset)} \geq \frac{C_6 M}{\delta} \cdot \max\Bigl\{n^{-1/2},n^{-1/d}\Bigr\}.
	\end{equation}
	Then with probability at least $1 - 5 \delta$,
	\begin{equation}
	\label{eqn:empirical_norm_sobolev}
	\norm{f}_n^2 \geq \delta \cdot \Ebb\Bigl[\norm{f}_n^2\Bigr].
	\end{equation}
\end{lemma}

\paragraph{Proof of Lemma~\ref{lem:empirical_norm_sobolev}.}
In this proof, we will find it more convenient to deal with the parameterization $b = 1/\delta$. To establish \eqref{eqn:empirical_norm_sobolev}, it is sufficient to show that
\begin{equation*}
\mathbb{E}\bigl[\norm{f}_n^4\bigr] \leq \left(1 + \frac{1}{b^2}\right)\cdot \left(\mathbb{E}\bigl[\norm{f}_n^2\bigr]\right)^2;
\end{equation*}
then \eqref{eqn:empirical_norm_sobolev} follows from the Paley-Zygmund inequality (Lemma~\ref{lem:paley_zygmund}). Since $p \leq p_{\max}$ is uniformly bounded, we can relate $\mathbb{E}\bigl[\norm{f}_n^4\bigr]$ to the $\Leb^4(\Xset)$-norm,
\begin{equation*}
\mathbb{E}\bigl[\norm{f}_n^4\bigr] = \frac{(n-1)}{n}\left(\mathbb{E}\Bigl[\norm{f}_n^2\Bigr]\right)^2 + \frac{\mathbb{E}\Bigl[\bigl(f(X_1)\bigr)^4\Bigr]}{n} \leq \left(\mathbb{E}\Bigl[\norm{f}_n^2\Bigr]\right)^2 + p_{\max}\frac{\norm{f}_{\Leb^4(\Xset)}^4}{n}.
\end{equation*}
We will use the Sobolev inequalities as a tool to show that $\|f\|_{\Leb^4(\Xset)}^4/n \leq \bigl(\mathbb{E}[\|f\|_n^2]\bigr)^2/(b^2p_{\max})$, whence the claim of the Lemma is shown. The nature of the inequalities we use depend on the value of $d$. In particular, we will use the following relationships between norms: for any $f \in H^1(\mc{X};M)$, 
\begin{equation*}
\begin{rcases*}
\sup_{x \in \mc{X}}|f(x)|,& \textrm{$d = 1$}\\
\|f\|_{\Leb^{q}(\mc{X})},& \textrm{$d = 2$, for all $0 < q < \infty$} \\
\|f\|_{\Leb^{q}(\mc{X})},& \textrm{$d \geq 3$, for all $0 < q \leq 2d/(d - 2)$}
\end{rcases*}
\leq C_7 \cdot M.
\end{equation*}
(See Theorem 6 in Section 5.6.3 of \citet{evans10} for a complete statement and proof of the various Sobolev inequalities.)

As a result, we divide our analysis into three cases: (i) the case where $d < 2$, (ii) the case where $d > 2$, and (iii) the borderline case $d = 2$.

\textit{Case 1: $d < 2$.}
The $\Leb^4(\Xset)$-norm of $f$ can be bounded in terms of the $\Leb^2(\Xset)$ norm,
\begin{align*}
\norm{f}_{\Leb^4(\Xset)}^4 & \leq \left(\sup_{x \in \Xset} \abs{f(x)}\right)^2 \cdot \int_{\Xset} [f(x)]^2 \,dx \leq C_7^2 M^2 \cdot \|f\|_{\Leb^2(X)}^2.
\end{align*}
Since by assumption
\begin{equation*}
\norm{f}_{\Leb^2(\Xset)}^2 \geq C_6^2 \cdot b^2 \cdot M^2 \cdot \frac{1}{n},
\end{equation*}
we have
\begin{equation*}
p_{\max} \frac{\norm{f}_{\Leb^4(\Xset)}^4}{n} \leq C_7^2 M^2p_{\max} \cdot \frac{\norm{f}_{\Leb^2(\Xset)}^2}{n} \leq \frac{C_7p_{\max}}{C_6^2 b^2} \norm{f}_{\Leb^2(\Xset)}^4 \leq \frac{\bigl(\Ebb\bigl[\norm{f}_n^2\bigr]\bigr)^2}{b^2},
\end{equation*}
where the last inequality follows by taking $C_6 \geq C_7 \sqrt{p_{\max}/p_{\min}}$. 

\textit{Case 2: $d > 2$.}
Let $\theta = 2 - d/2$ and $q = 2d/(d - 2)$. Noting that $4 = 2\theta + (1 - \theta)q$, Lyapunov's inequality implies
\begin{equation*}
\norm{f}_{\Leb^4(\Xset)}^4 \leq \norm{f}_{\Leb^2(\Xset)}^{2\theta} \cdot \norm{f}_{\Leb^q(\Xset)}^{(1 - \theta)q} \leq \norm{f}_{\Leb^2(\Xset)}^{4} \cdot \left(\frac{C_7\norm{f}_{H^1(\Xset)}}{\norm{f}_{\Leb^2(\Xset)}}\right)^{d}.
\end{equation*}
By assumption, $\norm{f}_{\Leb^2(\Xset)} \geq C_6 b \norm{f}_{H^1(\Xset)} n^{-1/d}$, and therefore
\begin{equation*}
p_{\max} \frac{\norm{f}_{\Leb^4(\Xset)}^4}{n} \leq \norm{f}_{\Leb^2(\Xset)}^4p_{\max} \cdot \left(\frac{C_7\norm{f}_{H^1(\Xset)}}{n^{1/d}\norm{f}_{\Leb^2(\Xset)}}\right)^{d} \leq \frac{C_7^dp_{\max}\norm{f}_{\Leb^2(\Xset)}^4}{C_6^db^{d}} \leq \frac{\bigl(\Ebb\bigl[\norm{f}_n^2\bigr]\bigr)^2}{b^2}.
\end{equation*}
where the last inequality follows by taking $C_6 \geq C_7(p_{\max}/p_{\min})^{1/d}$, and keeping in mind that $d > 2$ and $b \geq 1$. 

\textit{Case 3: $d = 2$.}
Fix $t \in (1/2,1)$, and suppose that
\begin{equation}
\label{pf:empirical_norm_sobolev_1}
\|f\|_{\Leb^2(\Xset)} \geq \frac{C_6 M}{\delta} \cdot n^{-t/2}.
\end{equation} 
Putting $q = 2/(1 - t)$, we have that $\|f\|_{\Leb^{q}(\mc{X})} \leq C_7 \cdot M$, and it follows from derivations similar to those in Case 2 that $\|f\|_{\Leb^4(\Xset)}^4/n \leq \bigl(\mathbb{E}[\|f\|_n^2]\bigr)^2/(b^2p_{\max})$ when $C_6 \geq C_7 \sqrt{p_{\max}/p_{\min}}$.

Now, suppose $f \in \Leb^4(\Xset)$ satisfies \eqref{pf:empirical_norm_sobolev_1} only when $t = 1$. For each $k = 1,2,\ldots$ let $f_k := n^{1/(2k)}f$, so that each $f_k$ satisfies \eqref{pf:empirical_norm_sobolev_1} with respect to $t = 1 - 1/k$. Clearly $\|f_k - f\|_{\Leb^4(\Xset)} \to 0$ as $k \to \infty$, and therefore
\begin{equation*}
\frac{1}{n}\|f\|_{\Leb^4(\Xset)}^4 = \frac{1}{n}\lim_{k \to \infty} \|f_k\|_{\Leb^4(\Xset)}^4 \leq \frac{1}{b^2p_{\max}} \lim_{k \to \infty} \bigl(\mathbb{E}[\|f_k\|_n^2]\bigr)^2 =  \frac{1}{b^2p_{\max}}\bigl(\mathbb{E}[\|f\|_n^2]\bigr)^2.
\end{equation*}
This establishes the claim when $d = 2$, and completes the proof of Lemma~\ref{lem:empirical_norm_sobolev}.

\section{Graph functionals under the manifold hypothesis}
\label{sec:manifold}

In this section, we restate a few results of \citet{trillos2019,calder2019}, which are analogous to Lemmas~\ref{lem:graph_sobolev_seminorm} and~\ref{lem:neighborhood_eigenvalue} but cover the case where $\Xset$ is an $m$-dimensional submanifold without boundary. As such, the results in this section will hold under the assumption~\ref{asmp:domain_manifold}. We refer to~\citet{trillos2019,calder2019} for the proofs of these results.

Proposition~\ref{prop:garciatrillos19_1} follows from Lemma~5 of~\citet{trillos2019} and Markov's inequality.
\begin{proposition}
	\label{prop:garciatrillos19_1}
	For any $f \in H^1(\Xset)$, with probability at least $1 - \delta$,
	\begin{equation*}
	f^{\top} L f \leq \frac{C}{\delta} n^2 r^{m + 2} |f|_{H^1(\Xset)}^2.
	\end{equation*}
\end{proposition}

In Proposition~\ref{prop:calder19_1}, it is assumed that $r$, $\wt{\delta}$ and $\theta$ satisfy the following smallness conditions.
\begin{enumerate}[label=(S\arabic*)]
	\item 
	\setcounter{enumi}{1}
	\begin{equation*}
	n^{-1/m} < \wt{\delta} \leq \frac{1}{4}r~~\textrm{and}~~C(\theta + \wt{\delta}) \leq \frac{1}{2}p_{\min}~~\textrm{and}~~C_4\bigl(\log(n)/n\bigr)^{1/m}\leq r \leq\min\{c_4,1\}.
	\end{equation*}
\end{enumerate}

\begin{proposition}[\textbf{c.f Theorem 2.4 of~\citet{calder2019}}]
	\label{prop:calder19_1}
	With probability at least $1 - Cn\exp(-cn\theta^2\wt{\delta}^m)$, the following statement holds. For any $k \in \mathbb{N}$ such that
	\begin{equation*}
	\sqrt{\lambda_k(\Delta_P)}r + C(\theta + \wt{\delta}) \leq \frac{1}{2},
	\end{equation*} 
	it holds that
	\begin{equation*}
	nr^{m+2} \lambda_k(\Delta_P) \biggl(1 - C\Bigl(r(\sqrt{\lambda_k(\Delta_P)} + 1) + \frac{\wt{\delta}}{r} + \theta\Bigr)\biggr)\leq \lambda_k(G_{n,r}) \leq nr^{m+2} \lambda_k(\Delta_P) \biggl(1 + C\Bigl(r(\sqrt{\lambda_k(\Delta_P)} + 1) + \frac{\wt{\delta}}{r} + \theta\Bigr)\biggr).
	\end{equation*}
\end{proposition}

Proposition~\ref{prop:calder19_2} follows from Lemma 3.1 of \citet{calder2019}, along with a union bound.
\begin{proposition}
	\label{prop:calder19_2}
	With probability at least $1 - 2Cn\exp(-cp_{\max}nr^m)$, it holds that
	\begin{equation*}
	D_{\max}(G_{n,r}) \leq Cnr^{m}.
	\end{equation*}
\end{proposition}

Finally, we note that a Weyl's Law holds for Riemmanian manifolds without boundary, i.e.
\begin{equation*}
\lambda_k(\Delta_P) \asymp k^{2/m}.
\end{equation*}
Put $B_{n,r}(k) := \min\{nr^{m + 2}k^{2/m}, nr^m\}$. Following parallel steps to the proof of Lemma~\ref{lem:neighborhood_eigenvalue}, one can derive from Propositions~\ref{prop:calder19_1} and~\ref{prop:calder19_2}, and Weyl's Law, that with probability at least $1 - Cn\exp(-cnr^m)$, 
\begin{equation}
\label{eqn:neighborhood_eigenvalue_manifold}
cB_{n,r}(k) \leq \lambda_k \leq CB_{n,r}(k),~~\textrm{for all $2 \leq k \leq n$}.
\end{equation}

\section{Proofs of main results}
\label{sec:main_results}

We are now in a position to prove Theorems~\ref{thm:laplacian_smoothing_estimation1}-\ref{thm:laplacian_smoothing_testing_manifold}, as well as a few other claims from our main text. In Section~\ref{subsec:laplacian_smoothing_estimation1_pf} we prove all of our results regarding estimation and in Section~\ref{subsec:laplacian_smoothing_testing_pf} we prove all of our results regarding testing; in Section~\ref{subsec:convenient_estimate}, Lemmas~\ref{lem:variance_term_estimation} and~\ref{lem:variance_term_testing}, we provide some useful estimates on a particular pair of sums that appear repeatedly in our proofs. Throughout, it will be convenient for us to deal with the normalization $\wt{\rho} := \rho nr^{d + 2}$. We note that in each of our Theorems, the prescribed choice of $\rho$ will always result in $\wt{\rho} \leq 1$. 

\subsection{Proof of estimation results}
\label{subsec:laplacian_smoothing_estimation1_pf}

\paragraph{Proof of Theorem~\ref{thm:laplacian_smoothing_estimation1}.}
We have shown that the inequalities \eqref{eqn:graph_sobolev_seminorm} and \eqref{eqn:neighborhood_eigenvalue} are satisfied with probability at least $1 - \delta - C_1n\exp(-c_1nr^d)$, and throughout this proof we take as granted that both of these inequalities hold.

Now, set $\wt{\rho} = M^{-4/(2 + d)}n^{-2/(2 + d)}$ as prescribed in Theorem~\ref{thm:laplacian_smoothing_estimation1}, and note that $\wt{\rho}^{-d/2} \leq n$ is implied by the assumption $M \leq n^{1/d}$. Therefore  from \eqref{eqn:neighborhood_eigenvalue} and Lemma~\ref{lem:variance_term_estimation}, it follows that
\begin{equation*}
\sum_{k = 1}^{n}\biggl(\frac{1}{\rho \lambda_k + 1}\biggr)^{2} \geq 1 + \frac{1}{C_3^2}\sum_{k = 2}^{n}\biggl(\frac{1}{\wt{\rho} k^{2/d} + 1}\biggr)^{2} \geq \frac{1}{8C_3^2} \wt{\rho}^{-d/2}.
\end{equation*}
As a result, by Lemma~\ref{lem:ls_fixed_graph_estimation} along with \eqref{eqn:graph_sobolev_seminorm} and \eqref{eqn:neighborhood_eigenvalue}, with probability at least $1 - \delta - C_1n\exp(-c_1nr^d) - \exp(-\wt{\rho}^{-d/2}/8C_3^2)$ it holds that,
\begin{align}
\|\wh{f} - f_0\|_n^2 & \leq \frac{C_2}{\delta} \wt{\rho} M^2 + \frac{10}{n} + \frac{10}{n}\sum_{k = 2}^{n} \Biggl(\frac{1}{c_3 \wt{\rho}\min\{k^{2/d},r^{-2}\} + 1}\Biggr)^2 \nonumber \\
& \leq \frac{C_2}{\delta} \wt{\rho} M^2 + \frac{10}{n} + \frac{10}{nc_3^2}\sum_{k = 2}^{n} \biggl(\frac{1}{\wt{\rho}k^{2/d} + 1}\biggr)^2 + \frac{10r^4}{c_3^2 \wt{\rho}^2}. \label{pf:laplacian_smoothing_estimation1_1}
\end{align}
The first term on the right hand side of \eqref{pf:laplacian_smoothing_estimation1_1} is a bias term, while the second, third, and fourth terms each contribute to the variance. Of these, under our assumptions the third term dominates, as we show momentarily. First, we use Lemma~\ref{lem:variance_term_estimation} to get an upper bound on this variance term,
\begin{equation*}
\sum_{k = 2}^{n}\biggl(\frac{1}{\wt{\rho}k^{2/d} + 1}\biggr)^2 \leq 4\wt{\rho}^{-d/2}.
\end{equation*}
Then plugging this upper bound back into \eqref{pf:laplacian_smoothing_estimation1_1}, we have that
\begin{align*}
\|\wh{f} - f_0\|_n^2 & \leq\frac{C_2}{\delta} \wt{\rho} M^2 + \frac{10}{n} + \frac{40\wt{\rho}^{-d/2}}{c_3^{2}n}  + \frac{10r^4}{c_3^2 \wt{\rho}^2} \\
& = \biggl(\frac{C_2}{\delta} + \frac{40}{c_3^{2}}\biggr)M^{2d/(2+d)} n^{-2/(2 + d)} + \frac{10}{n} + \frac{10}{c_3^2}r^4M^{8/(2 + d)}n^{4/(2 + d)} \\
& \leq \biggl(\frac{C_2}{\delta} + \frac{50}{c_3^{2}}\biggr)M^{2d/(2+d)} n^{-2/(2 + d)},
\end{align*}
with the last inequality following from~\ref{asmp:ls_kernel_radius_estimation} and the assumption $M \geq n^{-1/2}$. This completes the proof of Theorem~\ref{thm:laplacian_smoothing_estimation1}.

\paragraph{Proof of Theorem~\ref{thm:laplacian_smoothing_estimation2}.}
We first establish that $\wh{f}$ achieves nearly-optimal rates when $d = 4$, and then establish the claimed sub-optimal rates when $d > 4$.

\textit{Nearly-optimal rates when $d = 4$.}

Continuing on from \eqref{pf:laplacian_smoothing_estimation1_1}, from Lemma~\ref{lem:variance_term_estimation} we have that
\begin{equation*}
\|\wh{f} - f_0\|_n^2 \leq \frac{C_2}{\delta} \wt{\rho} M^2 + \frac{10}{n} +  \frac{10}{nc_3^2\wt{\rho}^2} +  \frac{10 \log n}{nc_3^2\wt{\rho}^2} + \frac{10r^4}{c_3^2\wt{\rho}^2}.
\end{equation*}
Setting $r = (C_0\log(n)/n)^{1/4}$, we obtain
\begin{equation*}
\|\wh{f} - f_0\|_n^2 \leq \frac{C_2}{\delta} \wt{\rho} M^2 + \frac{10}{n} +  \frac{10}{nc_3^2\wt{\rho}^2} +  \frac{10 \log n}{nc_3^2\wt{\rho}^2} + \frac{10C_0\log n}{nc_3^2\wt{\rho}^2},
\end{equation*}
and choosing $\wt{\rho} = M^{-2/3}(\log n/n)^{1/3}$ yields
\begin{equation*}
\|\wh{f} - f_0\|_n^2 \leq \biggl(\frac{C_2}{\delta} + \frac{20}{c_3^2} + \frac{10C_0}{c_3^2}\biggr) M^{4/3} \biggl(\frac{\log n}{n}\biggr)^{1/3} + \frac{10}{n}.
\end{equation*}

\textit{Suboptimal rates when $d > 4$.}

Once again continuing on from \eqref{pf:laplacian_smoothing_estimation1_1}, from Lemma~\ref{lem:variance_term_estimation} we have that
\begin{equation*}
\|\wh{f} - f_0\|_n^2 \leq \frac{C_2}{\delta} \wt{\rho} M^2 + \frac{10}{n} + \frac{10}{nc_3^{2}\wt{\rho}^{d/2}} +  \frac{10}{n^{4/d}\wt{\rho}^2c_3^2} + \frac{10r^4}{\wt{\rho}^2c_3^2}.
\end{equation*}
Setting $r = (C_0\log n/n)^{1/d}$, we obtain
\begin{equation*}
\|\wh{f} - f_0\|_n^2 \leq \frac{C_2}{\delta} \wt{\rho} M^2 + \frac{10}{n} + \frac{10}{n\wt{\rho}^{d/2}c_3^{2}} +  \frac{10}{n^{4/d}\wt{\rho}^2c_3^2} + \frac{10C_0^{4/d}(\log n)^{4/d}}{n^{4/d}\wt{\rho}^2c_3^2},
\end{equation*}
and choosing $\wt{\rho} = M^{-2/3}n^{-4/(3d)}$ yields
\begin{equation*}
\|\wh{f} - f_0\|_n^2 \leq \biggl(\frac{C_2}{\delta} + \frac{10}{c_3^2} + \frac{10C_0^{4/d}}{c_3^2}\biggr)M^{4/3}\biggl(\frac{\log n}{n^{1/3}}\biggr)^{4/d} + \frac{10}{c_3^{d/2}}M^{d/3}n^{-1/3} + \frac{10}{n}.
\end{equation*}

\paragraph{Bounds on $\Leb^2(\Xset)$ error under Lipschitz assumption.} 
Let $V_1,\ldots,V_n$ denote the Voronoi tesselation of $\Xset$ with respect to $X_1,\ldots,X_n$. Extend $\wh{f}$ over $\Xset$ by taking it piecewise constant over the Voronoi cells, i.e.
\begin{equation*}
\wh{f}(x) := \sum_{i = 1}^{n} \wh{f}_i \cdot \1\{x \in V_i\}.
\end{equation*}
Note that we are abusing notation slightly by also using $\wh{f}$ to refer to this extension. 

In Proposition~\ref{prop:out_of_sample_error}, we establish that the out-of-sample error $\|\wh{f} - f_0\|_{\Leb^2(\Xset)}$ will not be too much larger than the in-sample error $\|\wh{f} - f_0\|_n$.
\begin{proposition}
	\label{prop:out_of_sample_error}
	Suppose $f_0$ satisfies $|f_0(x') - f_0(x)| \leq M \|x' - x\|$ for all $x',x \in \mc{X}$. Then for all $n$ sufficiently large, with probability at least $1 - \delta$ it holds that
	\begin{equation*}
	\|\wh{f} - f_0\|_{\Leb^2(\Xset)}^2 \leq C \log(1/\delta) \biggl(\log(n)\cdot \|\wh{f} - f_0\|_n^2 + M^2\Bigl(\frac{\log n}{n}\Bigr)^{2/d}\biggr).
	\end{equation*}
\end{proposition}
Note that $n^{-2/d} \ll n^{-2/(2 +d)}$. Therefore Proposition~\ref{prop:out_of_sample_error} together with Theorem~\ref{thm:laplacian_smoothing_estimation1} implies that with high probability, $\wh{f}$ achieves the nearly-optimal (up to a factor of $\log n$) estimation rates out-of-sample error---that is, $\|\wh{f} - f_0\|_{\Leb^2(\Xset)}^2 \leq C \log(n) M^{2d/(2 + d)}n^{-2/(2+d)}$---as long as $M \leq Cn^{1/d}$. This justifies one of our remarks after Theorem~\ref{thm:laplacian_smoothing_estimation1}. 

\paragraph{Proof of Proposition~\ref{prop:out_of_sample_error}.}
Suppose $x \in V_i$, so that we can upper bound the pointwise squared error $|\wh{f}(x) - f(x)|^2$ using the triangle inequality:
\begin{equation*}
\bigl(\wh{f}(x) - f_0(x)\bigr)^2 = \bigl(\wh{f}(X_i) - f_0(x)\bigr)^2 \leq 2\bigl(\wh{f}(X_i) - f_0(X_i)\bigr)^2 + 2\bigl(f_0(X_i) - f_0(x)\bigr)^2.
\end{equation*}
Integrating both sides of the inequality, we have
\begin{align*}
\int_{\Xset} \bigl(\wh{f}(x) - f_0(x)\bigr)^2 \,dx & \leq 2  \sum_{i = 1}^{n} \int_{V_i} \Bigl(\wh{f}(X_i) - f_0(X_i)\Bigr)^2 \,dx + 2 \sum_{i = 1}^{n} \int_{V_i} \Bigl(f_0(X_i) - f_0(x)\Bigr)^2 \dx \\
& = 2 \sum_{i = 1}^{n} \vol(V_i) \Bigl(\wh{f}(X_i) - f_0(X_i)\Bigr)^2 + 2 \sum_{i = 1}^{n} \int_{V_i} \Bigl(f_0(X_i) - f_0(x)\Bigr)^2 \dx,
\end{align*}
and so by invoking the Lipschitz property of $f_0$, we obtain
\begin{equation}
\label{pf:out_of_sample_error_0}
\|\wh{f} - f\|_{\Leb^2(\Xset)}^2 \leq 2 \sum_{i = 1}^{n} \vol(V_i) \Bigl(\wh{f}(X_i) - f_0(X_i)\Bigr)^2 + 2 M^2 \sum_{i = 1}^{n} \Bigl(\mathrm{diam}(V_i)\Bigr)^2.
\end{equation}
Here we have written $\mathrm{diam}(V)$ for the diameter of a set $V$. 

Now we will use some results of \citet{chaudhuri2010} regarding uniform concentration of empirical counts, to upper bound $\diam(V_i)$ Set
\begin{equation*}
\varepsilon_n := \biggl(\frac{2C_{o}\log(1/\delta)d\log n}{\nu_dp_{\min}a_3n}\biggr)^{1/d},
\end{equation*}
where $C_{o}$ is a constant given in Lemma~16 of \citet{chaudhuri2010}. Note that for $n$ sufficiently large, $\varepsilon_n \leq c_0$, and therefore by \eqref{eqn:integral_boundary} we have that for every $x \in \Xset$, $P(B(x,\varepsilon_n)) \geq 2C_{o}\log(1/\delta)d\frac{\log n}{n}$. Consequently, by Lemma~16 of \citet{chaudhuri2010} it holds that with probability at least $1 - \delta$,
\begin{equation}
\label{pf:out_of_sample_error_1}
\textrm{for all $x \in \mc{X}$},~~ B(x,\varepsilon_n) \cap \{X_1,\ldots,X_n\} \neq \emptyset.
\end{equation}
But if \eqref{pf:out_of_sample_error_1} is true, it must also be true that for each $i = 1,\ldots,n$ and for every $x \in V_i$, the distance $\|x - X_i\| \leq \varepsilon_n$. Thus by the triangle inequality, $\max_{i = 1,\ldots,n} \mathrm{diam}(V_i) \leq 2\varepsilon_n$. Plugging back in to \eqref{pf:out_of_sample_error_0}, and using the upper bound volume $\vol(V_i) \leq \nu_d \bigl(\mathrm{diam}(V_i)\bigr)^d$, we obtain the desired upper bound on $\|\wh{f} - f\|_{\Leb^2(\Xset)}^2$.

\paragraph{Proof of Theorem~\ref{thm:laplacian_smoothing_estimation_manifold}.}
The proof of Theorem~\ref{thm:laplacian_smoothing_estimation_manifold} follows exactly the same steps as the proof of Theorem~\ref{thm:laplacian_smoothing_estimation1}, replacing the references to Lemma~\ref{lem:graph_sobolev_seminorm} and~\ref{lem:neighborhood_eigenvalue} by references to Proposition~\ref{prop:garciatrillos19_1} and \eqref{eqn:neighborhood_eigenvalue_manifold}, and the ambient dimension $d$ by the intrinsic dimension $m$. 

\subsection{Proofs of testing results}
\label{subsec:laplacian_smoothing_testing_pf}

\paragraph{Proof of Theorem~\ref{thm:laplacian_smoothing_testing}.}
Let $\delta = 1/b$. Recall that we have shown that the inequalities \eqref{eqn:graph_sobolev_seminorm} and \eqref{eqn:neighborhood_eigenvalue} are satisfied with probability at least $1 - 1/b - C_1n\exp(-c_1nr^d)$, and throughout this proof we take as granted that both of these inequalities hold. 

Now, we would like to invoke Lemma~\ref{lem:ls_fixed_graph_testing}, and in order to do so, we must show that the inequality \eqref{eqn:ls_fixed_graph_testing_critical_radius} is satisfied with respect to $G = G_{n,r}$. First, we upper bound the right hand side of this inequality. Setting $\wt{\rho} = M^{-8/(4 + d)}n^{-4/(4 + d)}$ as prescribed by Theorem~\ref{thm:laplacian_smoothing_testing}, it follows from \eqref{eqn:graph_sobolev_seminorm} and \eqref{eqn:neighborhood_eigenvalue} that
\begin{align*}
\frac{2 \rho}{n} \bigl(f_0^{\top} \Lap f_0\bigr) + \frac{2\sqrt{2/\alpha} + 2b}{n} \biggl(\sum_{k = 1}^{n} \frac{1}{(\rho\lambda_k + 1)^4} \biggr)^{1/2} & \leq C_2b\wt{\rho}M^2 + \frac{2\sqrt{2/\alpha} + 2b}{n}\Biggl[1 + \frac{1}{c_3^2}\biggl(\sum_{k = 2}^{n} \frac{1}{(\wt{\rho}k^{2/d} + 1)^4} \biggr)^{1/2} + \frac{r^4n^{1/2}}{c_3^2\wt{\rho}^2}\Biggr] \\
& \leq C_2b\wt{\rho}M^2 + \frac{2\sqrt{2/\alpha} + 2b}{n}\biggl(1 + \frac{\sqrt{2}}{c_3^{2}}\wt{\rho}^{-d/4} + \frac{r^4n^{1/2}}{c_3^2\wt{\rho}^2}\biggr) \\ 
& \leq \Bigl(C_2 + 2 + \frac{2\sqrt{2}}{c_3^{2}} + \frac{2}{c_3^2}\Bigr)\cdot \Bigl(\sqrt{\frac{2}{\alpha}} + b\Bigr) \cdot M^{2d/(4 + d)} n^{-4/(4 + d)}.
\end{align*}
The second inequality in the above is justified by Lemma~\ref{lem:variance_term_testing}, keeping in mind that $M \leq M_{\max}(d)$ implies that $\wt{\rho}^{-d/2} \leq n$. The third inequality follows from the upper bound on $r$ assumed in~\ref{asmp:ls_kernel_radius_testing} as well as the fact that $M \geq n^{-1/2}$.

Next we lower bound the left hand side of the inequality \eqref{eqn:ls_fixed_graph_testing_critical_radius}---i.e. we lower bound the empirical norm $\|f_0\|_n^2$---using Lemma~\ref{lem:empirical_norm_sobolev}. Recall that by assumption, $M \leq M_{\max}(d)$. Therefore, taking $C \geq C_6$ in \eqref{eqn:laplacian_smoothing_testing} implies that the lower bound on $\|f\|_{\Leb^2(\mc{X})}$ in \eqref{eqn:empirical_norm_sobolev_1} is satisfied. As a result, it follows from \eqref{eqn:empirical_norm_sobolev} that
\begin{equation*}
\|f\|_n^2 \geq \frac{\mathbb{E}[\|f\|_n^2]}{b} \geq \frac{p_{\min}}{b} \|f\|_{\Leb^2(\Xset)}^2 \geq C\Bigl(\sqrt{\frac{1}{\alpha}} + b\Bigr)M^{2d/(4 + d)}n^{-4/(4 + d)},
\end{equation*}
with probability at least $1 - 5/b$. Taking $C \geq C_2 + 2 + (2\sqrt{2})/c_3^{2} + 2/c_3^2$ in \eqref{eqn:laplacian_smoothing_testing} thus implies \eqref{eqn:ls_fixed_graph_testing_critical_radius}, and we may therefore use Lemma~\ref{lem:ls_fixed_graph_testing} to upper bound the type II error the Laplacian smoothing test $\wh{\varphi}$. Observe that by \eqref{eqn:neighborhood_eigenvalue} and the lower bound in Lemma~\ref{lem:variance_term_testing}, 
\begin{equation*}
\sum_{k = 1}^{n} \biggl(\frac{1}{\rho \lambda_k + 1}\biggr)^4 \geq 1 + \frac{1}{C_3^4}\sum_{k = 2}^{n}\biggl(\frac{1}{\wt{\rho}k^{2/d} + 1}\biggr)^4 \geq \frac{1}{32C_3^4}\wt{\rho}^{-d/2}.
\end{equation*}
We conclude that
\begin{align*}
\Pbb_{f_0}\bigl(\wh{T} \leq \wh{t}_{\alpha}\bigr) & \leq \frac{6}{b} + \frac{1}{b^2} + \frac{16}{b} \Biggl(\sum_{k = 1}^{n} \frac{1}{(\rho\lambda_k + 1)^4} \Biggr)^{-1/2} + C_1n\exp(-c_1nr^d)\\
& \leq \frac{7}{b} + \frac{64\sqrt{2}}{b} C_3^{2}\wt{\rho}^{d/4} + C_1n\exp(-c_1nr^d),
\end{align*}
establishing the claim of Theorem~\ref{thm:laplacian_smoothing_testing}.

\paragraph{Proof of Theorem~\ref{thm:laplacian_smoothing_testing_manifold}.}
The proof of Theorem~\ref{thm:laplacian_smoothing_testing_manifold} follows exactly the same steps as the proof of Theorem~\ref{thm:laplacian_smoothing_testing}, replacing the references to Lemma~\ref{lem:graph_sobolev_seminorm} and~\ref{lem:neighborhood_eigenvalue} by references to Propositions~\ref{prop:garciatrillos19_1} and \eqref{eqn:neighborhood_eigenvalue_manifold}, and the ambient dimension $d$ by the intrinsic dimension $m$.

\paragraph{Proof of \eqref{eqn:laplacian_smoothing_testing_low_smoothness}.}
When $\rho = 0$, the Laplacian smoother $\wh{f} = \mathbf{Y}$, the test statistic $\wh{T} = \frac{1}{n}\|\mathbf{Y}\|_2^2$, and the threshold $\wh{t}_{\alpha} = 1 + n^{-1/2}\sqrt{2/\alpha}$. The expectation of $\wh{T}$ is 
\begin{equation*}
\Ebb\bigl[\wh{T}\bigr] = \mathbb{E}\bigl[f_0^2(X)\bigr] + 1 \geq p_{\min} \norm{f_0}_{\Leb^2(\Xset)}^2 + 1.
\end{equation*}
When $f_0 \in \Leb^4(\Xset,M)$, the variance can be upper bounded
\begin{equation*}
\Var\bigl[\wh{T}\bigr] \leq \frac{1}{n}\Bigl(3 + p_{\max} M^4 + p_{\max}\norm{f_0}_{\Leb^2(\Xset)}^2\Bigr).
\end{equation*}
Now, let us assume that
\begin{equation*}
\norm{f_0}_{\Leb^2(X)}^2 \geq \frac{2\sqrt{2/\alpha} + 2b}{p_{\min}} n^{-1/2},
\end{equation*}
so that $E[\wh{T}] - \wh{t}_{\alpha} \geq E[f_0^2(X)]/2$. Hence, by Chebyshev's inequality
\begin{align*}
\mathbb{P}_{f_0}\Bigl(\wh{T} \leq \wh{t}_{\alpha} \Bigr) & \leq 4 \frac{\Var_{f_0}\bigl[\wh{T}\bigr]}{\mathbb{E}[f_0^2(X)]^2} \\
& \leq \frac{4}{n} \cdot \frac{ 3 + p_{\max}\bigl(M^4 + \|f_0\|_{\Leb^2(\Xset)}^2 \bigr)}{p_{\min}^2 \|f_0\|_{\Leb^2(\Xset)}^4} \\
& \leq \frac{1}{b^2}\Bigl(3 + \frac{4bp_{\max}}{p_{\min}n^{1/2}} + p_{\max}M^4\Bigr).
\end{align*}

\subsection{Two convenient estimates}
\label{subsec:convenient_estimate}

The following Lemmas provides convenient upper and lower bounds on our estimation variance term (Lemma~\ref{lem:variance_term_estimation}) and testing variance term (Lemma~\ref{lem:variance_term_testing}).
\begin{lemma}
	\label{lem:variance_term_estimation}
	For any $t > 0$ such that $1 \leq t^{-d/2} \leq n$,
	\begin{equation*}
	\frac{1}{8}t^{-d/2} - 1 \leq \sum_{k = 2}^{n} \biggl(\frac{1}{tk^{2/d} + 1}\biggr)^2 \leq t^{-d/2} +
	\begin{cases*}
	3t^{-d/2},& ~~\textrm{if $d < 4$} \\
	\frac{1}{t^2}\log n,& ~~\textrm{if $d = 4$} \\
	\frac{1}{t^2}n^{1 - 4/d},&~~\textrm{if $d > 4$.}
	\end{cases*}
	\end{equation*}
\end{lemma}

\begin{lemma}
	\label{lem:variance_term_testing}
	Suppose $d \leq 4$. Then for any $t > 0$ such that $1 \leq t^{-d/2} \leq n$,
	\begin{equation*}
	\frac{1}{32}t^{-d/2} - 1 \leq \sum_{k = 2}^{n} \biggl(\frac{1}{tk^{2/d} + 1}\biggr)^4 \leq 2t^{-d/2}.
	\end{equation*}
\end{lemma}

\paragraph{Proof of Lemma~\ref{lem:variance_term_estimation}.}
We begin by proving the upper bounds. Treating the sum over $k$ as a Riemann sum of a non-increasing function, we have that
\begin{align*}
\sum_{k = 2}^{n} \biggl(\frac{1}{tk^{2/d} + 1}\biggr)^2 & \leq \int_{1}^{n} \biggl(\frac{1}{tx^{2/d} + 1}\biggr)^2 \,dx \leq t^{-d/2} + \int_{t^{-d/2}}^{n} \biggl(\frac{1}{tx^{2/d} + 1}\biggr)^2 \,dx \leq t^{-d/2} + \frac{1}{t^2} \int_{t^{-d/2}}^{n} x^{-4/d} \,dx.
\end{align*}
The various upper bounds (for $d < 4$, $d = 4$, and $d > 4$) then follow upon computing the integral.

For the lower bound, we simply recognize that for each $k = 2,\ldots,n$ such that $k \leq \floor{t^{-d/2}}$, it holds that $1/(tk^{2/d} + 1)^2 \geq 1/4$, and there are at least $\min\Bigl\{\floor{t^{-d/2}} - 1,n - 1\Bigr\} > \frac{1}{2}t^{-d/2} - 1$ such values of $k$.

\paragraph{Proof of Lemma~\ref{lem:variance_term_testing}.}
The upper bound follows similarly to that of Lemma~\ref{lem:variance_term_estimation}:
\begin{align*}
\sum_{k = 1}^{n} \biggl(\frac{1}{tk^{2/d} + 1}\biggr)^4 \leq t^{-d/2} + \frac{1}{t^4}\sum_{k = t^{-d/2} + 1}^{n} \frac{1}{k^{8/d}} \leq t^{-d/2} + \frac{1}{t^4} \int_{t^{-d/2}}^{n} x^{-8/d} \,dx \leq 2t^{-d/2}.
\end{align*}
The lower bound follows from the same logic as we used to derive the lower bound in Lemma~\ref{lem:variance_term_estimation}.

\section{Concentration inequalities}

\label{sec:concentration}
\begin{lemma}
	\label{lem:chi_square_bound}
	Let $\xi_1,\ldots,\xi_N$ be independent $N(0,1)$ random variables, and let $U := \sum_{k = 1}^{N} a_k(\xi_k^2 - 1)$.  Then for any $t > 0$,
	\begin{equation*}
	\Pbb\Bigl[U \geq 2 \norm{a}_2 \sqrt{t} + 2 \norm{a}_{\infty}t\Bigr] \leq \exp(-t).
	\end{equation*}
	In particular if $a_k = 1$ for each $k = 1,\ldots,N$, then
	\begin{equation*}
	\Pbb\Bigl[U\geq 2\sqrt{N t} + 2t\Bigr] \leq \exp(-t).
	\end{equation*}
\end{lemma}

The proof of Lemma~\ref{lem:empirical_norm_sobolev} relies on (a variant of) the Paley-Zygmund Inequality.
\begin{lemma}
	\label{lem:paley_zygmund}
	Let $f$ satisfy the following moment inequality for some $b \geq 1$:
	\begin{equation}
	\label{eqn:paley_zygmund_1}
	\Ebb\bigl[\norm{f}_n^4\bigr] \leq \left(1 + \frac{1}{b^2}\right)\cdot\Bigl(\Ebb\bigl[\norm{f}_n^2\bigr]\Bigr)^2.
	\end{equation}
	Then,
	\begin{equation}
	\label{eqn:paley_zygmund_2}
	\mathbb{P}\left[\norm{f}_n^2 \geq \frac{1}{b} \Ebb\bigl[\norm{f}_n^2\bigr]\right] \geq 1 - \frac{5}{b}.
	\end{equation}
\end{lemma}
\begin{proof}
	Let $Z$ be a non-negative random variable such that $\mathbb{E}(Z^q) < \infty$. The Paley-Zygmund inequality says that for all $0 \leq \lambda \leq 1$,
	\begin{equation}
	\label{eqn:paley_zygmund_pf1}
	\mathbb{P}(Z > \lambda \mathbb{E}(Z^p)) \geq \left[(1 - \lambda^p) \frac{\mathbb{E}(Z^p)}{(\mathbb{E}(Z^q))^{p/q}}\right]^{\frac{q}{q - p}}.
	\end{equation}
	Applying \eqref{eqn:paley_zygmund_pf1} with $Z = \norm{f}_n^2$, $p = 1$, $q = 2$ and $\lambda = \frac{1}{b}$, by assumption \eqref{eqn:paley_zygmund_1} we have
	\begin{equation*}
	\mathbb{P}\Bigl(\norm{f}_n^2 > \frac{1}{b} \mathbb{E}[\norm{f}_n^2]\Bigr) \geq \Bigl(1 - \frac{1}{b}\Bigr)^2 \cdot  \frac{\bigl(\mathbb{E}[\norm{f}_n^2]\bigr)^2}{\mathbb{E}[\norm{f}_n^4]} \geq \frac{\Bigl(1 - \frac{2}{b}\Bigr)}{\Bigl(1 + \frac{1}{b^2}\Bigr)} \geq 1 - \frac{5}{b}.
	\end{equation*}
\end{proof}

Let $Z_1,\ldots,Z_n$ be independently distributed and bounded random variables, such that $\Ebb[Z_i] = \mu_i$. Let $S_n = Z_1 + \ldots + Z_n$ and $\mu = \mu_1 + \ldots + \mu_n$. The multiplicative form of Hoeffding's inequality gives sharp bounds when $\mu \ll 1$. 
\begin{lemma}[Hoeffding's Inequality, multiplicative form]
	\label{lem:hoeffding}
	Suppose $Z_i$ are independent random variables, which satisfy $Z_i \in [0,B]$ for $i = 1,\ldots,n$. For any $0 < \delta < 1$, it holds that
	\begin{equation*}
	\Pbb\biggl(\Bigl|S_n - \mu\Bigr| \geq \delta \mu\biggr) \leq 2\exp\biggl(-\frac{\delta^2\mu}{3B^2}\biggr).
	\end{equation*}
\end{lemma}
We use Lemma~\ref{lem:hoeffding}, along with properties of the kernel $K$ and density $p$, to upper bound the maximum degree in our neighborhood graph, which we denote by $D_{\max}(G_{n,r}) := \max_{i = 1,\ldots,n} D_{ii}$.
\begin{lemma}
	\label{lem:max_degree}
	Under the conditions of Lemma~\ref{lem:neighborhood_eigenvalue},
	\begin{equation*}
	D_{\max}(G_{n,r}) \leq 2p_{\max}nr^d,
	\end{equation*}
	with probability at least $1 - 2n\exp\Bigl(-nr^da_3p_{\min}/(3[K(0)]^2)\Bigr)$. 
\end{lemma}

\paragraph{Proof of Lemma~\ref{lem:max_degree}.}
Fix $x \in \Xset$, and set
\begin{equation*}
D_{n,r}(x) :=  \sum_{i = 1}^{n} K\biggl(\frac{\|X_i - x\|}{r}\biggr);
\end{equation*}
note that $D_{n,r}(X_i)$ is just the degree of $X_i$ in $G_{n,r}$. By Hoeffding's inequality
\begin{equation}
\label{pf:max_degree_1}
\Pbb\biggl(\Bigl|D_{n,r}(x) - \Ebb\bigl[D_{n,r}(x)\bigr]\Bigr| \geq \delta \Ebb\bigl[D_{n,r}(x)\bigr]\biggr) \leq 2\exp\biggl(-\frac{\delta^2\Ebb\bigl[D_{n,r}(x)\bigr]}{3[K(0)]^2}\biggr).
\end{equation}
Now we lower bound $\Ebb[D_{n,r}(x)]$ using the boundedness of the density $p$, and the fact that $\Xset$ has Lipschitz boundary:
\begin{align*}
\Ebb\bigl[D_{n,r}(x)\bigr] & = n \int_{\Xset} K\biggl(\frac{\|x' - x\|}{r}\biggr) p(x) \,dx \\
& \geq n p_{\min} \int_{\Xset} K\biggl(\frac{\|x' - x\|}{r}\biggr) \,dx \\
& \geq n p_{\min} a_3 \int_{\Xset} K\biggl(\frac{\|x' - x\|}{r}\biggr) \,dx \\
& \geq nr^d p_{\min},
\end{align*}
with the second inequality following from \eqref{eqn:integral_boundary}, and the final inequality from the normalization $\int_{\Rd} K(\|z\|) \,dz = 1$. Similar derivations yield the upper bound
\begin{equation*}
\Ebb\bigl[D_{n,r}(x)\bigr] \leq nr^{d} p_{\max},
\end{equation*} 
and plugging these bounds in to \eqref{pf:max_degree_1}, we determine that
\begin{equation*}
\Pbb\biggl(D_{n,r}(x) \geq (1 + \delta) nr^d p_{\max}\biggr) \leq 2\exp\biggl(-\frac{\delta^2nr^da_0p_{\min}}{3[K(0)]^2}\biggr).
\end{equation*}
Applying a union bound, we get that
\begin{equation*}
\Pbb\biggl(\max_{i = 1,\ldots,n}D_{n,r}(X_i) \geq (1 + \delta) nr^d p_{\max}\biggr) \leq 2n\exp\biggl(-\frac{\delta^2nr^da_0p_{\min}}{3[K(0)]^2}\biggr),
\end{equation*}
and taking $\delta = 1$ gives the claimed upper bound.

\end{document}